\documentclass[reqno, oneside]{amsart}
\usepackage{pstricks}

\newif\ifSa
\newif\ifSb
\newif\ifSc
\newif\ifSd
\newif\ifSe
\newif\ifSf
\newif\ifSg
\newif\ifSh
\newif\ifTa
\newif\ifTb
\newif\ifTc
\newif\ifTd
\newif\ifTe
\newif\ifTf
\newif\ifTg
\newif\ifTh

\newcommand\clearlabels{\Safalse \Sbfalse \Scfalse \Sdfalse \Sefalse \Sffalse \Sgfalse \Shfalse%
                        \Tafalse \Tbfalse \Tcfalse \Tdfalse \Tefalse \Tffalse \Tgfalse \Thfalse}

\newcommand\parboxsize{1.2in}
\newcommand\nodesize{.1}
\newcommand\circsize{.4}

\newcommand\nast{\Large \textbf \textasteriskcentered } 
\newcommand\past[1]{  \rput(#1){\nast} } 
\newcommand\pcirc[1]{   \pscircle (#1) {\circsize}  } 
\newcommand\pnode[1]{   \pscircle[fillcolor=black,fillstyle=solid] (#1) {\nodesize}  }  

\newcommand\Esix{%
\parbox{\parboxsize}
{
\psset{unit=.45}
\begin{pspicture}(-.75,-1)(4.5,1.2)  
\psline(0,0)(4,0)  \psline(2,0)(2,1)
\pnode{0,0}  \ifSa  \pcirc{0,0} \fi  \ifTa \past{0,0}\fi
\pnode{1,0}  \ifSb \pcirc{1,0}  \fi  \ifTb \past{1,0}\fi
\pnode{2,0}  \ifSc \pcirc{2,0}  \fi  \ifTc \past{2,0}\fi
\pnode{3,0}  \ifSd \pcirc{3,0}  \fi  \ifTd \past{3,0}\fi
\pnode{4,0}  \ifSe \pcirc{4,0}  \fi  \ifTe \past{4,0}\fi
\pnode{2,1}  \ifSf \pcirc{2,1}  \fi  \ifTf \past{2,1}\fi
\end{pspicture}
}
}%

\newcommand\Esixtable{%
\parbox{\parboxsize}
{
\psset{unit=.45}
\begin{pspicture}(-.5,-1)(4.5,1.7)  
\psline(0,0)(4,0)  \psline(2,0)(2,1)
\pnode{0,0}  \ifSa  \pcirc{0,0} \fi  \ifTa \past{0,0}\fi
\pnode{1,0}  \ifSb \pcirc{1,0}  \fi  \ifTb \past{1,0}\fi
\pnode{2,0}  \ifSc \pcirc{2,0}  \fi  \ifTc \past{2,0}\fi
\pnode{3,0}  \ifSd \pcirc{3,0}  \fi  \ifTd \past{3,0}\fi
\pnode{4,0}  \ifSe \pcirc{4,0}  \fi  \ifTe \past{4,0}\fi
\pnode{2,1}  \ifSf \pcirc{2,1}  \fi  \ifTf \past{2,1}\fi
\end{pspicture}
}
}%

\newcommand\Esixlabelled{%
\parbox{\parboxsize}
{
\psset{unit=.48}
\begin{pspicture}(-.5,-1)(4.5,1)
\psline(0,0)(4,0)  \psline(2,0)(2,1)
\pnode{0,0}
\pnode{1,0}
\pnode{2,0}
\pnode{3,0}
\pnode{4,0}
\pnode{2,1}
	\put(-.2,-.6){$\mu_1$}
	\put(.8,-.6){$\mu_2$}
	\put(1.8,-.6){$\mu_3$}
	\put(2.8,-.6){$\mu_4$}
	\put(3.8,-.6){$\mu_5$}
	\put(2.2,.8){$\mu_6$}
\end{pspicture}
}
}%

\newcommand\Esixextended{%
\parbox{\parboxsize}
{
\psset{unit=.48}
\begin{pspicture}(-.5,-1)(4.5,2.1)
\psline(0,0)(4,0) \psline(2,0)(2,2)
\pnode{0,0}
\pnode{1,0}
\pnode{2,0}
\pnode{3,0}
\pnode{4,0}
\pnode{2,1}
\pnode{2,2}
	\put(-.2,-.6){$\mu_1$}
	\put(.8,-.6){$\mu_2$}
	\put(1.8,-.6){$\mu_3$}
	\put(2.8,-.6){$\mu_4$}
	\put(3.8,-.6){$\mu_5$}
	\put(2.2,.8){$\mu_6$}
    \put(2.2,1.8){$\mu^-$}
\end{pspicture}
}
}%

\newcommand\Eseven{%
\parbox{\parboxsize}
{
\psset{unit=.45}
\begin{pspicture}(-.5,-1)(5.5,1.5)  
\psline(0,0)(5,0) \psline(2,0)(2,1)
\pnode{0,0}  \ifSa  \pcirc{0,0} \fi  \ifTa \past{0,0}\fi
\pnode{1,0}  \ifSb \pcirc{1,0}  \fi  \ifTb \past{1,0}\fi
\pnode{2,0}  \ifSc \pcirc{2,0}  \fi  \ifTc \past{2,0}\fi
\pnode{3,0}  \ifSd \pcirc{3,0}  \fi  \ifTd \past{3,0}\fi
\pnode{4,0}  \ifSe \pcirc{4,0}  \fi  \ifTe \past{4,0}\fi
\pnode{5,0}  \ifSf \pcirc{5,0}  \fi  \ifTf \past{5,0}\fi
\pnode{2,1}  \ifSg \pcirc{2,1}  \fi  \ifTg \past{2,1}\fi
\end{pspicture}
}
}%

\newcommand\An{%
\parbox{\parboxsize}
{
\psset{unit=.45}
\begin{pspicture}(-.5,-1)(4,1)
\psline(0,0)(1,0) \psline(2.5,0)(3.5,0) \put(1.40,-.05){$\dots$}  
\pnode{0,0}  \ifSa  \pcirc{0,0} \fi  \ifTa \past{0,0}\fi
\pnode{1,0}  \ifSb \pcirc{1,0}  \fi  \ifTb \past{1,0}\fi
\pnode{2.5,0}  \ifSc \pcirc{2.5,0}  \fi  \ifTc \past{2,0}\fi
\pnode{3.5,0}  \ifSd \pcirc{3.5,0}  \fi  \ifTd \past{3,0}\fi
\end{pspicture}
}
}%

\newcommand\Anmidin{%
\parbox{\parboxsize}
{
\psset{unit=.45}
\begin{pspicture}(.1,-1)(5,1)
\psline(1.5,0)(3.5,0) \put(.15,-.05){$\dots$} \put(4.0,-.05){$\dots$}
\pnode{0,0}  \ifSa  \pcirc{0,0} \fi  \ifTa \past{0,0}\fi
\pnode{1.5,0}  \ifSb \pcirc{1.5,0}  \fi  \ifTb \past{1.5,0}\fi
\pnode{2.5,0}  \ifSc \pcirc{2.5,0}  \fi  \ifTc \past{2.5,0}\fi
\pnode{3.5,0}  \ifSd \pcirc{3.5,0}  \fi  \ifTd \past{3.5,0}\fi
\pnode{5,0}  \ifSe \pcirc{5,0}  \fi  \ifTe \past{5,0}\fi
\end{pspicture}
}
}%

\newcommand\Anmidout{%
\parbox{\parboxsize}
{
\psset{unit=.45}
\begin{pspicture}(.1,-1)(5,1)
\psline(1.5,0)(3.5,0) \put(0.45,-.05){$\dots$} \put(3.70,-.05){$\dots$}
\pnode{0,0}  \ifSa  \pcirc{0,0} \fi  \ifTa \past{0,0}\fi
\pnode{1.5,0}  \ifSb \pcirc{1.5,0}  \fi  \ifTb \past{1.5,0}\fi
\pnode{2.5,0}  \ifSc \pcirc{2.5,0}  \fi  \ifTc \past{2.5,0}\fi
\pnode{3.5,0}  \ifSd \pcirc{3.5,0}  \fi  \ifTd \past{3.5,0}\fi
\pnode{5,0}  \ifSe \pcirc{5,0}  \fi  \ifTe \past{5,0}\fi
\end{pspicture}
}
}%

\newcommand\Bn{%
\parbox{\parboxsize}
{
\psset{unit=.45}
\begin{pspicture}(-.5,-1)(5,1)
\psline(0,0)(1,0) \psline(2.5,0)(3.5,0) \psline[doubleline=true]{->}(3.5,0)(4.7,0) \put(1.45,-.05){$\dots$} 
\pnode{0,0}  \ifSa  \pcirc{0,0} \fi  \ifTa \past{0,0}\fi
\pnode{1,0}  \ifSb \pcirc{1,0}  \fi  \ifTb \past{1,0}\fi
\pnode{2.5,0}  \ifSc \pcirc{2.5,0}  \fi  \ifTc \past{2.5,0}\fi
\pnode{3.5,0}  \ifSd \pcirc{3.5,0}  \fi  \ifTd \past{3.5,0}\fi
\pnode{4.7,0}  \ifSe \pcirc{4.7,0}  \fi  \ifTe \past{4.7,0}\fi
\end{pspicture}
}
}%

\newcommand\Cn{%
\parbox{\parboxsize}
{
\psset{unit=.45}
\begin{pspicture}(-.5,-1)(5,1)
\psline(0,0)(1,0) \psline(2.5,0)(3.5,0) \psline[doubleline=true]{<-}(3.5,0)(4.7,0)
\put(1.40,-.05){$\dots$}
\pnode{0,0}  \ifSa  \pcirc{0,0} \fi  \ifTa \past{0,0}\fi
\pnode{1,0}  \ifSb \pcirc{1,0}  \fi  \ifTb \past{1,0}\fi
\pnode{2.5,0}  \ifSc \pcirc{2.5,0}  \fi  \ifTc \past{2.5,0}\fi
\pnode{3.5,0}  \ifSd \pcirc{3.5,0}  \fi  \ifTd \past{3.5,0}\fi
\pnode{4.7,0}  \ifSe \pcirc{5,0}  \fi  \ifTe \past{4.7,0}\fi
\end{pspicture}
}
}%

\newcommand\Dfive{%
\parbox{.7in}
{
\psset{unit=.45}
\begin{pspicture}(-.5,-1)(3.2,1)
\psline(0,0)(2,0) \psline(2.0,0)(2.7,.7) \psline(2.0,0)(2.7,-.7)
\pnode{0,0}  \ifSa  \pcirc{0,0} \fi  \ifTa \past{0,0}\fi
\pnode{1,0}  \ifSb \pcirc{1,0}  \fi  \ifTb \past{1,0}\fi
\pnode{2,0}  \ifSc \pcirc{2,0}  \fi  \ifTc \past{2,0}\fi
\pnode{2.7,.7}  \ifSd \pcirc{2.7,.7}  \fi  \ifTd \past{2.7,.7}\fi
\pnode{2.7,-.7}  \ifSe \pcirc{2.7,-.7}  \fi  \ifTe \past{2.7,-.7}\fi
\end{pspicture}
}
}%

\newcommand\Dn{%
\parbox{\parboxsize}
{
\psset{unit=.45}
\begin{pspicture}(-.5,-1)(4.5,1.2)
\psline(0,0)(1,0) \psline(2.5,0)(3.5,0) \psline(3.5,0)(4.2,.7) \psline(3.5,0)(4.2,-.7) \put(1.45,-.05){$\dots$} 
\pnode{0,0}  \ifSa  \pcirc{0,0} \fi  \ifTa \past{0,0}\fi
\pnode{1,0}  \ifSb \pcirc{1,0}  \fi  \ifTb \past{1,0}\fi
\pnode{2.5,0}  \ifSc \pcirc{2.5,0}  \fi  \ifTc \past{2.5,0}\fi
\pnode{4.2,.7}  \ifSd \pcirc{4.2,.7}  \fi  \ifTd \past{4.2,.7}\fi
\pnode{4.2,-.7}  \ifSe \pcirc{4.2,-.7}  \fi  \ifTe \past{4.2,-.7}\fi
\end{pspicture}
}
}%

\newtheorem{theorem}{Theorem}[subsection]  
\newtheorem{lemma}[theorem]{Lemma}
\newtheorem{proposition}[theorem]{Proposition}
\newtheorem{corollary}[theorem]{Corollary}

\theoremstyle{definition}
\newtheorem{definition}[theorem]{Definition}
\newtheorem{construction}[theorem]{Construction}
\newtheorem{example}[theorem]{Example}

\newtheorem{remark}[theorem]{Remark}

\newlength\savedwidth
        \newcommand\whline{\noalign{\global\savedwidth\arrayrulewidth\global\arrayrulewidth 1.2pt}%
        \hline
        \noalign{\global\arrayrulewidth\savedwidth}}

\DeclareMathOperator\ad{ad}

\DeclareMathOperator\Aut{Aut}

\DeclareMathOperator\card{card}

\DeclareMathOperator{\comp}{comp}

\DeclareMathOperator\End{End}

\DeclareMathOperator\GL{GL}
\DeclareMathOperator\height{ht}
\DeclareMathOperator\Hom{Hom}
\DeclareMathOperator\id{id}
\DeclareMathOperator\KA{KA}

\DeclareMathOperator\Mat{M}

\DeclareMathOperator\spl{\mathfrak{sl}}

\DeclareMathOperator\SPA{SPA}
\DeclareMathOperator\spann{span}
\DeclareMathOperator\supp{supp}
\DeclareMathOperator\tr{tr}

\DeclareMathOperator\trcomp{t}

\newcommand\al{\alpha}

\newcommand\De{\Delta}
\newcommand\lm{\lambda}

\newcommand\ph{\varphi}
\newcommand\sg{\sigma}

\newcommand\bbK{\mathbb{K}}

\newcommand\bbQ{\mathbb{Q}}
\newcommand\bbR{\mathbb{R}}

\newcommand\bbZ{\mathbb{Z}}

\newcommand\cC{\mathcal{C}}

\newcommand\cF{\mathcal{F}}
\newcommand\cG{\mathcal{G}}
\newcommand\cH{\mathcal{H}}

\newcommand\cS{\mathcal{S}}

\newcommand\bP{{\breve P}}
\newcommand\brv{\breve{\hphantom{a}}} 
\newcommand\bS{{\breve S}}

\newcommand\ff{{\mathfrak f}}
\newcommand\fk{{\mathfrak k}}

\newcommand\Kan{\mathfrak{K}}
\newcommand\fP{\mathfrak{P}}

\newcommand\fT{\mathfrak{T}}

\newcommand\tPi{\tilde{\Pi}}

\newcommand\rL{L}  

\newcommand \andd{\quad\text{and}\quad}
\newcommand\bform{{   \langle \, ,\, \rangle }}
\newcommand\simgr{\simeq_{\text{gr}}}
\newcommand \suchthat { : }
\newcommand \set[1]{\{#1 \}}
\newcommand\sm{\setminus}
\newcommand\tprod{{   \{\, ,\, ,\, \} }}
\newcommand\type[2]{\rm{#1}_{#2}}  

\newcommand\BCone{\type{BC}{1}}
\newcommand\BCtwo{\type{BC}{2}}

\newcommand\chiST{\chi_{(S,T)}}
\newcommand\ES{E_\Sigma}

\newcommand\formtwo{}
\newcommand\formblank{(\ ,\ )}

\newcommand\griso{isomorphic}

\newcommand\hQo{\Hom(Q_\Sigma,\bbZ)}
\newcommand\hQt{\Hom(Q_\Sigma,\bbZ^2)}
\newcommand\msg{{-\sigma}}
\newcommand\op{{\textrm{op}}}
\newcommand\p[1]{{}^{#1}}
\newcommand\Pd{\Pi_\Delta}
\newcommand\Ps{\Pi_\Sigma}
\newcommand\psh{{\rm{psh}}}
\newcommand\pos{{\rm{p}}}
\newcommand\QD{{Q_\Delta}}  
\newcommand\QS{{Q_\Sigma}}  
\newcommand\rbrv{ {\}\brv}^{\sg}  }
\newcommand\rD{{n_\Delta}}
\newcommand\rS{{n_\Sigma}}
\newcommand\sh{{\rm{sh}}}

\newcommand\rkn{n}

\newcommand\WS{W_\Sigma}
\newcommand\WD{W_\Delta}

\newcommand\wPS{\tilde{w}_{\Pi\sm S}}
\newcommand\wPT{\tilde{w}_{\Pi\sm T}}
\newcommand\wX{\tilde{w}_X}
\newcommand\Xn{\type{X}{n}}

\begin{document}
\title[Dynkin diagrams and Kantor pairs]{Dynkin diagrams and short Peirce gradings of Kantor pairs}
\author{Bruce Allison}
\address[Bruce Allison]{Department of Mathematical and Statistical Sciences \\ University
of Alberta}
\email{ballison@ualberta.ca}
\author{John Faulkner}
\address[John Faulkner]
{Department of Mathematics\\
University of Virginia\\
Kerchof Hall, P.O.~Box 400137\\
Charlottesville VA 22904-4137 USA}
\email{jrf@virginia.edu}

\date{February 10, 2017}

\begin{abstract}In a recent article with Oleg Smirnov, we defined short Peirce (SP) graded Kantor pairs.
For any such pair $P$, we defined a family, parameterized by the Weyl group of type $\text{BC}_2$, consisting
of SP-graded Kantor pairs called Weyl images of $P$.   In this article, we classify finite dimensional
simple SP-graded Kantor pairs over an algebraically closed field of characteristic $0$ in terms of
marked Dynkin diagrams, and we show how to compute Weyl images using these diagrams.  The theory is
particularly attractive for close-to-Jordan Kantor pairs (which are variations of Freudenthal triple systems),
and we construct the reflections of such pairs (with  nontrivial gradings) starting from Jordan matrix pairs.
\end{abstract}

\maketitle

Suppose  in this introduction that ${\bbK}$ is a field of characteristic $\ne 2$ or $3$.
A \emph{Kantor pair} is a pair   $P = (P^-,P^+)$ of ${\bbK}$-modules together with trilinear products $\tprod^\sg : P^\sg \times P^\msg \times P^\sg \to P^\sg$, $\sg = \pm$, satisfying two $5$-linear identities (see Section 4.2) which were first written down by Isai Kantor in
\cite{K1} in the special case of \emph{Kantor triple systems} (when $P^- = P^+$ and $\tprod^- = \tprod^+$).
 Examples of Kantor pairs, or structures that give rise to Kantor pairs, have arisen in the work of many different authors (see Section 4.3  and \cite[\S 3.1]{AFS} for some references).

The motivation for the study of Kantor pairs is their relationship with
$5$-graded Lie algebras.  We now recall that relationship in the special case of primary interest to us when the structures involved are simple, in which case the relationship is a 1-1 correspondence.  (Our convention, as described in Subsection  \ref{subsec:termgrade}, is that the term \emph{simple} for a graded structure is interpreted in the ungraded sense.)
If $L = L_{-2} \oplus  L_{-1} \oplus L_{0} \oplus L_{1} \oplus L_{2}$ is a simple $5$-graded  Lie algebra, then $(L_{-1},L_{1})$ is a simple Kantor pair  with products given by $\{x,y,z\}^\sg = [[x,y],z]$ for $\sg = \pm$, in which case we say that the pair $(L_{-1},L_{1})$ is \emph{enveloped} by $L$.
Conversely, any simple Kantor pair $P$ is enveloped by a simple $5$-graded Lie algebra $L$, which is unique
up to graded isomorphism \cite{AFS} and which has the property that $L_{\sg 2} \simeq K^\sg(P^\sg,P^\sg)$
(as vector spaces) for $\sg = \pm$, where
$K^\sg(x,z) \in \Hom(P^\msg,P^\sg)$ is defined  by
$K^\sg(x,z) w = \{x,w,z\}^\sg - \{z,w,x\}^\sg$ for $x,z\in P^\sg$.

A Kantor pair $P$ satisfying $K^\sg(P^\sg,P^\sg) = 0$ for $\sg = \pm$ is called a (linear) \emph{Jordan pair} \cite{L}, and
the relationship between Kantor pairs and $5$-graded Lie algebras, generalizes the well-known relationship between Jordan pairs and  $3$-graded Lie algebras.
With this important special case in mind, we define a \emph{close-to-Jordan pair} to be a Kantor pair $P$
with $\dim(K^\sg(P^\sg,P^\sg)) = 1$ for $\sg = \pm$.

If $\Delta$ is a finite root system (possibly not reduced), we define a \emph{$\Delta$-grading} of a Lie algebra $L$, to be a grading of $L$ by the root lattice of $\Delta$, with support contained in $\Delta\cup \set{0}$ (see Section 3). With this terminology we see that $5$-graded Lie algebras can be thought of as
$\BCone$-graded Lie algebras, by which we mean a Lie algebra graded
by the root system $\set{-2\al_1,-\al_1, \al_1,2\al_1}$ of type $\BCone$.  Hence the correspondence described above
can be viewed as a 1-1 correspondence between simple Kantor pairs (up to isomorphism) and simple $\BCone$-graded Lie algebras (up to graded isomorphism).

We next recall how $\BCtwo$-graded Lie algebras also play a role in the study of Kantor pairs.
A \emph{short Peirce grading} (\emph{SP-grading}) of a Kantor pair $P$ is a $\bbZ$-grading
$P = P_0 \oplus P_1$ of $P$ with $P_i = 0$ for $i \ne   0,1$.
These gradings were introduced in \cite{AFS}, where it was shown that
there is a 1-1 correspondence between simple SP-graded  Kantor pairs and simple  $\BCtwo$-graded Lie algebras that
is analogous to the one just described.
Given a simple SP-graded Kantor pair $P$, this correspondence
was used to construct  a simple SP-graded Kantor pair $\p{u}P$, called the
\emph{$u$-image} of $P$, for each
$u$ in the Weyl group of the root system of type $\BCtwo$.
In this way one obtains eight \emph{Weyl images} of a simple SP-graded Kantor pair $P$, no two of which are (in general) graded-isomorphic.
Of particular interest is the $s_1$-image of $P$, which we simply call the \emph{reflection of $P$},
where $s_1$ is the reflection corresponding to the short basic root.
As an application of Weyl images, it was shown in \cite{AFS} that reflection applied to Jordan pairs of
skew-transformations yields a new class of (in general) infinite dimensional simple SP-graded Kantor pairs that are not themselves Jordan.

In this article, we use subsets of Dynkin diagrams to classify simple SP-graded Kantor pairs and  to compute their Weyl images in the special case  of finite dimensional pairs over  an algebraically closed field $\bbK$
of characteristic $0$.  For the rest of this introduction, we make these additional assumptions and briefly outline our main results.  In these results, a simple Kantor pair is said to be of \emph{type
$\Xn$}, if the root system of its enveloping simple 5-graded Lie algebra is of type $\Xn$.

With the exception of Section \ref{sec:KPfd}, which we discuss below, Sections \ref{sec:HomSD} to \ref{sec:SP}
are devoted to laying the  necessary groundwork on homomorphisms of root systems, root graded Lie algebras and Kantor pairs.  Of particular interest is Theorem \ref{thm:rootgrG}, which
uses homomorphisms of root systems to classify all $\Delta$-gradings (with $\Delta$ arbitrary)  of a finite dimensional semi-simple Lie algebra.

In Section \ref{sec:SPfd}, we assume that
$\Pi$ is the Dynkin diagram of type $\Xn$, where $\Xn$ is the type of an irreducible reduced finite root system.
We first define a set $\SPA(\Pi)$  of pairs of subsets of $\Pi$, whose elements  are said to be \emph{SP-admissible}.   In Proposition
\ref{prop:path}. we show how to easily write down the elements of $\SPA(\Pi)$  using the extended diagram of $\Pi$.
Then, in Theorem \ref{thm:SP}, we give a classification of simple SP-graded Kantor pairs of type $\Xn$  by showing that they  are (up to graded isomorphism) in 1-1 correspondence with the orbits in
$\SPA(\Pi)$ under the right action of $\Aut(\Pi)$.  The proof uses
 Theorem \ref{thm:rootgrG} when  $\Delta$ is irreducible of type $\BCtwo$.

If we specialize Theorem \ref{thm:SP}  to the case when the SP-grading is the zero grading (i.e.~$P = P^0$),
we obtain a classification of simple (ungraded) Kantor pairs of type $\Xn$, which states that they are (up to isomorphism)  in 1-1 correspondence with the
orbits in the set $\KA(\Pi)$ of \emph{Kantor admissible} subsets of $\Pi$ under the right action
of $\Aut(\Pi)$.
For the sake of readability, we actually state this classification theorem earlier as Theorem
\ref{thm:Kantor}  in Section \ref{sec:KPfd}.
We note that Theorem \ref{thm:Kantor} is equivalent to Kantor's classification
of non-polarized simple Kantor triple systems \cite{K1} (see Remark \ref{rem:Kantor}).
Nevertheless we include a complete treatment of the topic
here, since we are working in the context of pairs rather than triple systems, and since the article \cite{K1}  is unavailable to some readers and does not include all details.

 Both Theorems \ref{thm:Kantor} and \ref{thm:SP} take particularly simple and attractive forms in the case of close-to-Jordan pairs, and for this reason we highlight this example throughout the paper.  For example, we see in Theorems  \ref{thm:KanSko} and  \ref{thm:symplectic} that if $n\ge 2$ there is exactly one simple close-to-Jordan pair of type $\Xn$ and that this Kantor pair can be viewed as the signed double of a Freudenthal triple system \cite{M1} with a modified product.

If Section \ref{sec:Weylcompute}, we introduce a left  action, denoted by $*$, of the Weyl group $\WD$ of the root system $\Delta$ of type
$\BCtwo$ on the set $\SPA(\Pi)$.
In Theorem \ref{thm:uPST}, we show that, if $u\in \WD$, the $u$-image of the SP-graded Kantor pair corresponding to $(S,T) \in \SPA(\Pi)$ is the SP-graded Kantor pair corresponding to $u*(S,T)\in \SPA(\Pi)$.
Then in Theorem \ref{thm:uST}, we show how to compute $u*(S,T)$ using only information about the Dynkin diagram $\Pi$. The proof of this result is rather intricate, involving the longest element of the Weyl group of certain subsets of $\Pi$.
Combining Theorems \ref{thm:uPST} and \ref{thm:uST} with our classification Theorem \ref{thm:SP},  one can easily compute the Weyl images of any simple SP-graded Kantor pair using only information about $\Pi$.

As an application of our results, we obtain in Section \ref{sec:close}  a construction of the reflections of all simple nontrivially SP-graded close-to-Jordan pairs in the form $J \otimes U$,
where $J$ is a Jordan pair of matrices and $U$ is a pair of two-dimensional  spaces.
This generalizes Kantor's construction of his remarkable  Kantor triple system $C_{55}^2$, since the double
of that triple system is a reflection  of the  close-to-Jordan pair of type  $\type{E}{6}$
\cite[\S 7]{AFS}.

\emph{Acknowledgments:}  Isai Kantor's methods and ideas are used in many places in this article,
and we are pleased to acknowledge his  significant influence.  Also, we have greatly benefited from many conversations with Oleg Smirnov about this work, and about Kantor pairs and graded Lie algebras in general.

\section{Some sets of homomorphisms between root systems}
\label{sec:HomSD}

In this section we investigate some sets of homomorphisms between roots systems.

\subsection{Root systems}
\label{subsec:rootsystem}
In this paper, a \emph{root system} will mean a  finite root system $\De$ in a finite
dimensional real Euclidean space $E_\Delta$ as described for example
in \cite[VI, \S1.1]{Bour}.   $\Delta$ is said to be \emph{reduced}   if
$(2\Delta) \cap \Delta = \emptyset$.
The \emph{rank} $\rD$ of $\Delta$ is the dimension of $E_\Delta$.
If $\De$ is irreducible and $n = n_\Delta$, the \emph{type} of $\De$ is one of  $\type{A}{n} (n\ge 1))$, $\type{B}{n} (n\ge 2)$, $\type{C}{n} (n\ge 3))$, $\type{D}{n} (n\ge 4))$, $\type{E}{n} (n = 6,7,8) $, $\type{F}{4}$, $\type{G}{2}$, or $\type{BC}{n} (n \ge 1)$,
where the last case occurs if and only if $\Delta$ is not reduced
\cite[VI, \S 1.4, Prop.~14]{Bour}.

We  use the notation $\QD := \spann_\bbZ(\Delta)$
for the \emph{root lattice} of $\Delta$, which is a free abelian group  of rank $n_\Delta$.
The \emph{automorphism group of $\De$}, denoted by
$\Aut(\De)$, is the stabilizer of $\De$ in $\GL(E_\De)$.  Using the restriction map,
we often identify  $\Aut(\De)$ with the stabilizer of $\Delta$ in $\Aut(Q_\De)$.
The \emph{Weyl group}  of $\De$ (contained in $\Aut(\De)$), will be denoted by
$W_\De$. We denote the Euclidean product on $E_\Delta$
by $(\ ,\ )$,  and write  $\langle \al,\beta \rangle = 2(\al,\beta)/(\beta,\beta)$
for $\al,\beta\in\Delta$.
If $\Delta$ is irreducible, we denote the set of roots of minimum length (\emph{short roots}) in $\De$ by $\De_\sh$.

If a base $\Pd$ for the root system $\Delta$ is fixed, we let $\Delta^+$ be the
set of \emph{positive roots relative to $\Pd$}.
If   $\al = \sum_{\gamma\in \Pd}  n_\gamma \gamma\in \QD$, where $n_\gamma\in \bbZ$, we
define the \emph{height} of $\al$ to be  $\height_{\Pd}(\al) :=\sum_{\gamma\in \Pd}  n_\gamma$
and the  \emph{support} of $\al$ in $\Pd$ to be
$\supp_{\Pd}(\al) := \set{\gamma\in \Pd \suchthat n_\gamma \ne 0}$.
We denote the stabilizer of $\Pd$ in $\Aut(\Delta)$
by $\Aut(\Pd)$. If  $\Delta$ is reduced, the set $\Pd$ has the additional structure
of a \emph{Dynkin diagram}, in which case the restriction map
identifies  $\Aut(\Pd)$
with the  automorphism group of the Dynkin diagram $\Pd$ \cite[VI, \S 4.2]{Bour}.
If  $\Pd = \set{\al_i \suchthat i \in I}$ is indexed by a finite set $I$, the $I\times I$-matrix
$C(\Pd) = (\langle\al_i, \al_j\rangle)_{i,j\in I}$ is called the \emph{Cartan matrix}
of  $\Pd$.

If $\cG$ is a finite dimensional simple Lie algebra over an algebraically closed field of characteristic $0$ with Cartan subalgebra $\cH$, the classical theory for such Lie algebras tells us that the  root system $\Sigma(\cG,\cH)$ of $\cG$ relative to $\cH$
is an irreducible reduced root system  in the Euclidean space
$E_{\Sigma(\cG,\cH)} = \mathbb{R}\otimes_\mathbb{Q} \spann_\mathbb{Q}\Sigma(\cG,\cH)$ \cite[\S 8.5]{H}.  In that case, the \emph{type of $\cG$} is defined to be the type of  $\Sigma(\cG,\cH)$.

\subsection{The sets $\Hom(\Sigma,\Delta)$ and $\Hom_\pos(\Sigma,\Delta)$ \protect} \label{subsec:HomSD}

For the rest of the section, \emph{we assume that $\De$ is a root system of rank $\rD$ with base $\Pd$,
and $\Sigma$ is a root system of rank $\rS$ with base $\Ps$}.
Elements of $\QD$ will be normally denoted by $\al, \beta, \dots$,
while elements of $\QS$ will normally be denoted by $\mu, \nu, \dots$.

Let
\[\Hom(\Sigma,\Delta) :=
\set{\rho\in \Hom(\QS,\QD) \suchthat \rho(\Sigma) \subseteq \Delta\cup \set{0}}.\]
We call elements of $\Hom(\Sigma,\Delta)$ \emph{homomorphisms of $\Sigma$ into $\Delta$}.
We say that $\rho\in \Hom(\Sigma,\Delta)$ is \emph{positive (relative to
$\Ps$ and $\Pd$)}
if
$\rho(\Sigma^+) \subseteq \Delta^+ \cup \set{0}$ (or equivalently
$\rho(\Ps) \subseteq \Delta^+ \cup \set{0}$).
Let
\[\Hom_\pos(\Sigma,\Delta) := \set{\rho\in \Hom(\Sigma,\Delta) \suchthat \rho \text{ is positive}}.\]

\subsection{Some actions by composition}
\label{subsec:comphom}
Recall that if $G$ and $G'$ are abelian groups (written additively) then there is a \emph{right action of $\Aut(G)$ by composition} on the group  $\Hom(G,G')$, as well as a \emph{left action by composition of $\Aut(G')$ on $\Hom(G,G')$}.  These are defined respectively by
\[\rho \cdot \ph :=\rho\circ \ph \andd \theta \cdot \rho  :=\theta\circ \rho \]
for $\ph\in \Aut(G)$, $\rho\in \Hom(G,G')$ and $\theta\in \Aut(G')$.
Since these two actions commute, we can write
expressions like $\theta\cdot\rho\cdot\ph$ for $\ph\in \Aut(G)$, $\rho\in \Hom(G,G')$  and $\theta\in \Aut(G')$.

In particular, $\Aut(Q_\Sigma)$ acts on the right by composition and $\Aut(Q_\Delta)$ acts on the left by composition
 on $\Hom(\QS,\QD)$.  Moreover,
clearly \[\Aut(\Delta) \cdot  \Hom(\Sigma,\Delta) \cdot \Aut(\Sigma) \subseteq
\Hom(\Sigma,\Delta).\]
Hence \emph{$\Aut(\Sigma)$ (resp.~$\Aut(\Delta)$) acts on the right (resp.~left)
on $\Hom(\Sigma,\Delta)$ by composition}.

\begin{lemma}  \label{lem:WSorbit}  Each orbit in $\Hom(\Sigma,\Delta)$ under the right action
of $\WS$ by composition contains a unique element of the set $\Hom_\pos(\Sigma,\Delta)$.
\end{lemma}

\begin{proof}  We will need some notation for the proof.
First let $\overline{C(\Ps)} = \set{\mu\in \ES \suchthat (\mu,\nu) \ge 0 \text{ for } \nu\in \Ps}$
denote the closure of the fundamental Weyl chamber in $\ES$ determined by $\Ps$.
Also, if $\tau\in \hQo$, then (since $\QS$ is a lattice in $\ES$)
there exists a unique $\tau^*\in \ES$ such that
$(\tau^*,\mu) = \tau(\mu)$ for $\mu \in \QS$.  Since $\WS$  preserves the
form $(\ ,\ )$ we see that
\[(\tau\cdot w)^* = w^{-1} \tau^*\]
for $\tau\in \hQo$ and $w \in \WS$.
Moreover, if  $\tau\in \hQo$, then $\tau^*\in \overline{C(\Ps)}$
if and only if $\tau(\mu) \ge 0$ for $\mu\in \Sigma^+$.

To prove existence, suppose that $\rho \in \Hom(\Sigma,\Delta)$.
Since $\rho \in \Hom(\QS,\QD)$, there exists unique elements $\rho_1,\dots,\rho_\rD\in \Hom(\QS,\bbZ)$
such that $\rho(\mu) = \sum_{i=1}^\rD \rho_i(\mu) \al_i$ for $\mu\in \QS$, where
$\Pd= \set{\al_1,\dots,\al_\rD}$.
Then if $\mu\in\Sigma$, we have
\begin{equation}
\label{eq:rhocond}
\rho_i(\mu)\rho_j(\mu) \ge 0 \text{ for } \mu\in \Sigma\text{ and } 1\le i \ne j \le \rD,
\end{equation}
since  $\rho(\mu)\in \Sigma \cup \set{0}$.
Let  $\tilde{\rho}:= \sum_{i=1}^\rD\rho_i \in\hQo$.
Then, since $\overline{C(\Ps)}$ is a fundamental
domain for $\WS$
\cite[V, \S~3.2, Thm.~1]{Bour},
we can choose
$w\in \WS$ such that $w^{-1}\tilde \rho^*\in\overline{C(\Ps)}$.  Hence
$(\tilde\rho\cdot w)^* \in\overline{C(\Ps)}$, so  $(\tilde\rho \cdot w)(\mu)\ge 0$ for $\mu\in \Sigma^+$.
Therefore,  $\sum_{i=1}^\rD \rho_i(w\mu)  \ge 0$
for $\mu\in \Sigma^+$,
and hence, by \eqref{eq:rhocond}, $\rho_i(w\mu)\ge 0$ for $\mu\in \Sigma^+$, $1\le i \le \rD$. So
$\rho\cdot w$ is positive.

For uniqueness suppose that $\rho \in \Hom(\Sigma,\Delta)$, $w\in \WS$ and both
$\rho$ and $\rho\cdot w$ are positive.  Choose $\rho_1,\dots,\rho_\rD\in \Hom(\QS,\bbZ)$
as in the previous paragraph. Then $\rho_i(\mu) \ge 0$ and $(\rho_i\cdot w)(\mu) \ge 0$ for $\mu\in \Sigma^+$ and $1\le i \le \rD$.
Thus $\rho_i^*$ and $w^{-1}\rho_i^*$ are in $\overline{C(\Ps)}$. So,
again since $\overline{C(\Ps)}$ is a fundamental
domain for $\WS$, we have $\rho_i^* = w^{-1}\rho_i^*$ and hence $\rho_i = \rho_i\cdot w$.
Thus  $\rho = \rho\cdot w$.
\end{proof}

Finally, it is clear that
\[\Aut(\Pd) \cdot  \Hom_\pos(\Sigma,\Delta) \cdot \Aut(\Ps) \subseteq
\Hom_\pos(\Sigma,\Delta),\]
so
\emph{$\Aut(\Ps)$ (resp.~$\Aut(\Pd)$) acts on the right (resp.~left)
on $\Hom_\pos(\Sigma,\Delta)$ by composition}.

\subsection{The left action $*$ of $\Aut(\Delta)$ on $\Hom_\pos(\Sigma,\Delta)$ \protect}
\label{subsec:*}

If $\theta\in \Aut(\Delta)$ and $\rho\in \Hom_\pos(\Sigma,\Delta)$, we let
$\theta *\rho$ denote
the unique element of $\theta\cdot \rho \cdot W_\Sigma$ that is contained in
$\Hom_\pos(\Sigma,\Delta)$ (see Lemma \ref{lem:WSorbit}).
That is
\[\set{\theta *\rho} = \Hom_\pos(\Sigma,\Delta) \cap (\theta\cdot \rho \cdot W_\Sigma).\]
One checks easily that $* : (\theta,\rho) \mapsto \theta*\rho$ \emph{is a left action
of $\Aut(\Delta)$ on $\Hom_\pos(\Sigma,\Delta)$}, and that this action
commutes with the right action $\cdot$ of $\Aut(\Ps)$ on the same set.
(The latter fact uses only the observation that $\ph^{-1} W_\Sigma \ph \subseteq \WS$ for
$\ph \in \Aut(\Ps)$.)

\subsection{The sets $\Hom_\sh(\Sigma,\Delta)$ and $\Hom_\psh(\Sigma,\Delta)$ \protect}
\label{subsec:Homshort}
Suppose $\De$ is irreducible.   Let
\[\Hom_\sh(\Sigma,\Delta) := \set{\rho\in \Hom(\Sigma,\Delta) \suchthat \rho(\Sigma)\cap \De_\sh \ne \emptyset}.\]
Note that $\Aut(\Delta) \cdot  \Hom_\sh(\Sigma,\Delta) \cdot \Aut(\Sigma) \subseteq
\Hom_\sh(\Sigma,\Delta)$,
so \emph{$\Aut(\Sigma)$ (resp.~$\Aut(\Delta)$) acts on the right (resp.~left)
on $\Hom_\sh(\Sigma,\Delta)$ by composition}.

Also let
\[\Hom_\psh(\Sigma,\Delta) := \Hom_p(\Sigma,\Delta) \cap \Hom_\sh(\Sigma,\Delta).\]
Again $\Aut(\Pd) \cdot  \Hom_\psh(\Sigma,\Delta) \cdot \Aut(\Ps) \subseteq
\Hom_\psh(\Sigma,\Delta)$,
and hence  \emph{$\Aut(\Ps)$ (resp.~$\Aut(\Pd)$) acts on the right (resp.~left)
on $\Hom_\psh(\Sigma,\Delta)$ by composition}.

Finally, it is clear that  $\Aut(\Delta)*\Hom_\psh(\Sigma,\Delta)\subseteq \Hom_\psh(\Sigma,\Delta)$,
so \emph{we have a left action $*$ of $\Aut(\Delta)$
on  $\Hom_\psh(\Sigma,\Delta)$}.

\section{Root graded Lie algebras}
\label{sec:rootgraded}

\emph{We  will assume for the rest of the article that $\bbK$ is a
commutative associative ring of scalars containing $\frac 16$}.
 (We will add the additional assumption that $\bbK$ is an algebraically closed field of characteristic~0 in our classification results.)

\emph{All modules, algebras, trilinear pairs and triple systems will be assumed to be over~$\bbK$.}

\subsection{Terminology for graded structures}
\label{subsec:termgrade}  \emph{For graded algebras and graded trilinear pairs, we will use the unmodified terms simple and isomorphic
in the ungraded sense}.
More specifically, suppose $G$ be an abelian group.
A $G$-graded Lie algebra $L$ will be said to be \emph{simple} (resp.~\emph{graded simple}) if the only
non-trivial proper ideals (resp.~graded-ideals) of $L$ are $0$ and $L$.
If $L$ and $L'$ are $G$-graded algebras, we will say that $L$ is
\emph{isomorphic} (resp.~\emph{graded-isomorphic}) to $L'$, written $L\simeq L'$ (resp.~$L\simgr L'$), if there is an isomorphism (resp.~graded-isomorphism) of $L$ onto~$L'$.
We will use similar (and evident) terminology and notation for  graded trilinear pairs
(see  Subsection \ref{subsec:SPgr}).

If $L$ is an algebra, two $G$-gradings of $L$ are said to
be \emph{\griso{}}
if the corresponding graded algebras are graded-isomorphic.  (The modifier graded in not needed in this term since there is no ambiguity.)
Again we will use similar terminology for  gradings of trilinear pairs.

If $L$ is a  $G$-graded Lie algebra, an automorphism $\omega$ of $L$ is said to be \emph{grade reversing} if $\omega(L_\gamma) = L_{-\gamma}$ for  $\gamma \in G$.

If  $\rL$ is a $G$-graded algebra and  $\theta\in \Aut(G)$, we  let $\p{\theta}\rL$  be the
$G$-graded algebra  such that $\p{\theta}\rL = \rL$ as  algebras and
\[(\p{\theta}\rL)_\al = \rL_{\theta^{-1}\al}\]
for   $\al\in G$.  We  call $\p{\theta}\rL$ the \emph{$\theta$-image} of $\rL$.  Clearly   $\p{1} \rL = \rL$ and
\begin{equation}\label{eq:left0} \p{\theta_1}(\p{\theta_2}\rL) = \p{\theta_1\theta_2}\rL
\text{ for } \theta_1,\theta_2\in \Aut(G).
\end{equation}

\subsection{Root graded Lie algebras}
\label{subsec:rootgraded}
Let   $\De$ be a root system.

As in \cite{AFS}, a \emph{$\De$-grading} of a  Lie algebra $\rL$ will mean a $Q_\De$-grading of
$\rL$ such that $\supp_{Q_{\Delta}}(\rL) \subseteq \Delta \cup \set{0}$,
where  $\supp_{Q_{\Delta}}(\rL)$ denotes the \emph{support} of $\rL$ in $Q_\De$.
In that case we call $\rL$ together with the $\De$-grading a
\emph{$\De$-graded Lie algebra}. (\emph{We  do not assume the existence of a
grading subalgebra as in \cite{ABG} and \cite{BS},
or equivalently a  family of $\spl_2$-triples  as in \cite{N1}.})  Finally,
if $\De$ is irreducible of type $\Xn$, where $n = \rD$, we sometimes call
a $\De$-graded Lie algebra an  \emph{$\Xn$-graded Lie algebra}.

If $\Pd$ is a base for $\De$, \emph{we often use the basis $\Pd$ for $Q_\De$ to identify $Q_\De$
with $\bbZ^\rD$, and in this way $\De$-graded Lie algebras are
$\bbZ^\rD$-graded Lie algebras}.

If $\rL$ is  a $\De$-graded Lie algebra and $\theta\in \Aut(\De)$, then the $\theta$-image $\p{\theta}\rL$ of $L$
is a $\De$-graded Lie algebra.
If $\theta\in W_\De$,
we call   $\p{\theta}\rL$ a \emph{Weyl image} of~$\rL$.

\section{Root gradings of finite dimensional semi-simple Lie algebras}
\label{sec:rootgrG}

\emph{Suppose  in this section that
${\bbK}$ is an algebraically closed field
of characteristic 0, and that $\cG$ is a  finite dimensional semisimple Lie algebra over $\bbK$}.

\emph{Fix a Cartan subalgebra $\cH$ of $\cG$, let  $\cG = \oplus_{\mu\in \cH^*} \cG_\mu$  be the root space decomposition of $\cG$ relative to $\cH$, and let $\Sigma = \Sigma(\cG,\cH)$ be the root system of $\cG$ relative to $\cH$} (see Subsection \ref{subsec:rootsystem}).
\emph{We fix a base $\Ps$ for this root system}.

\subsection{Gradings of $\cG$ by a finitely generated free abelian group}
\label{subsec:ZngradG}
We first   describe the $G$-gradings of $\cG$
up to isomorphism, where $G$ is a free abelian group of finite rank written additively.

If $\rho\in\Hom(\QS,G)$, \emph{let $\cG(\rho)$ be the $G$-graded  Lie algebra whose underlying Lie algebra is $\cG$ and whose grading is defined by}
\begin{equation} \label{eq:Grho} \textstyle
\cG(\rho)_{\gamma}=\sum_{\mu\in \QS,\ \rho(\mu)=\gamma}\cG%
_{\mu}
\end{equation}
\emph{for $\gamma\in G$}.  We call this grading the
$\rho$\textit{-grading} of $\cG$.

It is easily checked  that
\begin{equation} \label{eq:left1}
\p{\theta}\cG(\rho) = \cG(\theta\cdot \rho) \end{equation}
as  $G$-graded Lie algebras for  $\theta\in \Aut(G)$ and $\rho \in \Hom(\QS,G)$.

\begin{remark} \label{rem:inducedgrading}
The  gradings
on $\p{\theta}L$ and $\cG(\rho)$ defined above are each examples of \emph{gradings induced by homomorphisms} as defined in \cite[\S 1.3]{EK}; and  \eqref{eq:left0} and \eqref{eq:left1} follow from an evident general fact about induced gradings. Nevertheless, we use different notations for $\p{\theta}L$ and $\cG(\rho)$ to emphasize the different roles they play in our  theory.
\end{remark}

The first statement  of the next proposition is
a corollary of a more general result on graded algebras
proved in \cite{EK} using methods from algebraic groups (see \cite[Prop.~1.34 and Cor.~1.35]{EK}).
It seems likely that the second and third statements are also known,
although we are not aware of a  reference.
Nevertheless, for the reader's convenience,
we give a proof of all three statements  using Lie algebra methods.

\begin{proposition}
\label{prop:Zngrading}  Suppose that $G$ is a finitely generated free abelian group.
\begin{itemize}
\item[(i)]  Any $G$-grading of
$\cG$ is \griso{}  to a $\rho$-grading for some $\rho\in\Hom(\QS,G)$.
\item[(ii)]  If $\rho,\rho'\in\Hom(\QS,G)$, then  the
$\rho$-grading and the $\rho'
$-grading are \griso{} if and only if $\rho,\rho'$ lie in the same
orbit in $\Hom(\QS,G)$ under the right action
of $\Aut(\Sigma)$ by composition.
\item[(iii)]  If $\cG = \bigoplus_{\gamma\in G} \cG_\gamma$ is a $G$-grading of $\cG$,
there exists a  period 2 grade reversing automorphism  of $\cG$. So
$\dim(\cG_{-\gamma}) = \dim(\cG_\gamma)$  for  $\gamma\in G$.
\end{itemize}
\end{proposition}

\begin{proof}  If $G = 0$ the statements are trivial, so we can assume that $G = \bbZ^\ell$, where $\ell \ge 1$.

(i) Suppose that $\cG$ is $G$-graded.  Define derivations $d_{1},\ldots,d_{\ell}$
of $\cG$ by $d_{i}(x)=k_{i}x$ for $x\in\cG_{(k_{1},\ldots,k_{\ell})}$.  Since $\cG$ is semisimple, every derivation is
inner \cite[Thm.~5.3]{H}, so we can write $d_{i}=ad(h_{i}) $ with
$h_{i}\in\cG$ for $1\le i \le n$.  Since the $d_{i}$ are semisimple and commute, the
$h_{i}$ span an abelian  ad-diagonalizable subalgebra of $\cG$, and hence they are contained in a Cartan subalgebra
$\cH'$ of $\cG$  \cite[Cor.~15.3]{H}.  By \cite[Cor.~16.4]{H}, we may choose
$\eta\in \Aut(\cG)$ with $\eta
(\cH')=\cH$.  Using $\eta$
to transfer the grading, we can assume that $\cH'=\cH%
$.  Define $\rho_i\in \hQo$ by $\rho_{i}(\mu)=\mu(h_{i})$ for $\mu\in \QS$.
Then,  $d_{i}(x)=[h_{i},x]=\rho
_{i}(\mu)x$ for  $x\in\cG_{\mu}$.
Thus, the grading is the
$\rho$-grading of $\cG$, where $\rho (\al) = (\rho_1(\al),\dots,\rho_\ell(\al))$
for $\al\in \QS$.

(ii) Suppose first that $\rho=\rho'\cdot\ph$, where $\rho,\rho'\in\Hom(\QS,G)$, $\ph\in \Aut(\Sigma)$.
Then  \cite[Thm.~14.2]{H}
 tells us that there is $\eta\in \Aut(\cG)$ with $\eta(\cG_{\mu
})=\cG_{\ph(\mu)}$ for $\mu\in \QS$.  Hence for $\gamma \in G$ we have
$\eta(\cG(\rho)_{\gamma})=\sum_{\mu\in \QS,\ \rho(\mu)=\gamma
}\cG_{\ph(\mu)}=\sum_{\nu\in \QS,\ \rho'(\nu)=\gamma
}\cG_{\nu} \ = \ \cG(\rho')_{\gamma}$.
so the $\rho$ and $\rho'$-gradings are isomorphic.

Conversely,
suppose that $\rho,\rho'\in\Hom(\QS,G)$ and $\eta:\cG(\rho)\to \cG(\rho')$  is a $G$-graded-isomorphism. Clearly
$\cH$ is a Cartan subalgebra of $\cG(\rho)_{0}$, so
$\eta(\cH)$ and $\cH$ are Cartan subalgebras of
$\cG(\rho')_{0}$.  By Corollary 16.4 of \cite{H}, there exists
$\eta'\in \Aut(\cG)$ such that
$\eta'\eta(\cH)=\cH$ and $\eta'$ is
a product of automorphisms of $\cG$ the form
$\exp(\ad(x))$, where $x\in \cG(\rho')_{0}$
and the transformation $\ad(x)$ of $\cG$ is nilpotent.
Clearly, $\eta'$ is a
graded automorphism of $\cG(\rho')$, so we can
replace $\eta$ by $\eta'\eta$ to assume that $\eta(\cH%
)=\cH$.  Let $\cH^*$ be the dual space of $\cH$ and let $\eta^\sharp\in \GL(\cH^*)$ be the inverse dual of
$\eta|_\cH \in \GL(\cH)$, so $(\eta^\sharp(\mu))(\eta( h)) = \mu(h)$
for $h\in \cH$, $\mu\in \cH^*$. Then $\eta(\cG_\mu) = \cG_{\eta^\sharp(\mu)}$
for $\mu\in\cH^*$, so $\eta^\sharp(\Sigma) = \Sigma$.
Hence $\ph := \eta^\sharp|_\QS \in \Aut(\Sigma)$.
Thus for $\mu\in \Sigma$,  we have
\[\cG_{\ph(\mu)} = \eta(\cG_\mu) \subseteq \eta(\cG(\rho)_{\rho(\mu)})
= \cG(\rho')_{\rho(\mu)}.
\]
So $\rho'(\ph(\mu)) = \rho(\mu)$ for $\mu\in \Sigma$, and hence
$\rho  = \rho' \cdot \ph$.

(iii) By (i), we  can assume that the given grading of $\cG$ is
the $\rho$-grading where $\rho\in\Hom(\QS,G)$.
Choose $\omega\in\Aut(\cG)$
of period 2 of $\cG$ such that $\omega(\cG_\mu) = \cG_{-\mu}$ for $\mu \in \QS$
\cite[VIII, \S4.4, Prop.~5]{Bour}.    Then, by
\eqref{eq:Grho}, we have $\omega(\cG(\rho)_\gamma) = \cG(\rho)_{-\gamma}$ for~$\gamma \in G$.
\end{proof}

\subsection{Classification of the root gradings of \protect $\cG$ \protect}
\label{subsec:rootgrG}  The following lemma is clear from the  definitions involved.

\begin{lemma} \label{lem:rootgrG}\
If $\Delta$ is a root system and $\rho \in  \Hom(\QS,Q_\Delta)$, then  the $\rho$-grading of $\cG$ is
a $\Delta$-grading if and only if $\rho \in  \Hom(\Sigma,\Delta)$.
Moreover, if $\De$ is  irreducible and $\rho \in  \Hom(\Sigma,\Delta)$,
then $\cG(\rho)_\al \ne 0$ for some $\al\in \De_\sh$ if and only if $\rho\in \Hom_\sh(\Sigma,\Delta)$.
\end{lemma}

We now classify the $\De$-gradings of $\cG$ up to isomorphism in  terms of
orbits in $\Hom_\pos(\Sigma,\De)$.

\begin{theorem} \label{thm:rootgrG}  Suppose
$\Delta$ is a root system with  base $\Pi_\Delta$.  Then
\begin{itemize}
\item[(i)]  Any $\Delta$-grading
of $\cG$ is \griso{} to the  $\rho$-grading of $\cG$ for some
$\rho\in \Hom_\pos(\Sigma,\Delta)$, where  $\Hom_\pos(\Sigma,\Delta)$ is the space of positive
homomorphisms of $\Sigma$ into $\Delta$ relative to $\Pi_\Sigma$ and $\Pi_\Delta$.
\item[(ii)]
If $\rho$ and $\tau$
are in $\Hom_\pos(\Sigma,\Delta)$, then the
$\rho$ and $\tau$-gradings of $\cG$ are \griso{}
if and only if $\rho$ and $\tau$ are in the same orbit in $\Hom_\pos(\Sigma,\Delta)$ under the right action of $\Aut(\Ps)$ by composition.
\end{itemize}
\end{theorem}

\begin{proof}  (i):  By   Proposition \ref{prop:Zngrading}(i),
we can assume the grading is a $\rho$-grading for some $\rho\in \Hom(\QS,\QD)$.
Then, by Lemma \ref{lem:rootgrG}, $\rho\in \Hom(\Sigma,\Delta)$.
Thus, by Lemma \ref{lem:WSorbit} and
Proposition \ref{prop:Zngrading}(ii), we can assume that $\rho\in  \Hom_\pos(\Sigma,\Delta)$.

(ii): The implication ``$\Leftarrow$'' follows from
Proposition \ref{prop:Zngrading}(ii).
For the converse, suppose that the $\rho$ and $\tau$-gradings are \griso{}.
Then, by Proposition \ref{prop:Zngrading}(ii), there exists
$\ph\in \Aut(\Sigma)$ such that $\tau = \rho\cdot \ph$.
By
\cite[VI, \S~1.5, Prop.~16]{Bour},
we may write $\ph = \pi w$, where $\pi\in \Aut(\Ps)$ and
$w\in \WS$.  So $\tau \cdot w^{-1} = \rho \cdot \pi$ is positive,
and hence by uniqueness  in Lemma \ref{lem:WSorbit}, $\tau = \tau \cdot w^{-1}$.
Thus $\tau = \rho \cdot \pi$ as  desired.
\end{proof}

\section{Kantor pairs}
\label{sec:rgKP}

We assume again that
\emph{${\bbK}$ is a commutative associative ring containing $\frac 1 6$}.
 \protect
\subsection{Trilinear pairs and triple systems}
\label{subsec:trilinearpair}
A \emph{trilinear pair}  is by definition a triple $(P,\tprod^-,\tprod^+)$, consisting of a pair $P = (P^-,P^+)$ of ${\bbK}$-modules together with two trilinear
products $\tprod^\sg : P^\sg \times P^\msg \times P^\sg \to P^\sg$, $\sg = \pm$.
The map $\tprod^\sg$ is called the \emph{$\sg$-product} for $P$, $\sg = \pm$.  We usually abbreviate $(P,\tprod^-,\tprod^+)$ as $P$ or $(P^-,P^+)$.

If $P = (P^-,P^+)$ is   a trilinear pair, the \emph{opposite} of $P$ is the
trilinear pair
$P^\op$, with $(P^\op)^\sg = P^\msg$ and with $\sg$-product
equal to the $\msg$-product of $P$ for $\sg = \pm$.

A \emph{triple system} (sometimes also called a \emph{ternary algebra}) is a pair $(X,\tprod)$ consisting of a ${\bbK}$-module $X$ together with a trilinear
product $\tprod : X \times X \times X \to X$.  Again we usually write $(X,\tprod)$ simply as $X$.

We have evident notions of ideal, simplicity and isomorphism for trilinear pairs \cite[\S 2.1]{AFS} and for triple systems.

There are two natural ways to obtain a trilinear pair from a triple system $X$.  Indeed,  if $\xi = \pm 1$,
the $\xi$-\emph{double} of a triple system $X$
is the  trilinear pair $(X,X)$ with $\sg$-products defined by
$\tprod^+ = \tprod$ and $\tprod^- = \xi\tprod$.  We  call the $1$-double (resp.~the $(-1)$-double) of $X$
the \emph{double} of $X$ (resp.~\emph{the signed double}) of~$X$.

A \emph{polarization} of a triple system $X$ is a module decomposition
$X = X^-\oplus X^+$ such that $\{X^\sg,X^\msg,X^\sg\} \subseteq X^\sg$,  $\{X^\sg,X^\sg,X\} = 0$ and
$\{X,X^\sg,X^\sg\} = 0$ for $\sg = \pm$.  We say that $X$ is \emph{non-polarized} if
$X$ does not have a polarization.

As noted in \cite[\S 2.1]{AFS}, the following fact is easy to verify:

\begin{lemma}\label{lem:double}  Suppose $X$ is a triple system and $\xi = \pm 1$.
Then the $\xi$-double of $X$ is  a simple  pair if and only if $X$ is non-polarized and simple.
\end{lemma}

We say that two triple systems $X$ and $X'$  are \emph{isotopic} if there exist
linear isomorphisms $\ph^\sg : X \to X'$ such that
$\ph^\sg\{x,y,z\} = \{ \ph^\sg x,\ph^\msg y,\ph^\sg z\}$ for $\sg = \pm$.  (In  \cite{K1,K2},
Kantor uses the term \emph{weakly isomorphic} rather than isotopic.)  Note that if $X$ and $X'$ are triple systems, then
$X$ and $X'$ are isotopic if and only if their doubles (resp.~their signed doubles) are  isomorphic.

\subsection{Kantor pairs}
\label{subsec:Kantorpair}
If $P$ is a trilinear pair, $\sg = \pm$,   $x\in P^\sg$, $y\in P^\msg$ and $z\in P^\sg$,
 we define
 $D^\sg(x,y)\in \End(P^\sg)$ and $K^\sg(x,z)\in \Hom(P^\msg,P^\sg)$ by
 \[D^\sg(x,y)u  = \{x,y,u\}^\sg \andd K^\sg(x,z)w= \{x,w,z\}^\sg-\{z,w,x\}^\sg.\]
for $u\in P^\sg$, $w\in P^\msg$, $\sg = \pm$.
We then say that $P$ is a  \emph{Kantor pair} \cite{AF1, AFS} if the following identities hold:
\begin{gather*}
[D^\sg(x,y),D^\sg(z,w)] =
D^\sg(D^\sg(x,y)z,w) -
D^\sg(z,D^\msg(y,x)w),\\
K^\sg(x,z)D^\msg(w,u)+ D^\sg(u,w)K^\sg(x,z) =
K^\sg(K^\sg(x,z)w,u)
\end{gather*}
for $x,z,u  \in P^\sg$,  $y,w \in P^\msg$, $\sg = \pm$,

Clearly, the opposite of a Kantor pair
is a Kantor pair.

A (linear) \emph{Jordan pair} is a Kantor pair $P$ such that $K^\sg(P^\sg,P^\sg) = 0$ for $\sg = \pm 1$.
These pairs make up the best understood and most studied class of Kantor pairs  \cite{L,N1}.  The reader can consult
\cite{AFS} (and its references) to see many other examples of Kantor pairs.

\subsection{Freudenthal-Kantor triple systems}  \label{subsec:FKTS}
If  $\xi = \pm 1$, a \emph{$(-\xi,1)$-Freudenthal-Kantor triple system (abbreviated as $(-\xi,1)$-FKTS)}, is
a triple system $X$ whose $\xi$-double is a Kantor pair \cite{YO}.  The defining identities for
such a triple system are:
\begin{gather*} \label{eq:FKTS1}
[L(x,y),L(z,w)] =
L(L(x,y)z,w) -
\xi L(z,L(y,x)w),\\
\xi K(x,z)L(w,u)+ L(u,w)K(x,z) =
K(K(x,z)w,u)  \label{eq:FKTS2}
\end{gather*}
for $x,y,z,w, u  \in X$,
where $L(x,y), K(x,z)\in \End(X)$ are defined by
$L(x,y)z  = \{x,y,z\}$, $K(x,z)y = \{x,y,z\}-\{z,y,x\}$.

\begin{example} \label{ex:FKTS}  We now mention two special cases which together with Remark \ref{rem:11FKTS}  explain the names used in the term
Freudenthal-Kantor triple system.

(i):  We call a $(-1,1)$-FKTS  a \emph{Kantor triple system}, since these triple systems were studied in depth by Kantor in \cite{K1,K2} (where they were called generalized Jordan triple systems of second order).

(ii):  Suppose that $\bbK$ is a field.   A $(1,1)$-FKTS $X$ is said to be
\emph{balanced} if $K(x,y) = \langle x,y \rangle \id_X$ for some  bilinear
form $\bform : X\times X \to \bbK$ \cite[\S 1]{EKO}.  In that case
$\bform$ is a uniquely determined skew-symmetric form that we call the \emph{skew form} of $X$.
If $\bform$ is non-degenerate, we see  in the following remark that $X$ is a  Freudenthal triple system \cite{M1}
with a slightly modified product.
\end{example}

\begin{remark}\label{rem:11FKTS}  There are three other classes of
triple systems which have been studied in the literature,
whose definitions involve a skew form, and  which are none other
than balanced $(1,1)$-FKTS's with slightly modified products.
To be more precise, suppose  $\bbK$ is a field, $X$ is a finite dimensional vector space,
$\langle \, ,\, ,\, \rangle : X \times X \times X \to X$   is a trilinear product
and  $\bform : X \times X \to \bbK$ is a non-degenerate skew-symmetric form.
Then $(X,\langle \, ,\, ,\, \rangle,\bform)$ is a
\emph{balanced symplectic ternary algebra} \cite{FF},
a \emph{symplectic triple system} \cite{YA}, or a  \emph{Freudenthal triple system}  if and only if
$(X,\tprod)$ is a balanced $(1,1)$-FKTS with skew form $\bform$, where
\begin{gather*}
\{x,y,z\} = \langle z,x,y \rangle, \quad
\{x,y,z\} = -\frac 12 \langle x,y,z \rangle + \frac 12 \langle x,y\rangle z \quad \text{ or } \\
\{x,y,z\}= -\frac 12 \langle x,y,z \rangle +
\frac 12 \langle x,y \rangle z+ \frac 12 \langle z,x \rangle y+ \frac 12 \langle z,y \rangle x
\end{gather*}
respectively. (This is easy to check in the first case.
For the other two cases see \cite[Thm.~2.18]{E1} and \cite[Thm.~4.7]{E2}.)
Here for the definition of a Freudenthal
triple system we use the one given in \cite{M1} except that we allow
the quartic form  $\langle \, , \langle \, , \, , \, \rangle \rangle$   to be trivial; \emph{in other words
we allow the product $\langle \, , \, , \, \rangle$  in a Freudenthal triple system to be  trivial}.
%
\end{remark}

\subsection{$\BCone$-graded Lie algebras and Kantor pairs}
\label{subsec:BC1}

The motivation for the study of Kantor pairs is their relationship with $\BCone$-graded Lie algebras.

To recall this from \cite{AFS}  suppose that
\[\De = \De_{\BCone} := \set{-2\al_1,-\al_1,\al_1,2\al_1}\]
\emph{is the irreducible  root system of type $\BCone$ with base  $\Pd = \set{\al_1}$}.
We identify $Q_\De = \bbZ$ using the $\bbZ$-basis $\Pd$ for $Q_\De$.
Then a $\BCone$-graded Lie algebra  is the same thing as a \emph{$5$-graded} Lie algebra
$L = L_{-2} \oplus L_{-1} \oplus L_{0} \oplus L_{1} \oplus L_{2}$.

If
$\rL =  L_{-2} \oplus L_{-1} \oplus L_{0} \oplus L_{1} \oplus L_{2}$
is a  $5$-graded Lie algebra, then $P = (\rL_{-1},\rL_{1})$ is a Kantor pair
with products defined by
\[\{x,y,z\}^\sg = [[x,y],z]\]
for $x,z\in \rL_{\sg 1}$, $y\in \rL_{\msg 1}$, $\sg = \pm$ \cite[Thm.~7]{AF1}.
We call  $P$ the \emph{Kantor pair enveloped by the $5$-graded Lie algebra
$\rL$}, and we say that \emph{the $5$-graded Lie algebra $L$ envelops~$P$}.

Note that $\De_\sh = \set{\pm\al_1}$.  So if $L$ is a $5$-graded Lie algebra
and $P$ is the  Kantor pair enveloped by $L$, then $P^-\oplus P^+ =  \bigoplus_{\al\in \De_\sh} L_\al$ in~$L$.

Every Kantor pair is enveloped by some $5$-graded Lie algebra.  Indeed,
given a Kantor pair $P$, there exists a $5$-graded Lie algebra $\Kan(P)$,
which is unique up to graded-isomorphism, such that
$\Kan(P)$ envelops $P$, $\Kan(P)$ is generated as an algebra by $T$, where
$T= \Kan(P)_{-1} + \Kan(P)_{1}$, and the centre of $\Kan(P)$ intersects trivially with~$[T,T]$ (see \cite[Cor.~3.5.2]{AFS}).
The   \emph{$5$-graded Lie algebra $\Kan(P)$ is called the
Kantor Lie algebra} of $P$, and it is constructed explicitly in \cite[\S 3--4]{AF1}
(see also \cite[\S 3.3]{AFS})
with
\begin{equation}
\label{eq:Kandef}
\Kan(P)_{\sg 1} \simeq_{\text{$\bbK$-mod}} P^\sg ,\ \Kan(P)_{\sg 2} \simeq_{\text{$\bbK$-mod}} K^\sg(P^\sg,P^\sg),\ \Kan(P)_{0} = [\Kan(P)_{-1}, \Kan(P)_{1}]
\end{equation}
for $\sg = \pm$, where $\simeq_{\text{$\bbK$-mod}}$ indicates isomorphism as $\bbK$-modules.

Note that isomorphisms of Kantor pairs  induce graded isomorphisms of their Kantor Lie algebras (see \cite[\S 3.5]{AFS}), and hence two Kantor pairs are isomorphic if and only if their Kantor Lie algebras are graded-isomorphic.

\begin{lemma} \label{lem:chardouble}  Suppose that $P$ is a Kantor pair and $\xi = \pm 1$.  If
$P$ is isomorphic to the $\xi$-double of a trilinear pair, then
there is grade reversing automorphism $\omega$ of $\Kan(P)$ such that $\omega^2\mid_{P^-+P^+} = \xi \id_{P^-+P^+}$.
Conversely if $L$ is a $5$-graded Lie algebra that envelops $P$ and
there is grade reversing automorphism $\omega$ of $L$ satisfying $\omega^2\mid_{P^-+P^+} = \xi \id_{P^-+P^+}$,
then $P$ is isomorphic to the $\xi$-double of the triple system $P^-$ with product
$\set{x,y,z} = \set{x,\omega y, z}^-$.
\end{lemma}

\begin{proof}   (See also  \cite[pp.~5--6]{L} and \cite[Thm.~1]{At} for related observations.) For the first statement we can assume that $P$ equals
the $\xi$-double of a trilinear pair $X$.  Then $(\id_X,\xi \id_X)$  is an isomorphism of $P$ onto $P^\op$, which therefore induces an isomorphism
$\omega$ of $\Kan(P)$ onto $\Kan(P^\op)$.   Since $\Kan(P^\op)$ is $\Kan(P)$ with the grading
reversed, we have the first statement.
For the second statement, an easy calculation shows that $(\id_{P^-},\omega|_{P^+})$ is an isomorphism of $P$ onto the $\xi$-double of the triple system $P^-$ with the indicated  product.
\end{proof}

If $P$ is a Kantor pair, then \cite[Prop.~2.7(iii)]{GLN}
\begin{equation}\label{eq:Kansimp}
P \text{ is simple if and only if } \Kan(P) \text{ is simple}.
\end{equation}
Hence, any simple Kantor pair is enveloped by a simple $5$-graded Lie algebra.
Conversely we  have  the following proposition:

\begin{proposition}
\label{prop:BC1}  \cite[Thm.~3.5.5 and Lemma 3.2.3]{AFS}
Let $P$ be a nonzero Kantor pair and suppose that $L$ is a simple $5$-graded
Lie algebra that envelops $P$.  Then $P$ is simple, $L \simgr \Kan(P)$ and
\begin{equation} \label{eq:Ldecomp}  L = [L_{-1}, L_{-1}] \oplus L_{-1} \oplus [L_{-1}, L_{1}] \oplus L_1
\oplus [L_{1}, L_{1}]\end{equation}
\end{proposition}

\subsection{Balanced dimension and close-to-Jordan pairs} \label{subsection:baldim}
Suppose  in this subsection that
$\bbK$ is a field.

We say that  a Kantor pair $P$ is \emph{finite dimensional} if each $P^\sg$ is finite dimensional.  In view of \eqref{eq:Kandef},
$P$ is finite dimensional if and only if $\Kan(P)$ is finite dimensional.

As in \cite[\S 2.1]{AFS}, we say that a  Kantor pair $P$ has
\emph{balanced dimension}, if $\dim(P^+) = \dim(P^-) =d$, where $d$ is a non-negative integer,
in which case we call $d$ the \emph{balanced dimension} of $P$.  Similarly, as in \cite[\S 3.8]{AFS},
we  say $P$ has \emph{balanced 2-dimension}, if $\dim(K^+(P^+,P^+)) = \dim(K^-(P^-,P^-)) = e$,
where $e$ is a non-negative integer, in which case we call $e$ the \emph{balanced 2-dimension} of $P$.
Note that by \eqref{eq:Kandef}, the balanced dimension of $P$ (resp.~the balanced $2$-dimension of $P$),
is given, when it is defined, by $d = \dim(\Kan(P)_{\sg 1})$
(resp.~$e = \dim(\Kan(P)_{\sg 2})$)
for $\sg = \pm$.

\begin{lemma}\label{lem:baldim}  If $P$ is finite dimensional and  $P^\op \simeq P$,
then $P$ has balanced dimension and balanced $2$-dimension.
\end{lemma}
\begin{proof} We have
$\dim(P^+) = \dim((P^\op)^-) = \dim(P^-)$ and
$\dim(K^+(P^+,P^+))= \dim(K^-((P^\op)^-,(P^\op)^-)) = \dim(K^-(P^-,P^-))$.
\end{proof}

Note  that  a Kantor pair $P$ is Jordan if and only $P$ has balanced
$2$-dimension $0$.  With this in mind, we view  balanced $2$-dimension (when it is defined) as a
measure of distance from  Jordan theory.
In this spirit, we make the following definition.

\begin{definition}  A \emph{close-to-Jordan Kantor pair}, or simply a
\emph{close-to-Jordan pair}, is a Kantor pair of balanced  $2$-dimension $1$.
\end{definition}

\subsection{Finite dimensional simple Kantor pairs}\label{subsection:KPsimple}
Suppose in this subsection that \emph{$\bbK$ is an algebraically closed  field of characteristic $0$}.

\begin{proposition} \label{prop:KTS}  A  trilinear pair $P$ is a finite dimensional
simple Kantor pair if and only it is isomorphic to the double of a
finite dimensional non-polarized simple Kantor triple system.  Moreover in that case
$P^\op  \simeq P$, and hence $P$ has balanced dimension and balanced $2$-dimension.
\end{proposition}

\begin{proof}
By Lemma \ref{lem:double}, we have the implication ``$\Leftarrow$'' in the first statement.
Suppose conversely that $P$ is a finite dimensional simple Kantor pair.
By \eqref{eq:Kansimp} and Proposition \ref{prop:Zngrading}(iii),
we can choose a period 2 grade reversing  $\omega \in \Aut(\Kan(P))$ .
Then
$\omega$ exchanges $P^+$ and $P^-$, so $(\omega \mid_{P^-},\omega \mid_{P^+})$
is an isomorphism of $P$ onto   $P^\op$.  The rest now follows from Lemmas
\ref{lem:chardouble} and \ref{lem:baldim}
\end{proof}

\begin{remark} \label{rem:equivclass}
We see from
Proposition \ref{prop:KTS} that the problem of classifying
finite dimensional simple Kantor pairs up to isomorphism is equivalent to
the problem of classifying finite dimensional non-polarized simple Kantor triple systems  up to isotopy.
\end{remark}

If $P$ is a finite dimensional
simple Kantor pair over $\bbK$, we define the \emph{type} of $P$ to be the
type of the simple Lie algebra~$\Kan(P)$.

\begin{proposition} \label{prop:fdsimp} Suppose $P$ is a nonzero Kantor pair.  If $L$ is any finite dimensional simple
$5$-graded Lie algebra that envelops $P$,  then $P$ is finite dimensional and simple, $L \simgr \Kan(P)$,
the type of $P$ is the same as the type of $L$,
the balanced dimension of $P$ equals  $\dim(L_{\sg 1})$ for $\sg = \pm$, and
the balanced
$2$-dimension of $P$ equals  $\dim(L_{\sg 2})$ for $\sg = \pm$.
\end{proposition}

\begin{proof}  This follows from
Propositions \ref{prop:BC1}  and \ref{prop:KTS}.
\end{proof}

\section{Classification of finite dimensional simple Kantor pairs}
\label{sec:KPfd}

\emph{Suppose  in this section that $\bbK$ is an algebraically closed field of
characteristic $0$ and  that $\Pi$ is the  connected Dynkin diagram of type $\Xn$}.
We will use certain admissible subsets of $\Pi$, or equivalently certain markings of $\Pi$,
to classify the finite dimensional simple Kantor pairs over $\bbK$.

For this purpose, \emph{we let $\cG$ be a finite dimensional simple Lie algebra
over $\bbK$ of type $\Xn$ with Cartan subalgebra $\cH$ and root system
$\Sigma = \Sigma(\cG,\cH)$ relative to $\cH$ (see Subsection \ref{subsec:rootsystem});
and let $\Ps$ be a base for this root system.  Then there exists a diagram
isomorphism $\iota: \Pi \to \Ps$.
For simplicity we
use $\iota$ to identify $\Pi = \Ps$.  }

\emph{Let  $\mu^+$ be the highest root
of $\Sigma$ relative to $\Pi$, $\mu^- := -\mu^+$, and}
\[\tPi :=\Pi\cup\{\mu^-\}.\]
The Dynkin diagram for $\tPi$ (or simply $\tPi$) is called the \emph{extended Dynkin diagram} (or
completed Dynkin graph in  \cite{Bour})
for~$\Pi$.

\emph{We also assume as in Subsection \ref{subsec:BC1} that
\[\De = \De_{\BCone} := \set{-2\al_1,-\al_1,\al_1,2\al_1}\]
is the irreducible root system of type $\BCone$ with base  $\Pd = \set{\al_1}$},
and we identify $Q_\De = \bbZ$ using the $\bbZ$-basis $\Pd$ for $Q_\De$.
Note that $\Aut(\De) = W_\De = \set{\pm 1}$.

\subsection{Kantor-admissible subsets of \protect $\Pi$ \protect}
\label{subsec:Kantoradmit}
To  emphasize the role of $\BCone$ here, we   write  the sets $\Hom(\Sigma,\Delta)$,
$\Hom_\sh(\Sigma,\Delta)$, $\Hom_\pos(\Sigma,\Delta)$, and
$\Hom_\psh(\Sigma,\Delta)$ respectively as
$\Hom(\Sigma,\BCone)$, $\Hom_\sh(\Sigma,\BCone)$, $\Hom_\pos(\Sigma,\BCone)$ and
$\Hom_\psh(\Sigma,\BCone)$.
Then we have
\begin{equation}
\label{eq:HomBC1}
\begin{gathered}
\Hom(\Sigma,\BCone) = \set{\rho\in \Hom(\QS,\bbZ) \suchthat \rho(\Sigma) \subseteq \set{0, \pm 1, \pm 2}} ,\\
\Hom_\sh(\Sigma,\BCone) =  \set{\rho\in \Hom(\Sigma,\BCone) \suchthat 1 \in \rho(\Sigma)}.
\end{gathered}
\end{equation}

If  $S\subseteq \Pi$, we define $\chi_{S}\in\hQo$ by
\begin{equation} \label{eq:chiS}
\chi_{S}(\lm)=\left\{
\begin{tabular}
[c]{ll}%
$1$ & if $\lm\in S$\\
$0$ & if $\lm\in\Pi\sm S$,
\end{tabular}
\right.
\end{equation}
and we set $\chi_{\lm}=\chi_{\{\lm\}}$ for $\lm\in\Pi$.  Note that $\chi_\emptyset = 0$.

A subset $S$ of $\Pi$ is said to be \emph{Kantor-admissible}
if $\chi_S \in \Hom_\sh(\Sigma,\BCone)$ (or equivalently  $\chi_S \in \Hom_\psh(\Sigma,\BCone)$).
Let
\[\KA(\Pi) = \text{the set of all Kantor-admissible subsets of
$\Pi$}.\]

\begin{remark} \label{rem:KAwelldef}  The
set
$\KA(\Pi)$ (or equivalently the term Kantor-admissible) has been defined using
our choice of $\cG$, $\cH$,  $\Ps$ and $\iota : \Pi \to \Ps$.  However, it is easy to see (since Dynkin diagram
isomorphisms induce root system isomorphisms), that a different choice
of these objects leads to the same set  $\KA(\Pi)$.  In short,
\emph{$\KA(\Pi)$ is well-defined}.
\end{remark}

\begin{proposition} \label{prop:Kadmit1}
The map $S \mapsto \chi_S$
is a bijection of  the set $\KA(\Pi)$
onto the set  $\Hom_\psh(\Sigma,\BCone)$.
\end{proposition}

\begin{proof} All but surjectivity is clear.  For surjectivity, suppose
$\rho\in\Hom_\psh(\Sigma,\BCone)$.  Then  $1\in \rho(\Sigma^+)$,  so $1\in \rho(\Pi)$.
Note also that
$\sum_{\lm\in \Pi} \lm \in \Sigma^+$, since $\Sigma$ is irreducible.
So $0\le \sum_{\lm\in \Pi} \rho(\lm) \le 2$.  At least one term in this
sum is $1$, so all are $0$ or $1$.  Let $S = \set{\lm\in \Pi \suchthat \rho(\lm) = 1}$.
Then $\rho$ and $\chi_S$ agree on $\Pi$, so $\rho = \chi_S$.  Finally
$S \in \KA(\Pi)$ by definition.
\end{proof}

We have the following characterization of Kantor-admissible subsets of $\Pi$.

\begin{proposition} \label{prop:Kadmit2}If   $S$ is a subset of $\Pi$, then the following are equivalent:
\begin{itemize}
  \item [(a)] $S\in \KA(\Pi)$.
  \item [(b)] $S \ne \emptyset$ and $\chi_S\in \Hom(\Sigma,\BCone)$.
  \item [(c)] $\chi_S(\mu^+) \in \set{1,2}$.
  \item [(d)]  $S = \set{\lm}$ with $\chi_\lm(\mu^+)\in \set{1,2}$; or $S = \set{\lm,\lm'}$
  with   $\lm\ne \lm'$ and $\chi_{\lm}(\mu^+) = \chi_{\lm'}(\mu^+) =1$.
\end{itemize}
\end{proposition}

\begin{proof} The implications in cyclic order are clear using \eqref{eq:HomBC1}.
\end{proof}

We define a right action
of $\Aut(\Pi)$ on the set of subsets of $\Pi$ by
\[S \cdot \ph := \ph^{-1}(S).\]
One checks that
\begin{equation} \label{eq:philmS}
\chi_{S}\cdot \ph=\chi_{S\cdot \ph}.
\end{equation}
for  $S\subseteq\Pi$ and  $\ph\in \Aut(\Pi)$.
Hence the right action of $\Aut(\Pi)$ stabilizes the set
$\KA(\Pi)$, so \emph{we have a right action $\cdot$ of
$\Aut(\Pi)$ on $\KA(\Pi)$}.

\subsection{Classification}
\label{subsec:classKP}  We  next use Kantor's approach for constructing Kantor triple
systems \cite{K1} to construct Kantor pairs.

\begin{construction} \label{con:Kantor} (The Kantor pair $\fP(\Pi;S)$)
Let $S \in \KA(\Pi)$.
Then $\chi_S\in \Hom_\sh(\Sigma,\BCone)$, so
$\cG(\chi_S)$ is a $5$-graded Lie algebra with  $\cG(\chi_S)_{-1} \ne 0$ or
$ \cG(\chi_S)_1 \ne 0$ by Lemma \ref{lem:rootgrG}.
Let $\fP(\Pi;S)$
be the Kantor pair enveloped by  $\cG(\chi_S)$ (see Section \ref{subsec:BC1}).
Then $\fP(\Pi;S) = (\fP(\Pi;S)^-,\fP(\Pi;S)^+)$ is nonzero and  given explicitly  by
\[\textstyle \fP(\Pi;S)^\sg = \cG(\chi_S)_{\sg 1} = \sum_{\mu \in \Sigma,\ \chi_S(\mu) = \sg 1}
\cG_\mu \]
for $\sg = \pm$ with products  given by
$\{x,y,z\}^\sg = [[x,y],z]$ in $\cG$.
\end{construction}

\begin{remark} \label{rem:PSwelldef} The  definition of $\fP(\Pi;S)$  uses our
choice of $\cG$, $\cH$,  $\Ps$ and $\iota : \Pi \to \Ps$.
It is easy to see (since Dynkin diagram
isomorphisms induce Lie algebra isomorphisms) that a different choice
of these objects leads to a Kantor pair that is isomorphic to  $\fP(\Pi;S)$.
That is to say, \emph{$\fP(\Pi;S)$ is well-defined up to isomorphism}.
\end{remark}

Since  $\Pi$ is fixed in our discussion,
we will usually  abbreviate our notation and \emph{write $\fP(\Pi;S)$ simply as $\fP(S)$}

\begin{remark} \label{rem:fPS} Let $S\in\KA(\Pi)$.

(i) By  Proposition \ref{prop:fdsimp}, $\fP(S)$  is a finite dimensional simple Kantor pair
of type $\Xn$
 and
$\cG(\chi_S) \simgr\Kan(\fP(S))$.

(ii)  Again using Proposition \ref{prop:fdsimp}, we see that the balanced dimension
and the balanced $2$-dimension of $\fP(S)$ equal respectively
$\dim(\cG(\chi_S)_1)$ and $\dim(\cG(\chi_S)_2)$.  Hence these
quantities are given respectively by
\[|\set{\mu\in \Sigma : \chi_S(\mu) = 1}|  \andd
|\set{\mu\in \Sigma : \chi_S(\mu) = 2}|. \]

(iii) $\fP(S)$ is Jordan if and only if
it has balanced $2$-dimension $0$,  which by (ii) holds if and only if $\chi_S(\mu^+) = 1$.
\end{remark}

We  now state a classification theorem for finite dimensional simple Kantor pairs in terms of Kantor admissible
subsets of $\Pi$.

\begin{theorem}[Kantor] \label{thm:Kantor}
Suppose  $\Pi$ is the  connected Dynkin diagram of   type~$\Xn$.
\begin{itemize}
\item [(i)]  If  $S\in\KA(\Pi)$, then
$\fP(S)$ is a   finite dimensional simple
Kantor pair of type~$\type{X}{\rkn}$.
\item[(ii)]  If $P$ is a  finite dimensional  simple Kantor pair of type
$\type{X}{\rkn}$, then $P$ is isomorphic to
$\fP(S)$ for some $S\in \KA(\Pi)$.
\item[(iii)]  If $S$ and $S'$ are  in $\KA(\Pi)$, then
$\fP(S)$ and  $\fP(S')$ are
isomorphic if and only $S$ and $S'$ are in the same orbit in $\KA(\Pi)$ under  the right action of $\Aut(\Pi)$.
\end{itemize}
\end{theorem}

We  have seen (i) above; and (ii) and (iii) follow easily
using Theorem \ref{thm:rootgrG}, Proposition \ref{prop:fdsimp} and
Proposition \ref{prop:Kadmit1}.
We omit the details since we will obtain the  theorem as a special case of our classification of finite dimensional simple SP-graded Kantor  pairs (see Theorem \ref{thm:SP} and Remark \ref{rem:ungradedproof}).

\begin{remark} \label{rem:Kantor} We have attributed Theorem \ref{thm:Kantor}
to Isai Kantor, since in \cite[\S 4]{K1}  he proved  the  equivalent
classification result (see Remark \ref{rem:equivclass}) for finite dimensional
non-polarized simple Kantor triple systems up to isotopy,
although he omitted some details both in his statements and proofs.
We note that Kantor also gave models of each of the Kantor triple systems that he considered (\cite[\S 5--6]{K1}, \cite{K2}).
We will not look at these models in general, but rather focus on models for the reflections of simple close-to-Jordan pairs later in   Section~\ref{sec:close}.
\end{remark}

\subsection{Marked Dynkin diagrams for simple Kantor pairs}
\emph{We represent $S\in \KA(\Pi)$
by the Dynkin diagram for  $\Pi$ with the nodes in $S$ marked with a circle;
and we use
the same marked Dynkin diagram to represent the Kantor pair $\fP(S)$}.

\begin{example}(Type $\type{E}{6}$).  \label{ex:E6KP}
Suppose that $\Pi = \set{\mu_1,\dots,\mu_6}$ is of type $\type{E}{6}$ with Dynkin diagram
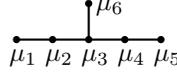
\[\Esixlabelled\]
Now   $\mu^+ = \mu_1 + 2\mu_2 + 3\mu_3 + 2\mu_4 + \mu_5 + 2\mu_6$,
and  the extended Dynkin diagram $\tPi$~is
\begin{equation} \label{eq:E6extended}
\Esixextended
\end{equation}
\cite[VI, \S4.12, (IV)]{Bour}.
So  by Proposition \ref{prop:Kadmit2})(d) there are,
up to the right action of $\Aut(\Pi)$, four Kantor-admissible subsets $S$ of $\Pi$:
\[(i): \set{\mu_1},\qquad (ii): \set{\mu_6},\qquad  (iii):\set{\mu_2},\qquad (iv): \set{\mu_1,\mu_5}.\]
These are represented respectively by   the following marked diagrams
\medskip
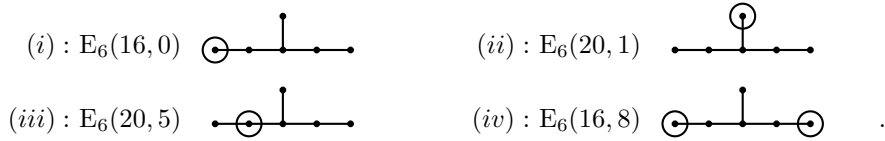
\begin{alignat*}{10}
\allowdisplaybreaks
(i)&: &&\ \type{E}{6}(16,0)\ &&  \clearlabels \Satrue \Esix  \qquad
           &(ii)&: &&\ \type{E}{6}(20,1)\ &&  \clearlabels \Sftrue \Esix \\[-1ex]
(iii)&: &&\ \type{E}{6}(20,5)\  && \clearlabels \Sbtrue \Esix \qquad
            &(iv)&: &&\ \type{E}{6}(16,8)\ &&\clearlabels \Satrue \Setrue \Esix .
\end{alignat*}
So by Theorem \ref{thm:Kantor} there are four  simple Kantor pairs of type $\type{E}{6}$ up to isomorphism,
which are represented by these marked diagrams.

For each $S$ above, we have labelled the marked diagram representing $S$ and $\fP(S)$ by $\type{E}{6}(d,e)$, where
$d$ is the balanced dimension of $\fP(S)$ and $e$ is the balanced $2$-dimension of $\fP(S)$.  ($d$ and $e$ were computed using
Remark \ref{rem:fPS}(ii)  and a list of roots in $\Sigma$
\cite[Plate V]{Bour}.)
\emph{We will sometimes also use the label
$\type{E}{6}(d,e)$ for the Kantor pair $\fP(S)$ itself.}

Note   that the unique pair in the list that is Jordan is  $\type{E}{6}(16,0)$,
and the unique pair in the list that is close-to-Jordan
is $\type{E}{6}(20,1)$.
\end{example}

For each of the other types $\type{X}{\rkn}$, it is easy using the
same method to write down the marked Dynkin diagrams representing
the simple Kantor pairs of type $\type{X}{\rkn}$ up to isomorphism.

\subsection{The close-to-Jordan case} \label{subsec:classclose}
We  next consider the classification in the special case of simple close-to-Jordan pairs
of type $\Xn$.  By Theorem  \ref{thm:Kantor}, there is only one  simple finite dimensional
Kantor pair of type $\type{A}{1}$, and it is Jordan so not close-to-Jordan. For the other types, we have the following:

\begin{theorem}[Kantor and Skopec]
\label{thm:KanSko}  Suppose that $X_\rkn \ne A_1$, and let $S$ be the set of nodes of $\Pi$ that are adjacent to $\mu^-$ in $\tPi$.
\begin{itemize}
\item[(i)]  $\chi_S(\mu) = \langle \mu, \mu^+ \rangle$
for $\mu \in Q_\Sigma$; so $\chi_S(\mu^+) = 2$.
\item[(ii)] $S$ is the unique element of $\KA(\Pi)$ such that $\fP(S)$ is close-to-Jordan.
\item[(iii)] $\fP(S)$  is, up to isomorphism,
the unique  finite dimensional  simple close-to-Jordan pair of type $\type{X}{\rkn}$.
\end{itemize}
\end{theorem}

\begin{proof} Although this result is not stated in this form
in \cite{KS}, it is implicit in the statement and proof of Theorem 4 in \cite{KS}.  For the reader's convenience we give a different proof.

(i): We can assume that $\mu \in \Pi$.  So $\mu\notin \bbQ \mu^+$ (since $X_\rkn \ne A_1$).  Hence
$\langle \mu,\mu^+\rangle \in \set{0,1}$  by \cite[VI, \S1.8, Prop.~25(iv)]{Bour}; whereas $\chi_S(\mu) \in \set{0,1}$ by definition of $\chi_S$.  Finally
$\langle \mu,\mu^+\rangle \ne 0$ iff $\mu \in S$ iff $\chi_S(\mu) \ne 0$.

(ii):  If $S'\in \KA(\Pi)$, recall that by Remark \ref{rem:fPS}(ii),
$\fP(S')$ is  close-to-Jordan  if and only if
$\set{\mu\in \Sigma \suchthat \chi_{S'}(\mu) = 2} \ = \  \set{\mu^+}$.

Now $S\in \KA(\Pi)$ by  Proposition \ref{prop:Kadmit2}(c).
To see that $\fP(S)$ is  close-to-Jordan, it is enough to show that
$\chi_S(\mu) \ne 2$ for $\mu \in  \Sigma^+\setminus\set{\mu^+}$. But if
$\mu \in  \Sigma^+\setminus\set{\mu^+}$, there exists $\kappa\in \Pi$ such that  $\mu \in \mu^+ -\kappa - \sum_{\lm\in \Pi}\bbZ_{\ge 0} \lm$
and $\mu^+ - \kappa\in\Sigma$.  Thus, $\kappa  \in S$, so $\chi_S(\mu) \le 2 - 1 = 1$.

Finally, suppose that $S'\in \KA(\Pi)$ and
$\fP(S')$ is  close-to-Jordan.  So $\chi_{S'}(\mu^+) = 2$.
If there exists $\kappa \in S$  such that $\kappa\notin S'$, then $\mu^+ - \kappa \in \Sigma^+\setminus \set{\mu^+}$ and  $\chi_{S'}(\mu^+ -\kappa) = 2$, a contraction.  So $S\subseteq S'$. Moreover, this inclusion is not proper, since otherwise we would
have $2 = \chi_S(\mu^+) < \chi_{S'}(\mu^+)$.   So $S' = S$

(iii) follows from (ii) and Theorem \ref{thm:Kantor}.
\end{proof}

Example  \ref{ex:E6KP} provides an illustration of Theorem \ref{thm:KanSko}(iii)
in type $\type{E}{6}$.

\begin{remark} \label{rem:mu+} Suppose we have the assumptions of Theorem \ref{thm:KanSko}.

(i)\hphantom{)} One sees checking types case-by-case that $\card(S) = 1$ if  $\Xn \ne \type{A}{n}$.

(ii)  The $5$-grading of $\cG(\chi_S)$ induces a $5$-grading of the set $\Sigma\cup\set{0}$
(which is viewed as a root system in \cite{LN}).   Since $\chi_S(\lambda) = \langle \lambda, \mu^+ \rangle$
for $\lm \in Q_\Sigma$, one sees that this $5$-grading of $\Sigma\cup\set{0}$ is the one described
previously in \cite[\S17.10]{LN}.  We won't need this fact, so we omit the details.
\end{remark}

\subsection{Close-to-Jordan pairs and  Freudenthal-Kantor triple systems} \label{subsec:FTS}
Suppose  in this subsection that
\emph{$S\in \KA(\Pi)$ and $P = \fP(S)$ is close-to-Jordan}.
Recall by Theorem \ref{thm:KanSko}(ii) that $X_\rkn \ne A_1$ and $S$ is the set of nodes of $\Pi$ that are adjacent to
$\mu^-$ in $\tPi$.

Choose  nonzero $e^\sg \in \cG_{\mu^\sg}$ for $\sg = \pm$  such that
\[[h^+,e^\sg] = \sg 2 e^\sg, \text{ where } h^+ = [ e^+,e^-].\]
Then $\cS := {\bbK} e^-\oplus {\bbK} h^+ \oplus e^+ \simeq \spl_2({\bbK})$.
Also, since $P$ is close-to-Jordan, we have
\begin{equation} \label{eq:esg}
\cG(\chi_S)_{\sg 2} = \bbK e^\sg
\end{equation}
for $\sg = \pm$.
We set
\begin{equation*} \label{eq:omega1}
\omega := \exp(\ad e^+) \exp(-\ad e^-) \exp(\ad e^+) \in \Aut(\cG),
\end{equation*}
in which case
\begin{equation}  \label{eq:omega2}
\omega(e^\sg) = -e^\msg,\, \sg = \pm, \andd \omega(h^+)= - h^+.
\end{equation}

We now use $\spl_2$-theory to verify some  properties of the triple
$(e^-,h^+,e^+)$ and the automorphism $\omega$.

\begin{lemma} \label{lem:omega} We have
\begin{itemize}
\item[(i)] $\mu(h^+) = \langle \mu, \mu^+ \rangle= \chi_S(\mu)$ for $\mu \in Q_\Sigma$.
\item[(ii)] $\cG(\chi_S)_i = \set{ g\in \cG \suchthat [h^+,g] = ig}$ for $i\in \bbZ$.
 \item[(iii)]  $\omega$ reverses the $5$-grading of $\cG(\chi_S)$, and hence
$\omega(P^\sg) = P^\msg$ for $\sg = \pm$.
\item[(iv)] $\omega^2\mid_{P^-+P^+} = - \id_{P^- + P^+}$.
\item[(v)] $\omega(x) = - \sg [e^\msg, x]$ for $x\in P^\sg$, $\sg =  \pm$.
\item[(vi)] If $\sg = \pm$, there  exists a unique bilinear form $\zeta^\sg: P^\sg \times P^\msg \to \bbK$
such that
\begin{equation} \label{eq:defzeta1} [[x,a],e^\sg] = \zeta^\sg(x,a)e^\sg
\end{equation}
for $x\in P^\sg$, $a\in P^\msg$. Further
\begin{gather}
\label{eq:defzeta2} [x,\omega a] = -\sg \zeta^\sg(x,a)e^\sg, \\
\label{eq:zetasym}
\zeta^\sg(x,a) = \zeta^\msg(a,x)
\end{gather}
for $x\in P^\sg$, $a\in P^\msg$; and
$\zeta^\sg$ is nondegenerate.
\item[(vii)]  If  $x,y\in P^\sg$ and $a \in P^\msg$, then
\[[[\omega x,a],y] = \zeta^\sg(x,a)\omega y \andd
\omega \big([[x,a],y]\big)  = \zeta^\sg(x,a)\omega y + [[x,a], \omega y].\]
\item[(viii)] If $x,y,z\in P^\sg$, then    $[[x,\omega y],z] - [[z,\omega y],x]  =  -\zeta^\sg(x,\omega z) y.$
\end{itemize}
\end{lemma}

\begin{proof}
(i): The first equality is  standard \cite[Prop.~8.2(g)]{H} and the second was seen in  Theorem \ref{thm:KanSko}(i).

(ii) and (iii):  (ii) follows from (i), and (iii) follows from (ii) and \eqref{eq:omega2}.

(iv) and (v): Let $T = P^- \oplus P^+$ in $\cG$.  Then $T$ is an $\cS$-submodule of $\cG$ under the adjoint action.
Also,  by (ii), $P^-$ and $P^+$ are respectively the $-1$ and $1$
eigenspaces for $\ad(h^+)$ in    $T$.  So, by $\spl_2$-theory
\cite[VIII, \S1.2--1.3]{Bour},
the $\cS$-module $T$ is the direct sum of
copies of the  $2$-dimensional irreducible $\cS$-module.
Then (iv) and (v) follow by well known $2\times 2$-matrix calculations (see for example
\cite[VIII, \S1.5]{Bour}
with $X_\sg = \sg e^\sg$).

(vi):  The first statement follows from \eqref{eq:esg}.
Then $[x,\omega a] = \sg [x,[e^\sg,a]] = -\sg \zeta^\sg(x,a)e^\sg$;
and  \eqref{eq:zetasym} follows by   applying $\omega$
and using \eqref{eq:omega2} and (iv).  For nondegeneracy, it is enough to show that if $x\in P^\sg$  and $[x,P^\sg] = 0$, then $x = 0$.   For this, let  $\kappa$ be the Killing form.  Then
$0 =  \kappa([x,P^\sg],e^\msg) = \kappa(x,[P^\sg,e^\msg]) = \kappa(x,P^\msg)$
by (iii) and (v).  But $\cG(\chi_S)_1$ and $\cG(\chi_S)_{-1}$ are paired by $\kappa$, so $x=0$.

(vii):   Using  \eqref{eq:defzeta2}, \eqref{eq:zetasym} and (v), we have
$[[\omega x,a],y] = -\sg \zeta^\msg(a,x) [e^\msg, y] = \zeta^\sg(x,a)\omega y$.
Then,
since $\ad(e^\msg)$ is a derivation of $\cG$ which kills $P^\msg$, we have
$\omega \big([[x,a],y]\big) = [[\omega x,a],y] + [[x,a],\omega y] = \zeta^\sg(x,a)\omega y + [[x,a],\omega y]$.

(viii):  Using (iv), the first identity of (vii) and \eqref{eq:zetasym}, we have
$[[x,\omega y],z] - [[z,\omega y],x]
= [ [x,z],\omega y]  =  [[\omega^2 z,x],\omega y]  = \zeta^\msg(\omega z,x) \omega^2 y
= - \zeta^\sg(x,\omega z)y$.
\end{proof}

\begin{theorem} \label{thm:symplectic}
 A   trilinear pair is a finite dimensional simple close-to-Jordan pair if and only if it is isomorphic to the signed double of a nonzero finite dimensional balanced $(1,1)$-FKTS with non-degenerate skew-form (see Example \ref{ex:FKTS}(ii)).
 \end{theorem}

 \begin{proof}
 ``$\Rightarrow$" By  Theorem \ref{thm:KanSko}(iii) we can assume that the trilinear pair is $P = \fP(S)$.  Then by Lemma \ref{lem:chardouble}
 and Lemma 5.5.1 (parts (iii) and (iv)), $P$ is isomorphic to the signed double
 of the triple system $X = P^-$ with product $\{a,b,c\} := \{a,\omega b, c\}^-$. So by definition
 $X$ is a $(1,1)$-FKTS.
 Moreover $X$ is balanced with skew form
$\langle a, b \rangle = -\zeta^-(a,\omega b)$
 by Lemma  \ref{lem:omega}(viii); $\bform$ is nondegenerate by Lemma \ref{lem:omega}(vi); and  $X \ne 0$ since $P$ is simple.

 ``$\Leftarrow$" Conversely suppose that  $X$ is a non-zero balanced $(1,1)$-FKTS with nondegenerate skew form.
 Then, we can use the construction  from \cite[Thm.~1]{F}  of a simple $5$-graded Lie algebra
 $\mathfrak S = \mathfrak S(X, \mathfrak R(X))$
 with $\dim(\mathfrak S)_{\sg 2} = 1$.  (Actually, the construction starts with a balanced symplectic ternary algebra, but we have seen the translation in  Remark \ref{rem:11FKTS}.)  One can easily check
 that $\mathfrak S$ envelops the signed-double of $X$.  Hence the signed-double of $X$
 is a simple close-to-Jordan pair by Proposition \ref{prop:fdsimp}.  \end{proof}

In \cite[Cor.~2]{FF}, the authors used the work of Meyberg \cite{M1} on Freudenthal triple systems, to give constructions
(mainly as $2\times 2$ matrix systems  with Jordan entries) of all
balanced symplectic algebras with nondegenerate skew forms. (Note
however that there is a missing term $(a_1 \times a_2) \times b_3$
in the expression for $c$ in \cite[(3.4)]{FF}.)  In view of Remark
\ref{rem:11FKTS}, this gives constructions of all finite dimensional
balanced $(1,1)$-FKTS's with nondegenerate skew forms.  Hence, by
Theorem \ref{thm:symplectic}, we obtain constructions of all finite dimensional simple close-to-Jordan pairs.

\section{SP-graded Kantor pairs} \label{sec:SP}

\emph{Assume again that $\bbK$ is a commutative associative ring containing $\frac 16$}.

\subsection{SP-graded Kantor pairs} \label{subsec:SPgr}
Suppose  $P$ is a trilinear pair.    If  $G$ is an abelian group and $P_g = (P_g^-,P_g^+)$ for $g\in G$,
where $P_g^\sg$ is an ${\bbK}$-submodule of $P^\sg$ for $g\in G$, $\sg = \pm$, we write
$P =  \bigoplus_{g\in G} P_g$ to mean that $P^\sg = \bigoplus_{g\in G} P^\sg_g$ for $\sg = \pm$.

We say that $P =  \bigoplus_{g\in G} P_g$ is a \emph{$G$-grading} of $P$ if
\begin{equation*} \label{eq:Zgrading}
\{P_g^\sg,P_{g'}^\msg,P_{g''}^\sg\} \subseteq P_{g-g'+g''}^\sg
\end{equation*}
for $g,g',g''\in G$, $\sg = \pm$.  In that case, each $P_g$ is a subpair of  $P$,
and we say that \emph{$P$ is $G$-graded}.

If
$P$ is $G$-graded, we endow $P^\op$ with the $G$-grading given by
\[(P^\op)_g^\sg = P_g^\msg.\]

We recall from Subsection \ref{subsec:termgrade} that the unmodified terms \emph{simple} and \emph{isomorphic}
for $G$-graded pairs will be used in the ungraded sense, and that we have a notion of \emph{isomorphism} for $G$-gradings on a trilinear pair.

If $P$ is a Kantor pair,
a \emph{short Peirce grading} (or \emph{SP-grading}) of $P$ is a
$\bbZ$-grading $P =  \bigoplus_{i\in \bbZ}P_i$ such that
$P_i^\sg = 0$ for $\sg=\pm$  and $i\ne 0,1$.
In that case we have
$P = P_0 \oplus P_1$, and we call the graded pair $P$
a \emph{short Peirce graded} (or \emph{SP-graded}) Kantor pair.

Any Kantor pair $P$ has at least two SP-gradings, the \emph{zero SP-grading}
$P = P_0$ with $P_1 = 0$, and the
\emph{one SP-grading} $P = P_1$ with $P_0 = 0$. We call these two SP-gradings \emph{trivial}.

Clearly, the opposite of an SP-graded Kantor pair
is an SP-graded Kantor  pair.

\subsection{$\BCtwo$-graded Lie algebras and SP-graded Kantor pairs}
\label{subsec:BC2}
We now  recall from \cite[\S 4]{AFS} the relationship between SP-graded Kantor pairs
and $\BCtwo$-graded  Lie algebras.

For this purpose, \emph{suppose for the rest of this section that}
\begin{equation*}
\label{eq:BC2}
\Delta = \Delta_{\BCtwo} := \{\pm \al_1, \pm \al_2, \pm(\al_1+\al_2), \pm 2\al_1, \pm(2\al_1+\al_2), \pm(2\al_1+2\al_2)\}.
\end{equation*}
\emph{is  the irreducible root system of type $\BCtwo$ with base $\Pd = \set{\al_1,\al_2}$}.
We identify $Q_\De = \bbZ^2$  using the
$\bbZ$-basis $\Pd$ for $Q_\De$, so  any $\BCtwo$-graded Lie algebra
is a $\bbZ^2$-graded Lie algebra.

If  $L=\bigoplus_{(i,j)\in \bbZ^2}\rL_{(i,j)}$ is a $\bbZ^{2}$-graded Lie algebra,
we often write $L_{(i,j)}$ as $L_{i,j}$ for brevity. Then the \textit{first component grading}
of $L$ is defined to be  the $\bbZ$-grading $L=\bigoplus_{i\in \bbZ} L_{i,*}$, where  $L_{i,*}=\bigoplus_{j\in \bbZ} L_{i,j}$.

Suppose   that $L$ is a $\BCtwo$-graded Lie algebra. Then, $L$ with its first component grading is a $5$-graded Lie algebra,
which therefore envelops a Kantor pair
\[P = (L_{-1,*},L_{1,*}) =  ( L_{-1,0} \oplus L_{-1, -1} , L_{ 1,0} \oplus L_{ 1, 1} ).\]
Moreover, $P = P_0 \oplus P_1$, where
\begin{equation*} \label{eq:Pisg}
P_i^\sg = L_{\sg 1,\sg i}
\end{equation*}
for $\sg = \pm$, $i\in\bbZ$, is
an SP-grading of~$P$ \cite[\S 4.3]{AFS}.
We call $P$ with this grading the
\emph{SP-graded Kantor pair enveloped by the $\BCtwo$-graded Lie algebra $L$},
and we say that \emph{the $\BCtwo$-graded Lie algebra $L$ envelops the SP-graded Kantor pair $P$}.

Observe that $\De_\sh = \set{\pm\al_1,\pm(\al_1+\al_2)}$.  So if $L$ is a $\BCtwo$-graded Lie algebra and $P$ is the
SP-graded Kantor pair enveloped by $L$, then (as in the $\BCone$ case)
\begin{equation} \label{eq:Psumshort}
P^-\oplus P^+ \textstyle =  \bigoplus_{\al\in \De_\sh} L_\al \quad \text{in } L.
\end{equation}

Every SP-graded Kantor pair is enveloped by some $\BCtwo$-graded Lie algebra.  Indeed, if $P$ is
an SP-graded Kantor pair, then $\Kan(P)$ has a unique $\BCtwo$-grading, called its \emph{standard $\BCtwo$-grading}, such that
the $\Kan(P)$ with this grading envelops
$P$  \cite[Prop.~4.4.1]{AFS}.

Note that   two $SP$-graded Kantor pairs are graded-isomorphic if and
only if the corresponding Kantor Lie algebras with their standard $\BCtwo$-gradings
are graded-isomorphic \cite[Prop.~4.4.2]{AFS}.

We know that any simple SP-graded Kantor pair is enveloped by a simple $\BCtwo$-graded
Lie algebra, namely $\Kan(P)$ with its standard $\BCtwo$-grading.  Conversely, we have
the following  proposition:

\begin{proposition}
\label{prop:BC2} \cite[Prop.~4.4.2(iii)]{AFS}
Let   $P$ be a nonzero SP-graded Kantor pair and suppose that $L$
is a simple $\BCtwo$-graded Lie algebra that envelops $P$.
Then $L$ is graded-isomorphic to $\Kan(P)$ with its standard
$\BCtwo$-grading.
\end{proposition}

\begin{proposition} \label{prop:Pop=Pgr}  If $\bbK$ is an algebraically closed field
of characteristic $0$ and $P$ is a  finite dimensional
simple SP-graded Kantor pair over $\bbK$, then $P^\op \simgr P$.  Consequently $P_i \simeq P_i^\op$, so
the Kantor pair $P_i$ has balanced dimension and balanced $2$-dimension for $i=0,1$.
\end{proposition}

\begin{proof}
By Proposition \ref{prop:Zngrading}(iii), we can choose  $\omega \in \Aut(\Kan(P))$ of period 2
such that $\omega(\Kan(P)_{k,i}) = \Kan(P)_{-k,-i}$ for $k,i\in \bbZ$. Then
$\omega$ exchanges $P^+_i$ and $P^-_i$ for $i=0,1$, so $(\omega \mid_{P^-},\omega \mid_{P^+})$
is a graded-isomorphism of $P$ onto  $P^\op$.  Hence $P_i \simeq P_i^\op$ for $i=0,1$, and the proof is complete by
Lemma \ref{lem:baldim}.
\end{proof}

\subsection{Weyl images of SP-graded Kantor pairs}
\label{subsec:SPWeyl}

Let  $s_\al\in W_\De$ be the reflection through the hyperplane
orthogonal to $\al$ for $\al \in \De$, and put $s_i = s_{\al_i}$ for $i=1,2$.
The generators $s_1$ and $s_2$ of $W_\De$ satisfy  $s_1s_2s_1s_2 = s_2s_1s_2s_1 = -1$,
and
$\Aut(\De) = W_\De = \set{1, s_1, s_2,  s_2s_1, -1,  -s_1, -s_2,  -s_2s_1}$
is the dihedral group of order $8$.

Let $P$ be an SP-graded Kantor pair and let $u\in W_\De$.
If we choose a $\BCtwo$-graded Lie algebra $L$ that envelops $P$, then $\p{u}L$
is  also a $\BCtwo$-graded Lie algebra that therefore
envelops an SP-graded Kantor pair which we denote by $\p{u}P$.
It turns out that $\p{u}P$ is independent of
the choice of $L$ \cite[Lemma 5.1.2(iv)]{AFS}, and we call
$\p{u}P$ the \emph{$u$-image} (or a \emph{Weyl image})
of $P$.  It is clear that  \emph{Weyl images respect graded isomorphisms}; that
is  $P \simgr Q \implies \p{u}P \simgr \p{u}Q$.
Moreover,  $\p{1} P = P$, and, by \eqref{eq:left0},
\[\p{u_1}(\p{u_2}P) = \p{u_1u_2} P \quad \text{  for } u_1,u_2\in \WD.\]

Since $s_1$ and $s_2$ generate $W_\Delta$, the SP-graded Kantor pairs $\p{s_1}P$ and
$\p{s_2}P$ are of particular importance.
For convenience, we use the notation
\[\bP := \p{s_1}P\]
and call this SP-graded Kantor pair the \emph{reflection} of $P$.
It  is easy to check that
\begin{equation} \label{eq:refKP}
\bP^\sg_i = P^{\pi(i)\sg}_i,
\end{equation}
for $\sg = \pm$, $i=0,1$, where $\pi(0) = -$ and $\pi(1) = +$ \cite[Prop.~5.2.1]{AFS}.
Moreover the $\sg$-product $\{,,\rbrv$ on $\bP$ is  given in terms
of the products on $P$ by
\begin{gather*}
\{x_i^\sg,y_i^\msg,z^\sg_i\rbrv
= \{x_i^\sg,y_i^\msg,z_i^\sg\}^{\pi(i)\sg},\
\{x_{1-i}^\sg,y_{1-i}^\msg,z^\sg_{i}\rbrv
= -\{y_{1-i}^\msg,x_{1-i}^\sg,z_{i}^\sg\}^{\pi(i)\sg},\\
\{x_i^\sg,y_{1-i}^\msg,z^\sg_i\rbrv
= 0, \andd
\{x_{i}^\sg,y_{1-i}^\msg,z^\sg_{1-i}\rbrv
= K^{\pi(i)\sg}(x_{i}^\sg,y_{1-i}^\msg)z_{1-i}^\sg.
\end{gather*}
for $\sg = \pm$, $i=0,1$, where $x^\tau_j, y^\tau_j, z^\tau_j\in \bP^\tau_j$ in each case \cite[Prop.~5.2.1]{AFS}.
However, we will not use these expressions in this article, but rather directly use the definition of $\bP$ given above.  It turns out that in
general $\bP$ is not isomorphic to $P$ as an ungraded Kantor pair (as we saw in \cite{AFS}
and will see again in Example~\ref{ex:e6reflect}).

In contrast,  the SP-graded Kantor pair $\p{s_2}P$ has an easy description.  We have
\[\p{s_2}P = \bar P,\]
where $\bar P$ is the SP-graded Kantor pair, called the \emph{shift}
of $P$ that equals $P$ as a Kantor pair and has $\bbZ$-grading given
by  $\bar P_i = P_{1-i}$ for $i\in \bbZ$ \cite[(21)]{AFS}.
In particular, $\bar P$ is isomorphic as an ungraded pair to~$P$.

Finally, it is clear that  the SP-graded pair $\p{-1}P$ is simply
$P^\op$.

\begin{proposition} Suppose  $P$ is a simple SP-graded Kantor pair and $u\in \WD$.  Then $\p{u}P$ is simple.  Moreover, if  $\bbK$ is an algebraically closed field of characteristic $0$ and $P$ is finite dimensional, then $\p{u}P$  has the same balanced dimension as $P$.
\end{proposition}

\begin{proof}  The  first statement is seen in \cite[Prop.~4.1.4]{AFS}.  For the second statement,  it suffices to prove that
$\bar P$ and $\bP$ have the same balanced dimension as $P$ (since $\WD = \langle s_1,s_2\rangle$).
This is clear for $\bar P$; while  for $\bP$ it follows from
\eqref{eq:refKP} and the fact that $P_0$ has balanced dimension by
Proposition \ref{prop:Pop=Pgr}.
\end{proof}

\begin{remark} \label{rem:shift2dim}
Under the assumptions of the second statement of the proposition, it is clear that  the shift $\bar P$ of an SP-graded $P$ has the same balanced $2$-dimension as $P$.  However  that is not true
for the reflection $\bP$ of $P$ (see Example \ref{ex:e6reflect} below).
\end{remark}

\section{Classification of finite dimensional simple SP-graded Kantor pairs}
\label{sec:SPfd}
\emph{For  the rest of this article we assume  that $\bbK$ is
an algebraically closed field of characteristic $0$,
and $\Pi$ is the connected Dynkin diagram of type $\Xn$}.

 \emph{We also assume $\cG$ is a finite dimensional simple Lie algebra
over $\bbK$ of type $\Xn$ with Cartan subalgebra $\cH$ and root system
$\Sigma = \Sigma(\cG,\cH)$ relative to $\cH$.  Let
$\Ps$ be a base for this root system, in which case there
exists a diagram
isomorphism $\iota: \Pi \to \Ps$, which we use to identify $\Pi = \Ps$}.

Also,  as in Section \ref{sec:KPfd},
\emph{$\mu^+$ is the highest root of $\Sigma$ relative to $\Pi$,
$\mu^- = -\mu^+$ and
$\tPi=\Pi\cup\{\mu^-\}$ is the extended Dynkin diagram for $\Pi$}.
If  $Y$ is a non-empty subset of $\tPi$, then $Y$ is the disjoint union of its connected components;  and if $\lm\in Y$, we use the notation $\comp(Y,\lm)$ for the \emph{connected component of $Y$ containing  $\lm$}.

\emph{We further assume that
\[\Delta = \Delta_{\BCtwo} := \{\pm \al_1, \pm \al_2, \pm(\al_1+\al_2), \pm 2\al_1, \pm(2\al_1+\al_2), \pm(2\al_1+2\al_2)\}\]
is   the irreducible root system of type $\BCtwo$ with base  $\Pd = \set{\al_1,\al_2}$}.

We identify $Q_\De = \bbZ^2$ using the $\bbZ$-basis $\Pd$ for $Q_\De$.
We then have the identification
\begin{equation*} \label{eq:identhom}
\Hom(\QS,Q_\Delta) = \Hom(\QS,\bbZ^2) = \hQo^2,
\end{equation*}
where in the last equality $(\rho_1,\rho_2)$ in $\hQo^2$ is identified with the element of $\Hom(\QS,\bbZ^2)$
given by   $\mu \mapsto(\rho_1(\mu),\rho_2(\mu))$.

\subsection{SP-admissible pairs of subsets of $\Pi$}
\label{subsec:SPadmit}

To
emphasize the  role of $\BCtwo$ here, we  will write  $\Hom(\Sigma,\Delta)$,
$\Hom_\sh(\Sigma,\Delta)$, $\Hom_\pos(\Sigma,\Delta)$ and
$\Hom_\psh(\Sigma,\Delta)$ respectively as
$\Hom(\Sigma,\BCtwo)$, $\Hom_\sh(\Sigma,\BCtwo)$, $\Hom_\pos(\Sigma,\BCtwo)$ and
$\Hom_\psh(\Sigma,\BCtwo)$.
So, setting $\square_2 := \set{0,1,2}\times \set{0,1,2}$, we have
\begin{align}
\label{eq:HomBC2}
\Hom(\Sigma,\BCtwo) &= \{\rho\in \Hom(\QS,\bbZ^2) \suchthat \notag
\\
&\qquad\qquad\rho(\Sigma)\subseteq \square_2 \cup (-\square_2),
\ (1,2)\notin \rho(\Sigma),\  (0,2) \notin \rho(\Sigma)\},
\\
\label{eq:HomBC2sh}
\Hom_\sh&(\Sigma,\BCtwo) =  \set{(\rho_1,\rho_2)\in \Hom(\Sigma,\BCtwo) \suchthat 1\in  \rho_1(\Sigma)}.
\end{align}

If $(S,T)$ is a pair of subsets of $\Pi$,  we  use the notation
\[\chiST := (\chi_S,\chi_T) \in \hQt\]
and say $(S,T)$ is \emph{SP-admissible}
if  $\chiST\in \Hom_\sh(\Sigma,\BCtwo)$
(or equivalently $\chiST\in \Hom_\psh(\Sigma,\BCtwo)$).
We let
\[\SPA(\Pi) = \text{the set of all SP-admissible pairs of subsets of
$\Pi$}.\]

\begin{remark} \label{rem:SPAwelldef}
Just as in Remark \ref{rem:KAwelldef}, the  set $\SPA(\Pi)$  is well-defined.
\end{remark}

\begin{proposition} \label{prop:SPadmit} The map $(S,T) \mapsto \chiST$ is a
bijection of the set $\SPA(\Pi)$ onto the set $\Hom_\psh(\Sigma,\BCtwo)$.
\end{proposition}

\begin{proof}  All but surjectivity is clear.  For surjectivity, suppose that $\rho =(\rho_1,\rho_2)\in \Hom_\psh(\Sigma,\BCtwo)$.
Then, by \eqref{eq:HomBC2} and \eqref{eq:HomBC2sh}, $\rho_1(\Sigma) \subseteq \set{0,\pm 1,\pm 2}$ and $1\in \rho_1(\Sigma)$.
Hence, since  $\rho_1(\al) \ge 0$ for $\al\in \Sigma^+$, we have $\rho_1 = \chi_S$ for some $S \subseteq\Pi$ by Proposition \ref{prop:Kadmit1}.
Also, if $1\in \rho_2(\Sigma)$, we have $\rho_2 = \chi_T$ for some $T\subseteq \Pi$ by the same argument.  So we can assume
that $1\notin \rho_2(\Sigma)$.  Hence, by \eqref{eq:HomBC2},
$\rho(\Sigma)\subseteq \set{(0,0), \pm (1,0) ,  \pm (2,0), \pm  (2,2)}$.
Therefore,  $\cG(\rho)_{-1,*} + \cG(\rho)_{1,*} \subseteq \cG(\rho)_{*,0}$
(using the notation of Section \ref{subsec:BC2}).  But the left hand side
of this inclusion is  nonzero by Lemma \ref{lem:rootgrG}, so this space generates $\cG$ as an algebra by
\eqref{eq:Ldecomp}.
Therefore $\cG(\rho) = \cG(\rho)_{*,0}$, so $\rho_2(\Sigma) = 0$.  Thus $\rho_2 = 0 = \chi_\emptyset$.
\end{proof}

\begin{lemma}
\label{lem:SPadmit}  Let $\rho = (\rho_1,\rho_2) \in \Hom(\QS,\bbZ^2)$.  Then
$\rho \in \Hom_\sh(\Sigma,\BCtwo)$ if and only if
$\rho(\Sigma)\subseteq \square_2 \cup (-\square_2)$, $(1,2)\notin \rho(\Sigma)$ and $1\in \rho_1(\Sigma)$.
\end{lemma}
\begin{proof}  In view of \eqref{eq:HomBC2}, we only need to prove ``$\Leftarrow$''.
Let $\rL = \cG(\rho)$, which is a $\bbZ^2$-graded Lie algebra with
$\supp_{\bbZ^2}(L) = \rho(\Sigma \cup \set{0})$. By our assumptions,
the first component grading $\rL=\bigoplus_{i\in \bbZ}\rL_{i,*}$ of $\rL$ is a $5$-grading
with $\rL_{-1,*} + \rL_{1,*} \ne 0$.  Hence, by   \eqref{eq:Ldecomp},
each element in $\rL$ has the form
$x + \sum [y_j,z_j]$, where $x,y_j,z_j\in  \rL_{-1,*} + \rL_{1,*}$.
Therefore each element of $\supp_{\bbZ^2}(\rL)$
is either an element in
$\supp_{\bbZ^2}(\rL_{-1,*} + \rL_{1,*})$ or it is a sum of two such elements.
But by our assumptions $\supp_{\bbZ^2}(\rL_{-1,*} + \rL_{1,*}) \subseteq \set{\pm (1,0),\pm (1,1)}$.
Hence $(0,2)\notin  \supp_{\bbZ^2}(L)$ as needed.
\end{proof}

\begin{lemma}
\label{lem:connected}  If $\mu\in\Sigma$, then $\{\mu^-\}\cup \supp_{\Pi}\{\mu^+-\mu\}$ is a
connected subset of $\tPi$.
\end{lemma}

\begin{proof}
Let $S(\mu) = \{\mu^-\}\cup \, \supp_{\Pi}\{\mu^+-\mu\}$ for $\mu\in \Sigma$.
We prove that $S(\mu)$ is connected  by  induction on the non-negative
integer $\height_{\Pi}(\mu^+-\mu)$. First, if
$\height_{\Pi}(\mu^+-\mu) = 0$, then $\mu = \mu^+$ and $S(\mu) = \set{\mu^-}$
is connected.  Suppose that $\height_{\Pi}(\mu^+-\mu) > 0$.  Then
$\mu\ne \mu^+$.  Furthermore, we can  assume that $\mu\notin -\Pi$, since
otherwise $S(\mu) = \tPi$ is connected.
So there exists
$\lm\in \Pi$ with
\[ \nu = \mu + \lm \in \Sigma.\]
Then, by the induction hypothesis, $S(\nu)$ is connected.  Further, $S(\mu) = S(\nu) \cup \set{\lm}$.
So it remains to show that $\langle S(\nu),\lm\rangle\ne 0$.
For this we can assume that $\lm\notin S(\nu)$, so
$\chi_\lm(\nu) = \chi_\lm(\mu^+)$.
Hence $\nu+\lm\notin \Sigma$.  Thus, since $\nu-\lm\in\Sigma$,
we have $\langle \nu,\lm\rangle \ne 0$.  But then, since $\nu$ lies in the group generated
by $S(\nu)$, we have $\langle S(\nu),\lm\rangle\ne 0$.
\end{proof}

We have the following characterization of SP-admissible pairs of subsets of $\Pi$.
\begin{proposition}
\label{prop:path}
Suppose that $S,T \subseteq \Pi$. Then the following are equivalent:
\begin{itemize}
  \item[(a)] $(S,T) \in \SPA(\Pi)$.
  \item[(b)] $S \ne \emptyset$ and $\chiST \in \Hom(\Sigma,\BCtwo)$
    \item[(c)] $\chi_S(\mu^+) \in \set{1,2}$; $\chi_T(\mu^+) \in \set{0,1,2}$; and if $\chi_T(\mu^+) = 2$, then
    $\chi_S(\mu^+) = 2$ and  $\comp(\tPi\setminus T, \mu^-)\cap S = \emptyset$.
\end{itemize}
\end{proposition}

\begin{proof}  The equivalence of (a) and (b) is clear, so we only consider
the equivalence of (b) and (c).  For this we can assume that
$\chi_S(\mu^+) \in \set{1,2}$ and $\chi_T(\mu^+) \in \set{0,1,2}$
(since these statements hold if either (b) or (c) is assumed).
Also, if $\chi_T(\mu^+) = 0$ or $1$, then $\chi_T(\mu) \ne \pm 2$ for $\mu\in \Sigma$,
so (b) and (c) are each true.  Thus we can  suppose
$\chi_T(\mu^+) = 2$.
Then
if $\chi_S(\mu^+) = 1$, statements (b) and (c) are each false.  So we can suppose
$\chi_S(\mu^+) = 2$.
It remains to show that $(S,T)\in \SPA(\Pi)$ if and only  if
$\comp(\tPi\setminus T, \mu^-)\cap S = \emptyset$.
We establish the contrapositives of these implications.

Suppose  that $\comp(\tPi\setminus T, \mu^-)\cap S \ne \emptyset$.
Then there is a path $\lm_{0},\lm_{1},\ldots,\lm_{r}$ in
$\tPi$ which does
not pass through $T$ with $r\ge 1$, $\lm_{0}=\mu^-$, $\lm_{r}\in S$.  We can shorten this path if necessary to assume that the $\lm_i$'s are distinct and
that $\lm_{i}\in\Pi\setminus S$ for $1 \le i \le r-1$.
Then  $P := \set{\lm_{0},\lm_{1},\ldots,\lm_{r}}$ is a non-empty, proper and connected subset of $\tPi$, so the diagram for $P$
is the  diagram of an irreducible reduced finite root system \cite[Prop.~4.7(c)]{Kac}.  Hence
$\mu_P := \sum_{i=0}^r\lm_i \in \Sigma$.
So
\[-\mu_{P}=\mu^+-\lm_{1}-\cdots-\lm_{r}\in\Sigma^{+}.\]
Since
$\lm_{i}\notin T$ for $1\le i \le r$, we have
$\chi_T(-\mu_{P})=\chi_T(\mu^+)=2$.
Also, since $\lm_i\not\in S$ for $1\le i \le r-1$, we have
$\chi_S(-\mu_{P})=\chi_S(\mu^+)-\chi_S(\lm_r) = 2-1 = 1$.  Thus $(S,T)\notin \SPA(\Pi)$.

Conversely, suppose that $(S,T)\notin \SPA(\Pi)$.  Then, by Lemma \ref{lem:SPadmit},
there exists $\mu\in
\Sigma^{+}$ with $(\chi_S(\mu),\chi_T(\mu))=(1,2)$. Let $S(\mu) = \{\mu^-\}\cup \supp_{\Pi}\{\mu^+-\mu\}$.
Now $\chi_T(\mu^+-\mu) = 2 -2 = 0$, so
$T\cap S(\mu)=\emptyset$.  Also,
$\chi_S(\mu^+-\mu) = 2 -1 = 1$, so
$S\cap S(\mu)\neq\emptyset$.  Since $S(\mu)$ is connected
by Lemma \ref{lem:connected}, $S(\mu)\subseteq \comp(\tPi\setminus T,\mu^-)$, so
$\comp(\tPi\setminus T,\mu^-) \cap S \ne \emptyset$.
\end{proof}

It follows from Propositions \ref{prop:path}(c) and \ref{prop:Kadmit2}(c) that if $(S,T)\in \SPA(\Pi)$,
then $S$ is Kantor-admissible and $T$ is either empty or Kantor-admissible.

We define a right action
of $\Aut(\Pi)$ on the set of pairs of subsets of $\Pi$ by
\[(S,T) \cdot \ph := (S\cdot \ph, T\cdot \ph) = (\ph^{-1}(S),\ph^{-1}(T)).\]

Now, by \eqref{eq:philmS}, we have
\begin{equation}
\label{eq:philmST}
\chiST \cdot \ph=     \chi_{(S, T)\cdot \ph}
\end{equation}
for $S,T\subseteq \Pi$, $\ph\in \Aut(\Pi)$.
Using this  and the fact that $\Hom_\sh(\Sigma,\BCtwo)$ is stabilized
by the right action of $\Aut(\Pi)$, we see  that the $\SPA(\Pi)$
is stabilized by the right action of $\Aut(\Pi)$.  So \emph{we have a right
action $\cdot$ of $\Aut(\Pi)$ on $\SPA(\Pi)$}.

\subsection{Classification}
\label{subsec:SPfd}

\begin{construction} (The SP-graded Kantor pair $\fP(\Pi;S,T)$) \label{con:SP}
Suppose$(S,T) \in \SPA(\Pi)$,  in which case $\chiST \in \Hom_\sh(\Sigma,\BCtwo)$.
Then,  by Lemma \ref{lem:rootgrG},
$\cG(\chiST)$ is a $\BCtwo$-graded Lie algebra with $\cG(\chiST)_\al \ne 0$
for some $\al \in \De_\sh$. Let $\fP(\Pi;S,T)$ be the SP-graded Kantor pair
enveloped by $\cG(\chiST)$  (see Section \ref{subsec:BC2}).
Note that  $\fP(\Pi;S,T)$ is nonzero by \eqref{eq:Psumshort}.
Explicitly
$\fP(\Pi;S,T)$ is the Kantor pair $\fP(\Pi;S)$ (see Construction \ref{con:Kantor})
with the SP-grading given by
\begin{equation} \label{eq:fPSP}
\textstyle \fP(\Pi;S)_{\sg i}^\sg = \cG(\chi_{(S,T)})_{\sg 1, \sg i} = \sum_{\mu \in \QS,\ \chi_S(\mu) = \sg 1,\ \chi_T(\mu) = \sg i} \cG_\mu
\end{equation}
for $\sg = \pm$ and $i=0,1$. We call this grading the  \emph{SP-grading of $\fP(\Pi;S)$ determined by~$T$}.
\end{construction}

\begin{remark} \label{rem:PSTwelldef}
Just as in Remark \ref{rem:PSwelldef}, $\fP(\Pi;S,T)$ is  well-defined up to graded-isomorphism.
\end{remark}

Again since $\Pi$ is  fixed in our discussion
we will usually \emph{write $\fP(\Pi;S,T)$  as $\fP(S,T)$}

\begin{remark} \label{rem:SP} Let $(S,T)\in \SPA(\Pi)$.

(i) By  Proposition \ref{prop:BC2},  $\cG(\chiST)$  is
graded-isomorphic to $\Kan(\fP(S,T))$ with its standard $\BCtwo$-grading.

(ii) If $i=0,1$, $\fP(S,T)_i$ has balanced dimension by Proposition \ref{prop:Pop=Pgr}.  Moreover,
this balanced dimension equals  $\dim(\cG(\chiST)_{1,i})$ by (i), which equals
\[|\set{\mu\in \Sigma : \chi_S(\mu) = 1,\ \chi_T(\mu) = i}|. \]
\end{remark}

We have the following classification of  simple
SP-graded Kantor pairs of type~$\type{X}{\rkn}$.

\begin{theorem}  \label{thm:SP} Suppose that $\Pi$ is the  connected Dynkin diagram of type $\Xn$.
\begin{itemize}
\item [(i)]  If $(S,T)\in\SPA(\Pi)$, then $\fP(S,T)$ is a finite dimensional  simple SP-graded Kantor pair of type
$X_\rkn$.
\item[(ii)]  If $P$ is a  finite dimensional simple SP-graded Kantor pair of type
$\type{X}{\rkn}$, then $P$ is graded-isomorphic to
$\fP(S,T)$ for some $(S,T)\in \SPA(\Pi)$.
\item[(iii)]  If $(S,T),(S',T')\in\SPA(\Pi)$, then the Kantor pairs
$\fP(S,T)$ and  $\fP(S',T')$ are graded-isomorphic if and only $(S,T)$ and $(S',T')$
are in the same  orbit in $\SPA(\Pi)$ under the right action  of $\Aut(\Pi)$.
\end{itemize}
\end{theorem}

\begin{proof}  (i): This follows from  Proposition \ref{prop:fdsimp}.

(ii): $\Kan(P)$ is simple of type $\type{X}{\rkn}$, so there is an isomorphism
$\eta: \Kan(P) \to \cG$.  We use $\eta$ to transport the standard  $\BCtwo$-grading of $\Kan(P)$ to $\cG$, so $\eta$ is a $\BCtwo$-graded
isomorphism.  Next by Theorem \ref{thm:rootgrG}(i) we can assume that
the $\BCtwo$-grading of $\cG$ is the $\rho$-grading for some
$\rho\in \Hom_\pos(\Sigma,\BCtwo)$.
Now $\Kan(P)_\al \ne 0$ for some $\al\in \De_\sh$
by \eqref{eq:Psumshort}, and thus\ $\cG(\rho)_\al \ne 0$ for some $\al\in \De_\sh$.
So by Lemma \ref{lem:rootgrG}(ii), $\rho \in \Hom_\psh(\Sigma,\BCtwo)$, and hence by Proposition \ref{prop:SPadmit},
$\rho = \chiST$ for some $(S,T)\in \SPA(\Pi)$.
Thus, under the restriction of $\eta$, $P \simgr\fP(S,T)$.

(iii):  Let $P = \fP(S,T)$ and $P' = \fP(S',T')$; and let $\rho = \chiST$ and
$\rho' =\chi_{(S',T')}$ in $\Hom_\psh(\Sigma,\BCtwo)$.
Then by Theorem \ref{thm:rootgrG}(ii) and \eqref{eq:philmST}, we know that
$\cG(\rho) \simgr \cG(\rho')$ if and only if
$(S,T)$ and $(S,'T')$ are in  the same orbit  in $\SPA(\Pi)$ under the right action   of $\Aut(\Pi)$.
So it remains to show that
$P\simgr P'$   if and only if
$\cG(\rho) \simgr \cG(\rho')$.
The  implication ``$\Leftarrow$'' in this statement is clear.  To
prove the converse,  suppose  that  $P\simgr P'$.
Then,  as noted in Subsection \ref{subsec:BC2}, $\Kan(P) \simgr \Kan(P')$; so, by Remark \ref{rem:SP}(i),
$\cG(\rho) \simgr \cG(\rho')$.
\end{proof}

\begin{remark} \label{rem:ungradedproof}
We note  that if we take $T=\emptyset$ everywhere in Theorem \ref{thm:SP}, we obtain
the classification  Theorem \ref{thm:Kantor}.
\end{remark}

As a corollary, we obtain the following
classification up to  isomorphism of the SP-gradings
on a fixed  simple Kantor pair.

\begin{corollary} \label{cor:SPPS} Suppose $S\in \KA(\Pi)$.
Any SP-grading on $\fP(S)$ is \griso{} to the SP-grading determined
by $T$ for some subset $T$ of $\Pi$ such that $(S,T)\in \SPA(\Pi)$.
Also, for two such subsets $T$ and $T'$ of $\Pi$, the SP-gradings
of $\fP(S)$ determined by $T$ and $T'$ are \griso{} if and only if  $(S,T)$
and $(S,T')$ are in the same orbit  in $\SPA(\Pi)$ under the  right action of $\Aut(\Pi)$.
\end{corollary}

If $S\in \KA(\Pi)$, then, by
Proposition \ref{prop:path}(c),
$(S,\emptyset)$ and
$(S,S)$ are in $\SPA(\Pi)$.  Moreover,  the SP-gradings of $\fP(S)$ determined by
$\emptyset$ and $S$ are respectively  the zero SP-grading and the one SP-grading of $\fP(S)$.
Of course \emph{our main interest is in non-trivial  gradings of $\fP(S)$, which occur in Corollary \ref{cor:SPPS} when $T$ is not equal to $ \emptyset$ or $S$}.

\subsection{The close-to-Jordan case} \label{closeSP}

The classification of   SP-gradings has a particularly
simple description for  close-to-Jordan pairs.

\begin{theorem}
\label{thm:closeSP}  Suppose  $S\in \KA(\Pi)$  and $P = \fP(S)$
is close-to-Jordan. If $T \subseteq \Pi$, then
$(S,T)\in \SPA(\Pi)$ if and only if
\begin{equation} \label{eq:Tchar} T = \emptyset,\ T= S \text{ or } T =  \set{\lambda} \text{ for some } \lambda\in \Pi \text{ with } \chi_\lambda(\mu^+) = 1.
\end{equation}
Hence the SP-gradings of $\fP(S)$ are, up to isomorphism, precisely the SP-gradings determined by subsets of $\Pi$ of the form
\eqref{eq:Tchar}. Finally,
if $T$ and $T'$ are subsets of $\Pi$ of the form \eqref{eq:Tchar}, then $T$ and $T'$ determine isomorphic SP-gradings
of $\fP(S)$ if and only if there exists $\varphi\in \Aut(\Pi)$ such that $T' = T \cdot \varphi$.
\end{theorem}

\begin{proof}  Recall that   $X_\rkn \ne A_1$, $S$ is the set of nodes of $\Pi$ that are adjacent to $\mu^-$ in $\tPi$ and
$\chi_S(\mu^+) = 2$ (see Theorem \ref{thm:KanSko}).
If $T$ of the  form \eqref{eq:Tchar}, then  $(S,T)\in \SPA(\Pi)$ by Proposition   \ref{prop:path}(c).  For the converse,
suppose $T \subseteq \Pi$ with $(S,T)\in \SPA(\Pi)$, and suppose
$T \ne \emptyset$ and $T \ne S$.  Then it suffices to show that  $\chi_T(\mu^+) = 1$, so we suppose the
contrary. Hence
$\mu_T(\mu^+) = 2$.
But if $\nu \in S\setminus T$, then $\mu^-,\nu$ is a path in $\tPi\setminus T$,
so  $\nu \in \comp(\tPi,\mu^-)\cap S$, which is empty by Proposition \ref{prop:path}(c).  Therefore $S \subseteq T$.
So since $\chi_T(\mu^+) = 2 = \chi_S(\mu^+)$,
we have  $T = S$.  With this contradiction we have proved the first statement.
The second and third statements now follow by Corollary \ref{cor:SPPS} and the fact that each  $\varphi\in \Aut(\Pi)$  fixes~$\mu^-$.
\end{proof}

Looking at the explicit expression for $\mu^+$ for each type, we see the following:

\begin{corollary} \label{cor:noGFE}  Suppose that $P$ is the simple close-to-Jordan Kantor pair
of type $X_\rkn \ne A_1$.  If  $X_\rkn = \type{G}{2}$, $\type{F}{4}$ or $\type{E}{8}$, then
$P$ does not have a non-trivial SP-grading.  For the other types,
the number of  isomorphism classes of non-trivial SP-gradings of $P$ is $\lfloor \frac {n+1}2 \rfloor$
if $X_n = \type{A}{n} (n \ge 2)$, $2$ if
$X_n = \type{D}{n} (n \ge 5)$  and $1$ if $X_n= \type{B}{n} (n \ge 2)$,
$\type{C}{n} (n \ge 3)$, $\type{D}{4}$, $\type{E}{6}$ or~$\type{E}{7}$.
\end{corollary}

See Example \ref{ex:E6SP}(ii) below for an  illustration of Theorem \ref{thm:closeSP}
and Corollary \ref{cor:noGFE} in type $\type{E}{6}$.

\subsection{Marked Dynkin diagrams for SP-gradings}
\label{subsec:markedSP}
\emph{We represent an element $(S,T)\in \SPA(\Pi)$
by the Dynkin diagram for $\Pi$  with the nodes in $S$ marked with a circle
and the nodes in $T$ marked with an asterisk; and we use
the same marked Dynkin diagram
to represent the SP-graded Kantor pair $\fP(S,T)$  as well as
the SP-grading of $\fP(S)$ determined by $T$}.

\begin{example}(Type $\type{E}{6}$)  \label{ex:E6SP}
Suppose  that $\Pi = \set{\mu_1,\dots,\mu_6}$ is of type $\type{E}{6}$ with  Dynkin
diagram as in Example \ref{ex:E6KP}.
Recall  that in Example \ref{ex:E6KP} we saw that
(up to the right action of $\Aut(\Pi)$)  there are four Kantor-admissible subsets $S$ of~$\Pi$:
(i) $\set{\mu_1}$, (ii) $\set{\mu_6}$, (iii) $\set{\mu_2}$, and (iv) $\set{\mu_1,\mu_5}$.
We  now use
Proposition \ref{prop:path}(c) and the extended diagram \eqref{eq:E6extended} for $\type{E}{6}$ to list for each choice of $S$ the marked diagrams that
represent,  up to the right action of $\Aut(\Pi)$, the
SP-admissible pairs of the form $(S,T)$ with $T\ne \emptyset$ and~$T \ne S$:
\allowdisplaybreaks
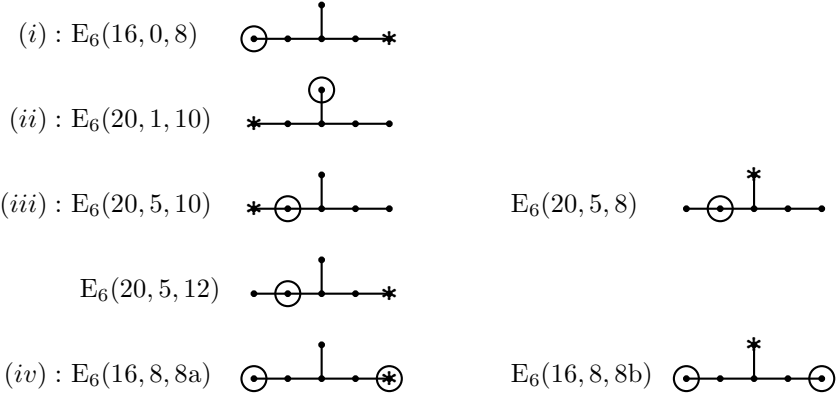
\begin{alignat*}{10}
(i)&: &&\ \type{E}{6}(16,0,8)\  && \clearlabels \Satrue \Tetrue \Esix \\
(ii)&: &&\ \type{E}{6}(20,1,10)\ && \clearlabels \Sftrue \Tatrue \Esix \\
(iii)&: &&\ \type{E}{6}(20,5,10)\  && \clearlabels \Sbtrue \Tatrue \Esix  \qquad
	&&  \type{E}{6}(20,5,8)\ &&\clearlabels \Sbtrue \Tftrue \Esix\\
&&&\ \ \type{E}{6}(20,5,12)\   && \clearlabels \Sbtrue \Tetrue \Esix \\
(iv)&:  &&\ \type{E}{6}(16,8,8a)\ &&  \clearlabels \Satrue \Setrue \Tetrue \Esix
           &&\type{E}{6}(16,8,8b)\   && \clearlabels \Satrue \Setrue \Tftrue \Esix
\end{alignat*}
Hence, for each~$S$,  Corollary \ref{cor:SPPS} tells us that the non-trivial SP-gradings
on $\fP(S)$ are represented up to  isomorphism by the listed marked
diagrams.
Thus, up to graded-isomorphism, each simple
SP-graded Kantor pair of type $\type{E}{6}$ with non-trivial grading is
represented by exactly one of the above   seven marked diagrams.

Note that each marked diagram representing
$(S,T)$ and $\fP(S,T)$  in (i), (ii) and (iii) is labelled as $\type{E}{6}(d,e,f)$, where $d$ is the balanced dimension of $\fP(S,T)$ and $e$ is the balanced $2$-dimension of $\fP(S,T)$
as in Example \ref{ex:E6KP}, and where $f$ is the balanced dimension of  $\fP(S,T)_1$.
($f$   was computed using Remark   \ref{rem:SP}(ii) and a list of
roots in $\Sigma$ \cite[Plate V]{Bour}.)
In (iv), we have used the notations
$\type{E}{6}(16,8,8a)$ and $\type{E}{6}(16,8,8b)$  because there are two
graded pairs with parameters 16,8,8  that are not  graded-isomorphic.
In any of the cases, we will sometimes use the label for the SP-graded Kantor pair $\fP(S,T)$ itself.
\end{example}

For each of the other types $\type{X}{\rkn}$, it is not difficult using the same
method to write down marked Dynkin diagrams representing up to graded-isomorphism the  simple
SP-graded Kantor pairs of type $\Xn$ with non-trivial gradings.
We leave this  to the interested reader as an   exercise.

\section{Weyl images of finite dimensional simple SP-graded Kantor pairs}
\label{sec:Weylcompute}

\emph{We  continue  with the assumptions  and notation of  Section \ref{sec:SPfd}.}
In this section (see Theorems \ref{thm:uPST} and \ref{thm:uST}), we compute the marked Dynkin diagram that represents the SP-graded Kantor
pair  $\p{u}\fP(S,T)$ for each $(S,T) \in \SPA(\Pi)$ and each $u\in \WD$.
In view of Theorem \ref{thm:SP}, this  computes
all  Weyl images of all   finite dimensional simple SP-graded Kantor pairs.

\subsection{The maps \protect $w_X$, $\sg_X$ and $\tilde w_X$ \protect}
\label{subsec:longWeyl}
Suppose that $X\subseteq \Pi$.

Let
\[E_{X} = \spann_\bbR(X),\quad \Sigma_{X} = \Sigma \cap E_X  \andd \Sigma^+_{X} = \Sigma^+ \cap E_X.\]
If $X \ne \emptyset$, then $\Sigma_X$ is a root system with base $X$ in the Euclidean space $E_{X}$,
and $\Sigma^+_{X}$ is the set of positive roots in $\Sigma_X$ relative to $X$.  We then use
the simplified notation
\[Q_X = Q_{\Sigma_X} \andd W_X = W_{\Sigma_X}\]
respectively for the root lattice and Weyl group of $\Sigma_X$.
If $X = \emptyset$, then  $E_{X} = \set{0}$,  $\Sigma_X = \emptyset$, $\Sigma^+_{X} = \emptyset$, and we set
 $Q_X = \set{0}$ and $W_X = \set{1}$.

If $X\ne \emptyset$, we let $w_X$ be the unique element  of $W_X$ such that
\[w_X(\Sigma_X^{+})=-\Sigma_X^{+}.\]
Equivalently, $w_X$ is the longest element of
$W_X$ relative to the base $X$
\cite[VI, \S1.6, Cor.~3]{Bour}.  Clearly $-w_X$ stabilizes $X$ and
we set
\[\sg_X = -w_X \in \Aut(X),\]
where recall that $\Aut(X)$ is the group of diagram automorphisms of $X$.
If $X  = \emptyset$, we adopt the conventions that  $w_X = 1$ and $\sg_X = 1$.
Evidently in all cases $w_X^2 = 1$ and $\sg_X^2 = 1$.

If $X\ne \emptyset,$ it is well known that the map $\sg_X$ can be read from the Dynkin diagram for $X$.
Indeed, if $X$ is connected, then  $\sg_X$ is the
nontrivial diagram automorphism of $X$ if $X$ has type
$\type{A}{\rkn}$ with $\rkn \geq 2$, $\type{D}{\rkn}$ with $\rkn$ odd $\ge 5$, or $\type{E}{6}$; and $\sg_X=1$
for all other  types
\cite[VI, \S4.5--\S4.13]{Bour}.  Furthermore,
if $X$ has  connected components $X_{i}$, $1\leq
i\leq r$, then $\sg_X$ stabilizes each $X_i$ and $\sg_X|_{X_i} = \sg_{X_i}$.

If $w\in W_X$, then  $w$ extends uniquely to an isometry $\tilde{w}$ of $E$ with $\tilde{w}(\tau)=\tau$
for $(\tau,E_{X})=0$.  It is clear that the map $\ph \mapsto \tilde\ph$ is a monomorphism
of $W_X$ into $W_\Sigma$.

We will abuse notation and write $\widetilde{w_X}$ as $\tilde w_X$.  Note that if
$\lm \in E$, we have
\begin{equation}
\label{eq:wX}
\tilde w_X(\lm) \in \lm + Q_X.
\end{equation}

\subsection{The left action \protect $*$ of $W_\Delta$ on $\SPA(\Pi)$ \protect}
\label{subsec:*action}
Recall from Section  \ref{subsec:*} that we have left action $*$ of $W_\Delta = \Aut(\Delta)$ on  $\Hom_\psh(\Sigma,\BCtwo)$.
We use the bijection $(S,T) \mapsto \chiST$ in
Proposition \ref{prop:SPadmit} to transfer this action to a \emph{left action
$*$ of $W_\Delta$ on $\SPA(\Pi)$}.  So by definition we have for $u\in W_\Delta$ and $(S,T)\in \SPA(\Pi)$ that
\begin{equation} \label{eq:left2}
\chi_{u*(S,T)} = u*\chiST. \end{equation}
Using
\eqref{eq:philmST} and \eqref{eq:left2}, we see that the left action  $*$ of $W_\Delta$ on $\SPA(\Pi)$
commutes with the right action $\cdot$ of $\Aut(\Pi)$ on the same set.

\subsection{Weyl images of \protect $\fP(S,T)$ \protect}
\label{subsec:WPST}

\begin{theorem} \label{thm:uPST}
Suppose that $(S,T)\in \SPA(\Pi)$.  Then
$\p{u}{\fP(S,T)}\simgr \fP(u*(S,T))$ for $u\in W_\Delta$.
\end{theorem}

\begin{proof} By definition of $*$, we have $u*\chiST = u \cdot \chiST \cdot w$
for some $w\in W_\Sigma$.  Then
\begin{align*} \p{u}\cG(\chiST) &= \cG(u\cdot \chiST) &&\text{by \eqref{eq:left1}}\\
&\simgr \cG(u\cdot \chiST\cdot w) &&\text{by Proposition \ref{prop:Zngrading}(ii)}\\
&= \cG(u* \chiST) = \cG( \chi_{u*(S,T)}) &&\text{by \eqref{eq:left2}}.
\end{align*}
 Since
$\p{u}\cG(\chiST)$ and $\cG( \chi_{u*(S,T)})$ determine
$\p{u}{\fP(S,T)}$ and $\fP(u*(S,T))$ respectively, we have our conclusion.
\end{proof}

Because of Theorem \ref{thm:uPST}. it is important to be able to compute
$u*(S,T)$ for $u\in W_\Delta$, $(S,T)\in \SPA(\Pi)$.  We do this in the next theorem, where
there is a case that requires special treatment, namely the case when:
\begin{equation}
\tag{ex}
\parbox{3.5in}{
$\Pi$ has type $\type{A}{\rkn}$, $\rkn \ge 3$, and there
exists distinct $\lm,\lm',\mu\in \Pi$  such that $S = \set{\lm,\lm'}$,
$T = \set{\mu}$, and $\lm'$ lies between $\lm$ and $\mu$ on the Dynkin diagram
for $\Pi$.
}
\end{equation}

Recall from Section \ref{subsec:SPWeyl}  that $W_\De = \set{1, s_1, s_2,  s_2s_1, -1,  -s_1, -s_2,  -s_2s_1}$.

\begin{theorem}
\label{thm:uST}  Suppose $(S,T)\in \SPA(\Pi)$.  Then  the elements $u*(S,T)$ for $u\in W_\Delta$
are given in the following table
\begin{equation}
\label{tab:uST}
\renewcommand{\arraystretch}{1.3}
\begin{tabular}[ht]
{c |   c }
\small  $u$ & $u*(S,T)$\\
\whline \rule{0ex}{3ex} $1$ & $(S,T)$\\
\cline{1-2} $s_1$ & $(\bS,T)$ \\
\cline{1-2} $s_2$ & $(S,\bar T)$\\
\cline{1-2} $s_2s_1$ & $(\bS,\bar T \cdot \sg_{\Pi})$\\
\cline{1-2} $-1$ & $(S\cdot \sg_{\Pi},T\cdot \sg_{\Pi})$\\
\cline{1-2} $-s_1$ & $(\bS\cdot \sg_{\Pi},T\cdot \sg_{\Pi})$\\
\cline{1-2} $-s_2$ & $(S\cdot \sg_{\Pi},\bar T\cdot \sg_{\Pi})$\\
\cline{1-2} $-s_2s_1$ & $(\bS\cdot \sg_{\Pi},\bar T )$
\end{tabular}\quad,
\end{equation}
where
\begin{equation} \label{eq:breveS}
\breve S =  \sg_{\Pi\sm T}(S\sm T) \ \bigcup \
\set{\lm \in T \suchthat \chi_{S\sm T}(\wPT(\lm)) + \chi_{S\cap T}(\lm) = 1}
\end{equation}
and
\begin{equation} \label{eq:shiftT}
\bar T
=  \sg_{\Pi\sm S}(T\sm S) \ \bigcup \
\set{\lm \in S\sm T \suchthat \supp_{\Pi}(\wPS(\lm)) \cap (T\sm S) = \emptyset}.
\end{equation}
(See Remark \ref{rem:brevebar} below about the notation  $\bS$ and $\bar T$.)
Moreover, we have the following expressions, which allow us to read $\breve S$ and $\bar T$
directly from the marked Dynkin diagram representing $(S,T)$ together with the expression
for the highest root $\mu^+$ in~$\Sigma$:
\begin{equation} \label{eq:breveSmore}
\bS :=
\begin{cases}
	S&\text{if } \chi_{S\sm T}(\mu^+) =0,\\
	\sg_{\Pi\sm T}(S\sm T) \cup(T\sm S)&\text{if } \chi_{S\sm T}(\mu^+) = 1,\\
	\sg_{\Pi\sm T}(S\sm T)&\text{if } \chi_{S\sm T}(\mu^+) = 2,
\end{cases}
\end{equation}
and
\begin{equation} \label{eq:shiftTmore}
\bar T :=
\begin{cases}
	S\sm T&\text{if } T\subseteq S,\\
	\set{\sg_{\Pi\sm S}(\mu),\lm}&\text{if \emph{(ex)} holds},\\
	\sg_{\Pi\sm S}(T\sm S)&\text{otherwise}.
\end{cases}
\end{equation}
\end{theorem}

\begin{proof}  We postpone the proofs of \eqref{eq:breveSmore} and
\eqref{eq:shiftTmore} until  Section \ref{subsec:bSbT}, since these proofs require some additional information about the Weyl group
elements $\wPS$ and $\wPT$.  We  prove the first statement here.

If $\lm\in \Pi$, we have
\[((-1)\cdot \chiST \cdot w_{\Pi})(\lm) = -(\chi_S(w_{\Pi}(\lm)),\chi_T(w_{\Pi}(\lm))) =
 (\chi_S(\sg_{\Pi}(\lm)),\chi_T(\sg_{\Pi}(\lm)),\]
which has non-negative entries.  Hence,
$(-1)*\chiST = (-1)\cdot\chiST\cdot w_{\Pi} = \chiST \cdot \sg_{\Pi} = \chi_{(S,T)\cdot \sg_{\Pi}}$
using \eqref{eq:philmST}.  So using \eqref{eq:left2},
$(-1)*(S,T) = (S,T)\cdot \sg_{\Pi} = (S\cdot \sg_{\Pi}, T\cdot \sg_{\Pi})$.
This establishes row 5 (not counting the header row) of \eqref{tab:uST}, and we see that rows  6, 7 and 8
follow respectively from rows 2,  3 and 4.  Thus it is sufficient to establish rows 2, 3 and 4.

\emph{Row 2:}  Let $\rho = (\rho_1,\rho_2) = s_1 \cdot \chiST \cdot \wPT\in \Hom_\sh(\Sigma,\BCtwo)$.
Recall that we are identifying $Q_\De = \bbZ^2$ using the $\bbZ$-basis $\Pd = \set{\al_1,\al_2}$ for $Q_\De$.
So
\begin{align*}
s_1 \cdot \chiST(\lm) &= s_1(\chi_S(\lm) \al_1 + \chi_T(\lm) \al_2) =
\chi_S(\lm) (-\al_1) + \chi_T(\lm) (2\al_1 + \al_2)\\
&=(2\chi_T(\lm) - \chi_S(\lm),\chi_T(\lm)).
\end{align*}
for $\lm\in \Sigma$.  Also,  if $\lm \in \Pi$, we have
$\chi_T(\wPT(\lm))  = \chi_T(\lm)$ by \eqref{eq:wX}.
  Thus
\begin{equation} \label{eq:rho12a}
\rho_1(\lm) = 2\chi_T(\lm) - \chi_S(\wPT(\lm)) \andd \rho_2(\lm) = \chi_T(\lm)
\end{equation}
for $\lm\in \Pi$.

We check next using \eqref{eq:rho12a}
that $\rho$ is positive; that is $\rho_1(\lm) \ge 0$ and $\rho_2(\lm) \ge 0$ for $\lm\in \Pi$.
First if $\lm \in T$, we have
$\rho_2 (\lm) =   1$.
But since $\rho\in \Hom(\Sigma,\BCtwo)$, we have $\rho_1(\mu)\rho_2(\mu) \ge 0$ for
$\mu\in \Sigma$.  So $\rho_1(\lm) \ge 0$ for $\lm\in T$.
Next, if $\lm \in \Pi\sm T$, we have $\rho_2(\lm) = 0$ and
$\rho_1(\lm) =  - \chi_S(\wPT(\lm))  =  \chi_S(\sg_{\Pi\sm T}(\lm)) \ge 0$.
So $\rho$ is positive as desired.

Consequently we have $s_1 * \chiST = \rho$, so $\chi_{s_1*(S,T)} = \rho$.
Thus  $s_1*(S,T) =  (\breve S, \breve T)$, where
\begin{equation}\label{eq:rho12b}
\breve S = \set{\lm\in \Pi \suchthat \rho_1(\lm) = 1} \andd
\breve T = \set{\lm\in \Pi \suchthat \rho_2(\lm) = 1}.
\end{equation}
It follows now from \eqref{eq:rho12a} that $\breve T = T$, so it remains establish
\eqref{eq:breveS}.

Now   $\breve S = (\breve S \sm T) \cup (\breve S \cap T)$.
Moreover, using \eqref{eq:rho12a} and \eqref{eq:rho12b},
\[\breve S \sm T =  \set{\lm \in \Pi\sm T \suchthat \chi_S(\sg_{\Pi\sm T} (\lm)) = 1}  = \sg_{\Pi\sm T} (S\sm T)\]
and
\[\breve S \cap T =  \set{\lm \in T \suchthat 2 - \chi_S(\wPT (\lm)) = 1} =  \set{\lm \in T \suchthat \chi_S(\wPT (\lm)) = 1}.\]
But if  $\lm \in T$, then
\[\chi_S(\wPT (\lm)) = \chi_{S\sm T}(\wPT (\lm)) + \chi_{S\cap T}(\wPT (\lm)) =   \chi_{S\sm T}(\wPT (\lm)) + \chi_{S\cap T}(\lm)\]
using   \eqref{eq:wX}.  So we have \eqref{eq:breveS}.

\emph{Row 3:}  Let $\tau = (\tau_1,\tau_2) = s_2 \cdot \chiST \cdot \wPS\in \Hom_\sh(\Sigma,\BCtwo)$.
Now
\begin{align*}
s_2 \cdot \chiST(\lm) &= s_2(\chi_S(\lm) \al_1 + \chi_T(\lm) \al_2) =
\chi_S(\lm) (\al_1+\al_2) + \chi_T(\lm) (-\al_2)\\
&=(\chi_S(\lm),\chi_S(\lm)-\chi_T(\lm)).
\end{align*}
for $\lm\in \Sigma$. Also if $\lm \in \Pi$, $\chi_S(\wPS(\lm))  = \chi_S(\lm)$  by \eqref{eq:wX}.
Thus
\begin{equation} \label{eq:rho12c}
\tau_1(\lm) = \chi_S(\lm) \andd \tau_2(\lm) = \chi_S(\lm)-\chi_T(\wPS(\lm))
\end{equation}
for $\lm\in \Pi$.  Arguing as in Row 2 we now easily see that $\tau$ is positive, $\chi_{s_2*(S,T)} = \tau$ and
$s_2*(S,T) =  (\bar S, \bar T)$, where
\begin{equation*}
\bar S = \set{\lm\in \Pi \suchthat \tau_1(\lm) = 1} \andd
\bar T = \set{\lm\in \Pi \suchthat \tau_2(\lm) = 1}.
\end{equation*}
It follows then from \eqref{eq:rho12c} that $\bar S = S$, so it remains to prove \eqref{eq:shiftT}.
Again arguing as in Row 2, we  easily see that
\begin{equation*}
\bar T =  \sg_{\Pi\sm S}(T\sm S) \ \bigcup \
\set{\lm \in S \suchthat \chi_{T\sm S}(\wPS(\lm)) + \chi_{T\cap S}(\lm) = 0}.
\end{equation*}
Finally, let $\lm \in S$.  Then, since $\wPS(\lm) \in \lm + Q_{\Pi\sm S}$, we see that
$\wPS(\lm)$ is positive and hence $\chi_{T\sm S}(\wPS(\lm))\ge 0$.
Thus $\chi_{T\sm S}(\wPS(\lm)) + \chi_{T\cap S}(\lm) = 0$ if and only if
$\lm \in S\sm T$ and $\chi_{T\sm S}(\wPS(\lm)) = 0$, which holds if and only if
$\lm \in S\sm T$ and $\supp_{\Pi}(\wPS(\lm)) \cap (T\sm S) = \emptyset$.

\emph{Row 4:} Using Row 2 and Row 3 (applied to $(\breve S, T)$), we have
\[(s_2s_1)*(S,T) =s_2* (s_1*(S,T)) = s_2*(\breve S, T) = (\breve S, T')\] for some subset $T'$ of $\Pi$.
On the other hand, $s_2 s_1=  - s_1 s_2$.  So, using Row 3, Row 2 (applied to $(S,\bar T)$) and Row 5 (applied to  $(S',\bar T)$), we have
\begin{align*}
(s_2s_1)*(S,T) &= (-1)*(s_1* (s_2*(S,T))) \\
&= (-1)*(s_1* (S,\bar T)) = (-1)*( S',\bar T) =
( S'\cdot \sg_{\Pi},\bar T\cdot \sg_{\Pi})
\end{align*}
for some subset $S'$ of $\Pi$.  Combining these equalities gives our conclusion.
\end{proof}

If $(S,T)\in \SPA(\Pi)$, then we see using  Theorems \ref{thm:uPST} and \ref{thm:uST}  that
\[\fP(S,T)^\op = \p{-1}{\fP(S,T)}\simgr \fP((-1)*(S,T)) = \fP(S\cdot \sg_{\Pi},T\cdot \sg_\Pi) \simgr \fP(S,T).\]
This together with Theorem \ref{thm:SP}(ii) gives another proof  of
Proposition \ref{prop:Pop=Pgr}.

The following corollary, which we state for emphasis and convenience of reference, follows taking
$u= s_1$ and $u=s_2$ in Theorems \ref{thm:uPST} and   \ref{thm:uST}.

\begin{corollary} \label{cor:uST2}
If $(S,T)\in \SPA(\Pi)$ and $\bS$ and $\bar T$ are given by \eqref{eq:breveSmore} and \eqref{eq:shiftTmore}
(or  \eqref{eq:breveS} and \eqref{eq:shiftT}),  then
\[ \fP(S,T)\brv \simgr \fP(\bS,T) \andd \overline{\fP(S,T)} \ {\simgr}\ \fP(S,\bar T).\]
\end{corollary}

\begin{remark}\label{rem:brevebar} Corollary \ref{cor:uST2} explains
our choice of notation for $\bS$ and $\bar T$.  It should be noted however that this is a
(convenient) abuse of notation, since $\bS$ and $\bar T$ each depend on  \emph{both} $S$ and~$T$.
\end{remark}

\begin{example}[The trivial SP-gradings] 
Suppose that $S\in \KA(\Pi)$.  Recall from  Section
\ref{subsec:SPfd}
that the zero and one SP-gradings on $\fP(S)$ are determined by $\emptyset$ and $S$ respectively.
For these SP-gradings one can check easily using Corollary
\ref{cor:uST2}, \eqref{eq:breveSmore} and \eqref{eq:shiftTmore} that
\[\fP(S,\emptyset)\brv \simgr \fP(S,\emptyset), \qquad \fP(S,S)\brv \simgr \fP(S,S),\]
and
\[\overline{\fP(S,\emptyset)} \simgr \fP(S,S), \qquad \overline{\fP(S,S)} \simgr \fP(S,\emptyset).\]
\end{example}

The  computation of Weyl images is particularly simple in the close-to-Jordan case.

\begin{corollary} \label{cor:refclose}
Suppose that  $(S,T)\in \SPA(\Pi)$ and $P = \fP(S,T)$ is close-to-Jordan with non-trivial SP-grading.  Then
\[\fP(S,T)\brv \simgr \fP(\sg_{\Pi\setminus T}(S\setminus T),T) \andd \overline{\fP(S,T)} \simgr \fP(S,T).\]
\end{corollary}

\begin{proof}  Recall  that
$X_\rkn \ne A_1$, $S$ is the set of
nodes of $\Pi$ that are adjacent to $\mu^-$ in $\tPi$, and $T = \set{\lambda}$ with $\lambda \in \Pi$ and $\chi_{\lambda}(\mu^+)=1$  (see Theorems
\ref{thm:KanSko}(ii) and  \ref{thm:closeSP}).
We assume that $\lambda$ is not an end node of
$\Pi$ if $\Xn = \type{A}{n}$, leaving the excluded case for the reader to check.
Now $\chi_S(\mu^+) = 2$ by Theorem \ref{thm:KanSko}(i).  Also $S \cap T = \emptyset$.  Indeed this holds by assumption if $\Xn = \type{A}{n}$, whereas it holds when
$\Xn \ne \type{A}{n}$ since in that case  $\card(S) = 1$ by Remark \ref{rem:mu+}(i).
Thus, we see from Corollary \ref{cor:uST2}, \eqref{eq:breveSmore} and \eqref{eq:shiftTmore}, that
$\fP(S,T)\brv \simgr \fP(\sg_{\Pi \sm T} (S \sm T),T)$ and
$\overline{\fP(S,T)} \simgr \fP(S,\sg_{\Pi \sm S} (T \sm S))$.
Finally, by uniqueness in Corollary \ref{cor:noGFE}, we have
$\overline{\fP(S,T)} \simgr \fP(S,T)$ if $\Xn \ne \type{A}{n}$ and $\Xn \ne \type{D}{n}$.
But if $\Xn = \type{A}{n}$ or $\Xn = \type{D}{n}$, $\sg_{\Pi \sm S}$ extends to an automorphism of $\Pi$, so $\overline{\fP(S,T)} \simgr \fP(S,T)$ by  Theorem  \ref{thm:closeSP}.
\end{proof}

\begin{example}(Type $\type{E}{6}$).
\label{ex:e6reflect}
Recall that in Example  \ref{ex:E6SP} we saw that there are, up to graded-isomorphism, seven
simple SP-graded Kantor pairs of type $\type{E}{6}$  whose gradings are non-trivial.
It is straightforward to apply Corollary \ref{cor:uST2},
together with \eqref{eq:breveSmore} and \eqref{eq:shiftTmore}, to calculate
the reflection and shift of each of these SP-graded Kantor pairs  up to graded-isomorphism.
One sees that
\begin{quote}
shifting exchanges $\type{E}{6}(20,5,8)$ and $\type{E}{6}(20,5,12)$;
\end{quote}
and that shifting fixes the other five graded Kantor pairs.
One also sees that
\begin{quote}
reflection exchanges $\type{E}{6}(16,0,8)$ and  $\type{E}{6}(16,8,8a)$;\\
reflection exchanges $\type{E}{6}(20,1,10)$ and  $\type{E}{6}(20,5,10)$;
\end{quote}
and that
reflection fixes $\type{E}{6}(20,5,8)$, $\type{E}{6}(20,5,12)$ and $\type{E}{6}(16,8,8b)$.
In particular, reflection does not preserve balanced $2$-dimension.  In fact, we see
that the reflection of the Jordan pair $\type{E}{6}(16,0,8)$ is not Jordan, and that
the reflection of the  close-to-Jordan pair $\type{E}{6}(20,1,10)$ is not close-to-Jordan.
\end{example}

\subsection{The proofs of  \protect \eqref{eq:breveSmore} \protect and \protect \eqref{eq:shiftTmore}  \protect }
\label{subsec:bSbT}

We now return to the proofs of \eqref{eq:breveSmore} and \eqref{eq:shiftTmore} that we
postponed earlier.  The reader  may elect to further postpone reading these arguments, as
they will not be used in the final section.

The information that we need about the Weyl group elements $\wPT$ and $\wPS$ is
provided by the next two lemmas:

\begin{lemma}
\label{lem:lmvalues2}
If  $(S,T)\in \SPA(\Pi)$, then
$\chi_{S\backslash T}(\wPT(\lm))
=\chi_{S\backslash T}(\mu^+)$ \ for $\lm\in T$.
\end{lemma}

\begin{proof}  Let $\lm\in T$.  We can  assume
\[\nu_0:=\tilde{w}_{\Pi\sm T}(\lm)\neq\mu^+.\]

Now
$\chi_T(\tilde{w}_{\Pi\sm T}(\lm)) = \chi_T(\lm)$  by \eqref{eq:wX}, so
\[\chi_T(\nu_0) = 1.\]
Hence $\nu_0\in\Sigma^+$, so there are $\lm_1,\dots,\lm_r$ in $\Pi$
with $r\ge 1$ such that
\[\nu_{i}:=\nu_{i-1}+\lm_{i}\in\Sigma^+ \ \text{ for } 1\le i \le r \andd
 \nu_r = \mu^+.\]

We next claim that $\lm_1 \in T$.  Indeed, otherwise
$\lm_1\in \Pi\setminus T = \sigma_{\Pi\sm T}(\Pi\sm T)$, so
$\lm_1 =  \sigma_{\Pi\sm T}(\nu)$ for some $\nu\in \Pi\sm T$.   Hence
\[\wPT(\lm-\nu) = \wPT(\lm) -\wPT(\nu) = \nu_0 + \lm_1 = \nu_1\in \Sigma\]
and therefore
$\lm-\nu \in \Sigma$.  This is a contradiction since $\lm\in T$ and $\nu\in \Pi\sm T$.
So we have our claim.

Next, for $1 \le p \le r$, we have
\[\chi_T(\nu_p) = \chi_T(\nu_0+\lm_1 + \dots +\lm_p)\ge
\chi_T(\nu_0)+\chi_T(\lm_1) = 1 + 1 = 2.\]
Hence  $\chi_T(\nu_p) = 2$ for $1\le p \le r$.  Thus, since $(S,T)\in \SPA(\Pi)$,
we have $\chi_S(\nu_p) \ne 1$ for $1\le p \le r$.  So, for $2\le p \le r$,
we have
$\chi_S(\lm_p) = \chi_S(\nu_p) - \chi_S(\nu_{p-1}) \in 2\bbZ$ and
hence $\lm_p\notin S$.

If follows from the previous two paragraphs that
$\lm_p\notin S\sm T$ for $1\le p \le r$.  So
$\chi_{S\sm T}(\nu_0) = \chi_{S\sm T}(\nu_r) = \chi_{S\sm T}(\mu^+)$.
\end{proof}

\begin{lemma}
\label{lem:lmvalues3}
If  $X\subseteq \Pi$ and $\lm \in \Pi\sm X$, then
$\supp_{\Pi}(\wX(\lm)) = \comp(X\cup\set{\lm},\lm)$.
\end{lemma}

\begin{proof}  Let $Y_1,\dots,Y_r$ be the connected components of $X\cup \set{\lm}$
with $\lm\in Y_1$, and let $Z_1,\dots,Z_s$ be the connected components of $Y_1\sm\set{\lm}$.
Then $Z_1,\dots,Z_s,  Y_2,\dots,Y_r$ are the connected components of $X$.  We must show that
$V = Y_1$,
where $V = \supp_{\Pi}(\wX(\lm))$.

Now it is clear that $\wX = \tilde w_{Z_1} \dots \tilde w_{Z_r} \tilde w_{Y_2} \dots \tilde w_{Y_r}$, so
$\wX(\lm) =  \tilde w_{Z_1} \dots \tilde w_{Z_r}(\lm) \in \lm + Q_{Z_1\cup \dots \cup Z_s} \subseteq Q_{Y_1}$.  Thus  $\lm \in V\subseteq Y_1$.

Next,  to show that $Y_1\subseteq V$, it suffices to show that
any $\mu\in X\cup \set{\lm}$ that is adjacent to some $\nu\in V$ lies in $V$.
For this we can assume that $\mu\ne \lm$, so $\mu\in X$.  Then
$\mu = \sg_X(\mu')$ for some $\mu'\in X$, so
$\wX(\lm) + \mu = \wX(\lm-\mu')\notin \Sigma$.
Hence $\langle \wX(\lm),\mu\rangle \ge 0$.  Since $\langle \nu,\mu \rangle < 0$,
this forces $\langle \nu',\mu \rangle > 0$ for some $\nu'\in V$.
So, since $\mu,\nu'\in \Pi$, we have $\mu = \nu'$, and hence  $\mu\in V$.
\end{proof}

\noindent\textbf{Proof of \eqref{eq:breveSmore}.}
By \eqref{eq:breveS} and Lemma \ref{lem:lmvalues2},
$\breve S =  \sg_{\Pi\sm T}(S\sm T) \cup (\breve S \cap T)$ and
\begin{equation*}
\breve S \cap T =  \set{\lm \in T \suchthat \chi_{S\sm T}(\mu^+) + \chi_{S\cap T}(\lm) = 1}
\end{equation*}
Furthermore, by Proposition \ref{prop:path}(c), $ \chi_{S}(\mu^+) \in \set{1,2}$, so
$\chi_{S\sm T}(\mu^+)\in \set{0,1,2}$.  If $\chi_{S\sm T}(\mu^+) = 0$ (or equivalently $S\subseteq T$),
then $\breve S \cap T= \set{\lm \in T \suchthat \chi_{S}(\lm) = 1} = S$. Also, if
$\chi_{S\sm T}(\mu^+) = 1$, then $\breve S \cap T= \set{\lm \in T \suchthat \chi_{S\cap T}(\lm) = 0} = T\sm S$.
Finally, if $\chi_{S\sm T}(\mu^+) = 2$, then $\breve S \cap T= \set{\lm \in T \suchthat \chi_{S\cap T}(\lm) = -1} = \emptyset$.\qed

\bigskip
\noindent\textbf{Proof of \eqref{eq:shiftTmore}.}
By \eqref{eq:shiftT} and Lemma \ref{lem:lmvalues3} (with $X = \Pi \sm S$), we have
\begin{equation} \label{eq:proofbT}
\bar T
=  \sg_{\Pi\sm S}(T\sm S) \ \bigcup \
\set{\lm \in S\sm T \suchthat \comp\big((\Pi \sm S)\cup\set{\lm}, \lm\big) \cap (T\sm S) = \emptyset}.
\end{equation}

If $T \subseteq S$, then $T\sm S = \emptyset$, so $\bar T = S\sm T$ by \eqref{eq:proofbT}.
Next if (ex) holds, then it is easy to check using \eqref{eq:proofbT} that $\bar T = \set{\sg_{\Pi\sm S}(\mu),\lm}$.
 We leave this to the reader.

Finally suppose that $T \not\subseteq S$ and (ex) does not holds.
We suppose for contradiction (which will complete the proof of \eqref{eq:shiftTmore}) that
\begin{equation} \label{eq:proofbT2}
\comp\big((\Pi \sm S)\cup\set{\lm}, \lm\big) \cap (T\sm S) = \emptyset,
\end{equation}
for some $\lm \in  S\sm T$.    If $S = \set{\lm}$, then $(\Pi \sm S)\cup\set{\lm} = \Pi$ which is connected,
so, by \eqref{eq:proofbT2}, we have $T\sm S = \emptyset$, giving a contradiction. Hence we can assume
that $S$ contains an element $\lm'\ne \lm$.  Thus, by Proposition \ref{prop:path}(c), we have
$S = \set{\lm,\lm'}$, with
\begin{equation} \label{eq:proofbT3}
\chi_\lm(\mu^+) = \chi_{\lm'}(\mu^+) = 1.
\end{equation}
Then $(\Pi \sm S)\cup\set{\lm} = \Pi\sm \set{\lm'}$ and
$T\sm S = T \sm \set{\lm'}$
since $\lm \notin T$.
So, by
\eqref{eq:proofbT2}, we have
\begin{equation}
\label{eq:proofbT4}
\comp(\Pi\sm \set{\lm'},\lm) \cap T = \emptyset.
\end{equation}
Now if $\Pi\sm \set{\lm'}$ is connected,  then $(\Pi\sm \set{\lm'}) \cap T = \emptyset$
implies that $T \subseteq \set{\lm'} \subseteq S$, a contradiction.
So
\begin{equation} \label{eq:proofbT5}
\Pi\sm \set{\lm'} \text{ has at least 2 connected components}.
\end{equation}
Now a check of $\mu^+$ for each type shows that the existence of distinct elements
$\lm, \lm' \in \Pi$ satisfying both \eqref{eq:proofbT3} and \eqref{eq:proofbT5} implies
that $\Pi$ has type $\type{A}{\rkn}$, where $\rkn \ge 3$.  Then, by \eqref{eq:proofbT4},
$\comp(\Pi\sm \set{\lm'},\lm)\cup\set{\mu^-}$ is a connected subset of
$\tPi \sm T$, so
\[\lm \in \comp(\tPi\sm T,\mu^-) \cap S.\]
Therefore, by Proposition \ref{prop:path}(c), we have $\chi_T(\mu^+) = 0$ or $1$.
But certainly $\chi_T(\mu^+) \ne 0$ since $T\ne \emptyset$.  Hence
$\chi_T(\mu^+) = 1$, so $T = \set{\mu}$ for some $\mu\in \Pi$, with $\mu \ne \lm'$ since  $T \not\subseteq S$.
It is now clear from \eqref{eq:proofbT4}   that we have (ex),
giving a contradiction.\qed

\section{Reflections of simple close-to-Jordan  pairs} \label{sec:close}
\emph{We  continue  with the assumptions  and notation of Section \ref{sec:SPfd}.}
In this section we use  results from Sections
\ref{sec:KPfd}---\ref{sec:Weylcompute}
to give a construction of the reflection of each
simple SP-graded close-to-Jordan pair   of type
$\Xn$ whose grading is   non-trivial.
We do this using the description of these graded pairs given in Theorem \ref{thm:closeSP}.

\subsection{The trilinear pair $\fT(J,\tau)$}
\label{subsec:PJtau}

In
order to construct some trilinear pairs, we fix a  pair $U = (U^-,U^+)$ of $2$-dimensional vector spaces
 and a nondegenerate bilinear map
$\formblank: U^- \times U^+\to {\bbK}$; and we set
$\formtwo(u^+,u^-) = \formtwo(u^-,u^+)$ for $u^\sg \in U^\sg$.

In order to construct gradings on our trilinear pairs, we fix bases
$\set{u_0^\sg,u_1^\sg}$ for $U^\sg$, $\sg = \pm$, such that
$\formtwo(u^-_i,u^+_j) = \delta_{ij}$
(so these bases are dual with respect to $\formblank$).

\begin{construction} \label{con:FJt} (The graded trilinear pair $\fT(J,\tau)$)
Let  $J$ be a  trilinear pair with products $\tprod_J^\sg$, $\sg = \pm$, and define
$D_J^\sg(x,a) \in \End(J^\sg)$  by  $D_J^\sg(x,a) y = \{x,a,y\}^\sg_J$.  Also, let
$\tau = (\tau^-,\tau^+)$, where $\tau^\sg: J^\sg \times J^\msg \to {\bbK}$ is a
bilinear map for $\sg = \pm$.
Then
\[\fT(J,\tau) = J\otimes U  := (J^-\otimes U^-, J^+\otimes U^+)\]
is a trilinear pair (but not in general a Kantor pair)
with products $\tprod^\sg$ given by
\begin{equation} \label{eq:FJtprod}
\{x\otimes r, a \otimes \ell, y \otimes s\}^\sg =
\{x,a,y\}_J^\sg \otimes \formtwo(r,\ell)s - \tau^\sg(x,a)y \otimes (r,\ell,s)
\end{equation}
for $x,y\in J^\sg$, $a\in J^\msg$, $r,s\in U^\sg$, $\ell\in U^\msg$, where
\[(r,\ell,s) = \formtwo(r,\ell)s - \formtwo(s,\ell)r\]
for $r,s\in U^\sg$, $\ell\in U^\msg$.
Moreover, one checks easily that
$\fT(J,\tau) =  \bigoplus_{i\in \bbZ} \fT(J,\tau)_i= \fT(J,\tau)_0 \oplus \fT(J,\tau)_1$ is a $\bbZ$-graded
trilinear pair with
\[\fT(J,\tau)_i = J^\sg \otimes u_i^\sg\]
for $i=0,1$ and $\fT(J,\tau)_i = 0$ otherwise.
\end{construction}

It is clear that the trilinear pair  $\fT(J,\tau)$ does not depend up to isomorphism
 on the choice of
$U^-$, $U^+$ or $\formblank$,
and that the grading
$\fT(J,\tau) = \fT(J,\tau)_0 \oplus \fT(J,\tau)_1$
does not depend up to isomorphism on the choice of the dual bases
$\set{u_0^\sg,u_1^\sg}$.

\begin{remark} \label{rem:Kronecker}  The  above construction
of trilinear pairs is a  pair version, with basis free products, of a construction of triple systems
given by Kantor in \cite{K1}.  More precisely suppose that $J = (X,X)$
is the double of a triple system $X$ and  $\tau^- = \tau^+ : X \times X \to \bbK$
is a bilinear map.  Then one can easily check that the pair $\fT(J,\tau)$ is
isomorphic to the double of the Kronecker product $X\otimes \bbK^2$ defined
in \cite[\S 6. Defn.~6]{K1} with $k = 2$, $\lambda = \mu = 1$ and $f(x,y) = \tau^\sg(y,x)$, $\sg = \pm$.
(It appears however that there is a typo in Kantor's definition:
the defining expression for $(YXZ)_i$ should read as
$\sum_{\al = 1}^k \left( (y_\al x_\al z_i)  + \lambda f(x_\al,y_i) z_\al - \mu f(x_\al,y_\al)z_i \right)$.)
\end{remark}

\subsection{Constructing the reflection of $\fP(\Pi;S,T)$}
\label{subsec:closeconstruct}

\emph{Suppose  for the rest of the article that $(S,T)\in \SPA(\Pi)$ and that
$\fP(\Pi;S,T)$ is close-to Jordan with nontrivial SP-grading.}
Our goal is to construct the reflection of $\fP(\Pi;S,T)$ in the form $\fT(J,\tau)$ for a specified
choice of $J$ and $\tau$.

As noted in Corollary \ref{cor:noGFE}, our assumptions imply that $\Xn =
\type{A}{n} \,(n \ge 2)$,  $\type{B}{n} \,(n \ge 2)$, $\type{C}{n}\, (n \ge 3)$, $\type{D}{n}\, (n \ge 4)$,
$\type{E}{6}$ or $\type{E}{7}$.  In particular, $n \ge 2$.

We write
the distinct elements of $\Pi$ as $\mu_1,\dots,\mu_n$.
By Theorems \ref{thm:KanSko} and  \ref{thm:closeSP}, we know that
\emph{$S$ is the set of nodes of $\Pi$ that are adjacent to $\mu^-$ in $\tPi$, $\chi_S(\mu^+) = 2$
and }
\begin{equation*} \label{eq:Tclose}
T = \set{\mu_t} \text{ for some } 1\le t \le n  \text{ with } \chi_{\mu_t}(\mu^+)=1.
\end{equation*}
We let
\[\Pi' = \Pi \setminus T \andd S' = S \setminus T.\]

In the remainder if this section \emph{we will  for convenience exclude from consideration the case when
$\Xn = \type{A}{n}$ and $\mu_t$ is an interior (not an end) node of $\Pi$}.
We will return
to consider this excluded case in the last subsection.

\begin{lemma}
\label{lem:Piprimeconn} If $\Xn = \type{A}{n}$, suppose that  the element $\mu_t$ of $T$ is an end node of $\Pi$.   Then
 \begin{equation*} \label{eq:Sclose}
S =
\begin{cases}
\hbox{$\set{\mu_s, \mu_t}$ where $\mu_s$ is the other end node in $\Pi$,} & \hbox{if $\Xn = \type{A}{n}$}; \\
\hbox{$\set{\mu_s}$ where $\mu_s\in \Pi$ with $\chi_{\mu_s}(\mu^+) = 2$,}  & \hbox{otherwise}.
\end{cases}
\end{equation*}
Also there exists a unique $\mu_{t'}$
in $\Pi'$ that is adjacent to $\mu_t$ in $\Pi$, so
$\Pi'$ is connected.
\end{lemma}

\begin{proof}  The first statement follows from Remark \ref{rem:mu+}(i),
and the second is easily checked considering types case-by-case. \end{proof}

With the assumptions and notation of Lemma \ref{lem:Piprimeconn}, we define a constant $\theta_{\Pi,T} \in \bbQ$ by
\begin{equation*} \label{eq:theta}
\theta_{\Pi,T} := \left\{
         \begin{array}{ll}
          p+1  & \hbox{if $\Xn = \type{A}{n}\, (n\ge 2)$}\\
          p+2  & \hbox{otherwise},
         \end{array}
       \right.
\end{equation*}
where $p$ is the product of the $(t,t')$-entry of the Cartan matrix  $C(\Pi)$
and the $(t',s)$-entry of the inverse of $C(\Pi')$.  (Note that $C(\Pi)$ is an
$I\times I$ matrix, where $I = \set{1,\dots,n}$, whereas $C(\Pi')$ is an
$I'\times I'$ matrix with $I' = I\setminus{\set{t}}$.  See Subsection~\ref{subsec:rootsystem}.)

\begin{theorem} \label{thm:closecon}  Suppose that $(S,T)\in \SPA(\Pi)$,
$P = \fP(\Pi;S,T)$ is close-to Jordan with nontrivial SP-grading, and, if $\Xn = \type{A}{n}$, the element $\mu_{t}$ of $T$ is an end node of~$\Pi$.
Then
\begin{itemize}
\item[(i)]  $P_0$ is a simple Jordan pair that is isomorphic
to $\fP(\Pi';S')$.
\item[(ii)]  Suppose that $J$ is any simple Jordan pair that is isomorphic to
$\fP(\Pi';S')$, and let
$\tau = \tau_{J,\Pi,T} :=(\tau^-,\tau^+)$, where $\tau^\sg: J^\sg \times J^\msg \to {\bbK}$
is given  by
\begin{equation} \label{eq:tau}
\tau^\sg(x,a) = \frac{\theta_{\Pi,T}}{\dim(J^\sg)} \tr (D_J^\sg(x,a))
\end{equation}
for $x\in J^\sg$, $x\in J^\msg$, $\sg = \pm$. Then  $\tau^\sg(x,a) =
\tau^\msg(a,x)$; $\tau^\sg$ is non-degenerate for $\sg = \pm$;  $\fT(J,\tau)$
is a simple SP-graded Kantor pair;  and
\[\bP \simgr \fT(J,\tau).\]
\end{itemize}
\end{theorem}

\begin{proof} We  first set some notation.  As usual, choose $h_i \in [\cG_{\mu_i},\cG_{-\mu_i}] $
so that $\mu_i(h_i) = 2$ for $1\le i \le n$, in which case $\set{h_i}_{i=1}^n$ is a basis for $\cH$.

Also, choose  nonzero $e^\sg \in \cG_{\mu^\sg}$ for $\sg = \pm$  such that
$[h^+,e^\sg] = \sg 2 e^\sg$,  where $h^+ = [ e^+,e^-]$.  Then since
$\chi_S(\mu^+) = 2$ and $\chi_T(\mu^+) = 1$, we see by \eqref{eq:esg} that
\begin{equation} \label{eq:esgmore}\cG(\chiST)_{\sg 2,*} = \cG(\chiST)_{\sg 2,\sg 1} = \cG_{\mu^\sg} = \bbK e^\sg.\end{equation}

(i): Let
$\Sigma' = \set{\mu\in \Sigma \suchthat \chi_T(\mu) = 0}$.
Then, since  $\Pi'$ is connected, $\Sigma'$ is an irreducible root system of rank $n-1$
(in its real span) with base
$\Pi'$.   Let $\cG'$ be the subalgebra of $\cG$
generated by $\set{\cG_{\mu_i} + \cG_{-\mu_i} \suchthat i \ne t}$, and let
$\cH' = \cH\cap \cG'$.
Then
\[\cG' = \textstyle \cH' \oplus ( \bigoplus_{\mu\in \Sigma'} \cG_\mu\big);\quad \cH' = \sum_{i\ne t} {\bbK} h;\]
$\cG'$ is a simple Lie algebra
with Cartan subalgebra $\cH'$; and $\Sigma(\cG',\cH') = \Sigma'$ (identifying
elements of $\Sigma'$ with their restrictions to $\cH'$).

We now use $\cG'$, $\cH'$, $\Pi'$ and $S'\in \KA(\Pi')$ in Construction \ref{con:Kantor} to obtain
a $5$-graded Lie algebra $\cG'(\chi_{S'}) = \oplus_{i\in \bbZ} \cG'(\chi_{S'})_i$ and hence the
Kantor pair $\fP(\Pi';S')$ enveloped by
$\cG'(\chi_{S'})$ (see Construction \ref{con:Kantor}).  Note that for $i\ne 0$ we have
\[\cG'(\chi_{S'})_i \textstyle = \sum_{\mu\in \Sigma',\, \chi_{S'}(\mu)= i} \cG'_\mu
= \sum_{\mu\in \Sigma,\, \chi_{S}(\mu)= i,\, \chi_T(\mu) = 0 } \cG_\mu =
\cG(\chiST)_{i,0}.\]
Hence
\begin{equation}\label{eq:fPS'} \fP(\Pi';S')^\sg = \cG'(\chi_{S'})_{\sg 1} = \cG(\chiST)_{\sg 1,0}  = P_0^\sg,\end{equation}
so
$P_0 = \fP(\Pi';S')$.  Moreover, $\cG'(\chi_{S'})_{\sg 2} = \cG(\chiST)_{\sg 2,0}  = 0$ by \eqref{eq:esgmore}, so
$\cG'(\chi_{S'})$ is $3$-graded and thus  $P_0$ is Jordan by Remark  \ref{rem:fPS}(ii).

(ii):  To prove (ii), we can assume that $J = P_0$ and use the notation and conclusions in the proof of (i).

Let $\omega := \exp(\ad e^+) \exp(-\ad e^-) \exp(\ad e^+) \in \Aut(\cG)$. We will use Lemma \ref{lem:omega},
which gives us detailed information about $\omega$.

First, by Lemma \ref{lem:omega}(v),
$\omega(x) = - \sg [e^\msg, x]$ for $x \in P^\sg$.  Thus, since
$e^\msg \in \cG(\chiST)_{\msg 2, \msg 1}$ by \eqref{eq:esgmore}, we see that
$\omega(\cG(\chiST)_{\sg 1, \sg i}) = \cG(\chiST)_{\msg 1, \msg(1-i)}$ for $i=0,1$. That is
\begin{equation*} \label{eq:rp8}
\omega \text{ exchanges } P_i^\sg \text { and } P^\msg_{1-i} .
\end{equation*}

Next let
$\tau^\sg = \zeta^\sg|_{J^\sg \times J^\msg} : J^\sg \times J^\msg \to \bbK$, with
$\zeta^\sg$ as defined in Lemma \ref{lem:omega}(vi),  so
\begin{equation} \label{eq:deftau} [[x,a],e^\sg] = \tau^\sg(x,a) e^\sg \end{equation}
for $x\in J^\sg$, $a\in J^\msg$.
At this point we will take this as the definition
of $\tau^\sg$, and then at the end of the proof we will prove \eqref{eq:tau}.

Note that $\tau^\sg(x,a) = \tau^\msg(a,x)$  by~\eqref{eq:zetasym}.
Also, for $\sg = \pm$, $i=0,1$, we have
\[[[P_i^\sg,P^\msg_{1-i}],e^\sg] \subseteq
[[\cG(\chiST)_{\sg 1,\sg i},\cG(\chiST)_{\msg 1,\msg(1-i)} ],\cG(\chiST)_{\sg 2,\sg 1} ], \]
which is contained in  $\cG(\chiST)_{\sg 2, \sg 2i} = 0$.
Thus, by \eqref{eq:defzeta1}, $\zeta^\sg(P_i^\sg,P^\msg_{1-i}) = 0$.  So the non-degeneracy of $\zeta^\sg$ implies that of $\tau^\sg$.

In the rest of the proof, it is more convenient to work with $Q := \bP^\op$, rather than $\bP$ itself.  We next show that the $\bbZ$-graded trilinear pairs $Q$ and $\fT(J,\tau)$ are  graded-isomorphic.
This will show that $\fT(J,\tau)$ is a simple SP-graded Kantor pair and, by
Proposition \ref{prop:Pop=Pgr}, that
$\bP \simgr \fT(J,\tau)$ (leaving only  \eqref{eq:tau} to prove).

Now $Q= Q^0 \oplus Q^1$, where   $Q_0^\sg = \bP_0^\msg = P_0^\sg$  and
$Q_1^\sg = \bP_1^\msg = P_1^\msg = \omega(P_0^\sg)$, so
\[Q_i^\sg = \omega^i(J^\sg)\]
for $\sg = \pm$, $i = 0,1$.
Moreover, by the definition of Weyl images (see Section \ref{subsec:SPWeyl}), the products in $Q$ are
given by $\{x,a,y\}^\sg = [[x,a],y]$ in $\cG$. On the other hand,
$\fT(J,\tau)= \fT(J,\tau)_0 \oplus \fT(J,\tau)_1$ with
$\fT(J,\tau)_i = J^\sg \otimes u_i^\sg$
and products given by \eqref{eq:FJtprod}.
With  this in mind, we define $\ph = (\ph^-,\ph^+)$, where
$\ph^\sg:  \fT(J,\tau)^\sg \to Q^\sg $ is the linear isomorphism such that
\[\ph^\sg(x \otimes u^\sg_i) = \omega^i x\]
for $x \in J^\sg$,  $\sg = \pm$ and $i = 0,1$.
In order to prove that $\ph$ is an isomorphism of trilinear pairs, we must show that
\[ \ph^\sg \big(\{x \otimes u^\sg_i, a \otimes u^\msg_j ,  y \otimes u^\sg_k  \} \big) =
[[ \omega^i x,\omega^j a],\omega^k y]\]
for $\sg = \pm$,  $x,y\in J^\sg$, $a\in J^\msg$ and $i,j,k = 0,1$.
But $\formtwo(u_i^\sg,u_j^\msg)u_k^\sg = \delta_{ij}u_k^\sg $ and
$(u_i^\sg,u_j^\msg,u_k^\sg) = \delta_{ij}u_k^\sg - \delta_{jk}u_i^\sg$.  So we must prove
that
\begin{equation}
\label{eq:ccpf5}
\delta_{ij}\omega^k([[x,a],y]) -\delta_{ij}\tau^\sg(x,a)\omega^k y
+ \delta_{jk}\tau^\sg(x,a)\omega^i y =
[[ \omega^i x,\omega^j a],\omega^k y].
\end{equation}
Now if $(i,j,k) = (0,0,0)$, \eqref{eq:ccpf5} is trivial; whereas if $(i,j,k) = (1,1,1)$,
\eqref{eq:ccpf5} holds  since $\omega$ is an automorphism.
If  $(i,j,k) = (1,0,0)$ or  $(0,0,1)$, then \eqref{eq:ccpf5} follows from Lemma \ref{lem:omega}(vii);
whereas if $(i,j,k) = (0,1,0)$, \eqref{eq:ccpf5} holds since its right hand side lies in
$\cG(\chiST){\sg 3, *}$, which is $0$.
Finally the cases $(1,1,0)$, $(1,0,1)$ and $(0,1,1)$ follow by applying
$\omega$ to the cases $(0,0,1)$, $(0,1,0)$ and $(1,0,0)$  respectively.

\emph{It remains  to prove \eqref{eq:tau}}.   For this let
$\set{\ell'_i}_{i\ne t}$ be the basis for $\cH'$ that is dual to $\set{\mu_i}_{i\ne t}$.

Recall   from the proof of $(i)$ that $\cG'(\chi_{S'}) = \oplus_{i\in \bbZ} \cG'(\chi_{S'})_i$ is 3-graded and that
$\cG'(\chi_{S'})_{\sg 1} = P_0^\sg = J^\sg$.
To simplify notation, we set $\ff :=  \cG'(\chi_{S'})_0$.  Then
\begin{equation} \label{eq:defff}
\ff =   \textstyle  \cH' \oplus \sum_{\mu\in \Sigma',\, \chi_{S'}(\mu)= 0} \cG'_\mu
= \cH' \oplus \sum_{\mu\in \Sigma,\, \chi_{S}(\mu)= 0,\, \chi_T(\mu) = 0 } \cG_\mu = \bbK h_s \oplus \fk
\end{equation}
as vector spaces, where $\fk$ is the  subalgebra  of $\cG$ generated by
$\set{\cG_{\mu_i} +  \cG_{-\mu_i} \suchthat i \ne s,t}$.  Further, $\fk$ is semi-simple and
$\ell'_{s}$ is in the centre of $\ff$.  So
$\ell_s'\notin \fk$ and hence
\begin{equation} \label{eq:ccpf6}
\ff = {\bbK} \ell'_s \oplus \fk.
\end{equation}
as algebras.
We next construct some elements of the one-dimensional space
\[\cF := \set{\lm \in \Hom(\ff,{\bbK}) \suchthat \lm(\fk) = 0}.\]

First note that   $\ff \subseteq \cG(\chi_S)_{0,0}$ by \eqref{eq:defff}.  So
by \eqref{eq:esgmore},   $[\ff,e^\sg] \subseteq {\bbK} e^\sg$
for $\sg = \pm$. Thus for $\sg = \pm$, there exists a unique   $\nu^\sg\in \Hom(\ff,{\bbK})$ such that
\begin{equation*} \label{eq:F1}
[z,e^\sg] = \nu^\sg(z) e^\sg
\end{equation*}
for $z\in \ff$.  Then since $\fk = [\fk,\fk]$, we have $\nu^\sg \in \cF$.  Also
$\nu^\sg(\ell_s') = \mu^\sg(\ell_s') = \sg \mu^+(\ell_s')$.

Next,  since $\cG'(\chi_{S'})$ is $3$-graded, we have
$[\ff,J^\sg] \subseteq J^\sg$ for $\sg = \pm$.   So
we can define $\xi^\sg \in \Hom(\ff,{\bbK})$ by
\begin{equation*}\label{eq:F2}
\xi^\sg (z) = \tr(\ad(z)\mid_{J^\sg})
\end{equation*}
for $z\in \ff$. Once again we see that $\xi^\sg\in \cF$.   Moreover, since
$J^\sg = \cG'(\chi_{S'})_{\sg 1}$ by \eqref{eq:fPS'}, we  have
$\ad(\ell_s')\mid_{J^\sg} = \sg \id_{J^\sg}$, so
$\xi^\sg(\ell_s') = \sg \dim(J^\sg) \ne 0$.

Now, since $\cF$ is one dimensional, we have
$\nu^\sg = r^\sg \xi^\sg$ for some  $r^\sg\in {\bbK}$, in which case
(evaluating at $\ell_s'$)  we have $r^\sg = \frac{\mu^+(\ell_s')}{\dim(J^\sg)}$ for $\sg = \pm$.
So for $z\in \ff$ we have
\begin{equation*}
\label{eq:F3}
\nu^\sg(z) = \frac{\mu^+(\ell_s')}{\dim(J^\sg)} \tr(\ad(z)\mid_{J^\sg}).
\end{equation*}

But if $x\in J^\sg$, $a\in J^\msg$, we have
$[x,a]\in \ff$ since $\cG'(\chi_{S'})$ is $3$-graded.
Also  $\tau^\sg(x,a) =  \nu^\sg([x,a])$ by \eqref{eq:deftau}, so
\[\tau^\sg(x,a) = \frac{\mu^+(\ell_s')}{\dim(J^\sg)} \tr(D_J^\sg(x,a)).\]

Finally, we compute $\mu^+(\ell_s')$.  Let
$C = C(\Pi)$ with
$(i,j)$-entry $c_{ij}  := \langle \mu_i,\mu_j \rangle$,
and let
$\set{\ell_i}$ be the basis for $\cH$ that is dual to $\set{\mu_i}$.  Then
\[h_j \textstyle = \sum_i c_{ij} \ell_i\]
for $1\le j \le n$.  Similarly, since  $C(\Pi')$ has $(i,j)$-entry $c_{ij}$ for $i,j\ne t$, we have
$h_j = \sum_{i\ne t} c_{ij} \ell'_i$  for $j \ne t$.  So letting $D' = {C(\Pi')}^{-1}$ with
$(i,j)$- entry $d'_{ij}$ for $i,j\ne t$, we have
$\ell'_j =  \textstyle \sum_{i\ne t} d'_{ij} h_i$
for $j\ne t$.  Therefore
\begin{align*} \ell_s' &= \textstyle \sum_{i\ne t} d'_{is} h_i
= \sum_{i\ne t} d'_{is} \sum_k c_{ki} \ell_k =  \sum_k \left(\sum_{i\ne t} c_{ki}  d'_{is}\right) \ell_k\\
&= \textstyle \left(\sum_{i\ne t} c_{ti}  d'_{is}\right) \ell_t + \sum_{k\ne t} \left(\sum_{i\ne t} c_{ki}  d'_{is}\right) \ell_k
= c_{tt'}  d'_{t's}\ell_t + \ell_s,
\end{align*}
since $c_{ti} = 0$ for $i\ne t'$.  Thus
$\mu^+(\ell_s') =  c_{tt'}  d'_{t's} \mu^+(\ell_t)
 + \mu^+(\ell_s) = \theta_{\Pi,T}$. \end{proof}

Theorem  \ref{thm:closecon} shows that the reflection $\bP$ of any Kantor pair $P$ satisfying the given assumptions  can be constructed
in the form $\fT(J,\tau)$, where  $J$ and $\tau$ are  described in terms of marked Dynkin diagrams.  In the next corollary, we describe $J$ and $\tau$ using classical matrix constructions.

Our notation in Columns 3 and 4 of Table  \ref{tab:TJtau}  follows \cite[\S 17.4]{L}.  Indeed,
$\text{I}_{p,q}$ is the Jordan pair
$(M_{p,q}(\bbK),M_{p,q}(\bbK))$ with products $\{x,a,y\}^\sg = xa^t y + ya^t x$;
$\text{II}_n$ is the subpair $(\text{A}_n(\bbK), \text{A}_n(\bbK))$
of $\text{I}_{n,n}$, where $\text{A}_n(\bbK)$ is the space of alternating $n\times n$-matrices;
$\text{IV}_n$ is the Jordan pair $(\bbK^n,\bbK^n)$ with products $\{x,a,y\}^\sg = q(x,a)y + q(y,a)x - q(x,y)a$,
where  $q : \bbK^n\times \bbK^n \to \bbK$ is the non-degenerate symmetric bilinear form on $\bbK^n$; and $\text{V}$ is the Jordan pair
$(\Mat_{1,2}(\cC), \Mat_{1, 2}(\cC))$ with products $\{x,a,y\}^\sg = x(\bar a^t  y) + y (\bar a^t x)$, where $\cC$ is the (split) Cayley algebra with standard involution
$c\mapsto \bar c$ and trace form  $\trcomp_\cC : \cC \to \bbK$ given by $\trcomp_\cC(c) = c + \bar c$  \cite[\S 12.10]{L}.

\begin{corollary} \label{cor:tableBP}  Suppose we have the assumptions and notation
of Theorem \ref{thm:closecon}.  The first column of Table \ref{tab:TJtau} lists the possibilities, up to diagram automorphism, for the marked Dynkin diagrams representing $P = \fP(\Pi;S,T)$ (with the restriction on the rank $n$ of $\Pi$ also indicated).
The second column lists the corresponding marked Dynkin diagram representing $\bP \simgr \fP(\Pi;\sg_{\Pi'}(S'),T)$ (see Corollary \ref{cor:refclose}).  Finally, in each row, we have
\[\bP  \simgr \fT(J,\tau),  \]
where $J$ and $\tau$ are listed in Columns 3 and 4 of the table respectively.
\end{corollary}

\begin{table}
\renewcommand{\arraystretch}{1.3}
\begin{tabular}[ht]
{|c   c | c| c | c | c | c |}

\hline
\small
$P= \fP(\Pi;S,T)$
&$n$
& $\bP \simgr \fP(\Pi;\sigma_{\Pi'}(S'),T)$
& $J$
& $\tau^\sigma(x,a)$
 \\ 
\whline

\clearlabels \Satrue \Sdtrue \Tatrue\An
&$\ge 2$
& \clearlabels \Sbtrue \Tatrue\An
& $\text{I}_{1,n-1}$
& $ x a^t$
\\
\hline


\clearlabels \Sbtrue  \Tatrue\Bn
& $\ge 3$
& \clearlabels \Sbtrue  \Tatrue\Bn
& $\text{IV}_{2n-3}$
& $ q(x,a)$
\\
\hline

\clearlabels \Satrue  \Tetrue\Cn
& $\ge 2$
& \clearlabels \Sdtrue  \Tetrue\Cn
& $\text{I}_{1,n-1}$
& $ 2 x a^t$
\\
\hline

\clearlabels \Sbtrue  \Tatrue\Dn
& $\ge 4$
& \clearlabels \Sbtrue  \Tatrue\Dn
& $\text{IV}_{2n-4}$
& $ q(x,a)$
\\
\hline

\clearlabels \Sbtrue  \Tetrue\Dn
&$\ge 5$
& \clearlabels \Sdtrue  \Tetrue\Dn
& $\text{I}_{2,n-2}$
& $ \tr(x a^t)$
\\
\hline

\clearlabels \Sftrue  \Tatrue\Esixtable
&$6$
& \clearlabels \Sbtrue  \Tatrue\Esixtable
& $\text{II}_5$
& $ \frac 1 2 \tr(x a)$
\\
\hline

\clearlabels \Satrue  \Tftrue \Eseven
&$7$
& \clearlabels \Setrue \Tftrue \Eseven
& $\text{V}$
& $ \text{t}_{\cC}(x \bar a^t)$
\\
\hline
\end{tabular}
\medskip
\caption{Reflections of simple close-to-Jordan pairs (with a case excluded)}
\label{tab:TJtau}
\end{table}

\begin{proof} Constructing Column 1 is an easy exercise considering types case-by-case using the discussion at the beginning of this  subsection.

To  complete row $i$, where $1\le i \le 7$, we chose $S$ and $\Pi$ as listed in Column 1.
To get the entry in Column 2,  we just calculate $\sg_{\Pi'}(S')$. To get the entries in
Columns 3 and 4, we do the following:
\begin{itemize}
\item[(a)]  Find a Jordan pair $J$ of matrices that is isomorphic to
$\fP(\Pi';S')$; and then
\item[(b)] Calculate $\tr(D_J^\sg(x,a))$ for
$x\in J^\sg$, $a\in J^\msg$ (which allows us to calculate $\tau^\sg$ using \eqref{eq:tau}).
\end{itemize}
Fortunately, (a) and (b) can be accomplished using well known
facts from the theory of Jordan pairs.  Indeed for (a),  Loos and Neher have written down a table
which lists the finite dimensional simple Jordan pairs (as pairs of matrices) together with their representing marked Dynkin diagrams.  (See \cite[\S 19.9]{LN}, which refers to \cite{N1} and \cite{N2} for the necessary arguments.)
For (b), Meyberg observes in \cite[(2.20)]{M2} (using Theorem 17.3 in \cite{L}) that
\begin{equation} \label{eq:Mey}
\tr(D_J^+(x,a)) = g_J \, m_J(x,a)
\end{equation}
for $x\in J^+$, $a\in J^-$, where $g_J$ is the genus of $J$ and $m_J : J^+ \times J^- \to \bbK$ is the generic trace of $J$.

We carry out  steps
(a) and (b) in detail for the second last row of the table, leaving the other rows to the reader.  In that row the marked diagram representing $\fP(\Pi';S')$ is isomorphic to
$\clearlabels \Sdtrue  \Dfive$.  \hbox{ } So from the table in \cite[\S 19.9]{LN}, we see that we may take
$J = \text{II}_5 = (\text{A}_5(\bbK),\text{A}_5(\bbK))$.
Then from  \cite[\S 17.2]{L}  we know that
$g_J = 8$ and $m_J(x,a) = \frac 12  tr(xa)$, which using \eqref{eq:Mey} gives us the equality $\tr (D^+(x,a))  = 4\tr(xa)$.  So since $J = J^\op$, we have
$\tr (D^\sg(x,a))  = 4\tr(xa)$ for $\sg = \pm$.  Thus, since
$\dim(J^\sg)= 10$, we have by  \eqref{eq:tau} that
$\tau^\sg(x,a) =  \frac 25\theta_{\Pi,T} \tr (xa)$.
Now, labeling the roots in $\Pi$ as in  Example \ref{ex:E6KP}, we have
$t = 1$, $t' = 2$ and $s = 6$.  Also, the
$(1,2)$-entry of $C(\Pi)$ is $-1$, and
the  $(2,6)$-entry of $C(\Pi')^{-1}$ is
$\frac 3 4$ \cite[Table 1,\S13.2]{H}, so $\theta_{\Pi,T} =-\frac 34 +2  = \frac 54$. Hence $\tau^\sg(x,a)  = \frac 1{2} \tr(xa)$.
\end{proof}

Note that if we ignore the grading, the ungraded Kantor pair $\bP$ in the second last row of Table
\ref{tab:TJtau} is the pair labelled $\type{E}{6}(20,5)$ in Example \ref{ex:E6KP}.  Our construction  of $\type{E}{6}(20,5)$
in the form $\fT(P,\tau)$
is a basis-free pair version, with full proofs, of the construction given by
Kantor in \cite[\S 6.6]{K1}  (see also \cite[\S 4]{K2}) of the Kantor triple system $C_{55}^2$.  More precisely,
$\fT(J,\tau)$ shown in the table is the double of  $C_{55}^2$.
The pair $\type{E}{6}(20,5)$   is of particular interest
since it is one of only two finite dimensional simple Kantor pairs of exceptional type that does not arise by doubling
a structurable algebra.  (See  \cite[\S 7.9]{AFS}, where another construction
of  $\type{E}{6}(20,5)$ is given as the reflection of an SP-graded Kantor pair that is constructed using exterior algebras.)

\subsection{The excluded case}  \label{subsec:excluded}  In this final subsection, we consider, without proofs, the case that was excluded in Theorem \ref{thm:closecon} and Corollary \ref{cor:tableBP}.   We make the  assumptions and use the notation of the first three paragraphs of Subsection \ref{subsec:closeconstruct}.

To treat the excluded case, we assume that
$\Pi=\set{\mu_1,\dots,\mu_n}$ is of type $\type{A}{n}$ with roots labelled in order on the diagram from left to right, $n\ge 3$,
$S = \set{\mu_1,\mu_n}$ and $T = \set{\mu_t}$ with $1 < t < n$.
Let $P = \fP(\Pi;S,T)$, in which case the marked Dynkin diagrams representing $P$ and $\bP \simgr \fP(\Pi;\sigma_{T'}(S'),T)$ are respectively:
\begin{equation*} \clearlabels \Satrue \Setrue \Tctrue \Anmidout \qquad \andd \qquad \clearlabels \Sbtrue \Sdtrue \Tctrue \Anmidin .\end{equation*}

Note that $\Pi'$ is not connected (which is the reason we are treating this case separately).  In fact, the connected components of $\Pi'$ are
$\Pi_1' = \set{\mu_1,\dots,\mu_{t-1}}$  and $\Pi_2' = \set{\mu_{t+1},\dots,\mu_{n}}$.  One sees as in Theorem \ref{thm:closecon}(i) that
$\fP(\Pi_1';\set{\mu_1})$ and
$\fP(\Pi_2';\set{\mu_n})$ are Jordan pairs with
\[P_0 = \fP(\Pi_1';\set{\mu_1}) \oplus  \fP(\Pi_2';\set{\mu_n}).\]

Next let $J = J_1 \oplus J_2$, where $J_1 \simeq \fP(\Pi_1';\set{\mu_1})$ and $J_2 \simeq  \fP(\Pi_2';\set{\mu_n})$ are simple Jordan pairs.  We obtain as in
Theorem \ref{thm:closecon}(ii) that $\bP \simgr \fT(J,\tau)$, where
\begin{equation}\label{eq:excludedtau} \textstyle \tau^\sg(x_1+x_2,a_1+a_2) = \frac 1 t\tr(D_{J_1}^\sg(x_1,a_1)) +
\frac 1 {n-t+1} \tr(D_{J_2}^\sg(x_2,a_2)) \end{equation}
for $x_i\in J_i^\sg$, $a_i\in J_i^\msg$.

Finally, we may choose
\begin{equation} \label{eq:Jexcluded}
J = J_1 \oplus J_2,\quad \text{ where } J_1 = \text{I}_{1,t-1} \text{ and } J_2 = \text{I}_{1,n-t}\ ;
\end{equation}
 and
we see as in Corollary \ref{cor:tableBP} that
\begin{equation*} \label{eq:excludemat} \bP \simgr \fT(J,\tau) \quad \text{with}\quad \tau^\sg(x_1+x_2,a_1+a_2) = x_1a_1^t + x_2 a_2^t \end{equation*}
(In this last equation, the superscript $t$ denotes the transpose map.)

\begin{remark} \label{rem:proofexcluded} The proof of the above facts is obtained by modifying the arguments
in Subsection \ref{subsec:closeconstruct}.  The main difference is in the proof of
\eqref{eq:excludedtau}, where we consider two different spaces of homomorphisms $\cF_1$ and
$\cF_2$ obtained from  simple $3$-graded Lie algebras which envelop
the Jordan pairs $\fP(\Pi_1';\set{\mu_1})$ and
$\fP(\Pi_2';\set{\mu_n})$ respectively (just as $\cF$ is obtained from
the simple $3$-graded Lie algebra $\cG'(\chi_{S'})$ in the proof of \eqref{eq:tau}).  We leave the details to the reader.
\end{remark}

\begin{remark}  \label{rem:moreTJtau} The reader may have noticed that the  Jordan
pairs $J$ that appear  in either Column 3 of Table
\ref{tab:TJtau} or in \eqref{eq:Jexcluded} consist of all finite dimensional semi-simple Jordan pairs of
rank $1$ or $2$ \cite[\S 17]{L}; and that $\tau$ is almost uniquely determined by~$J$.  The intrinsic reason for these mysterious facts will be explained in \cite{AF2}, where we will use Jordan techniques to study the construction $\fT(J,\tau)$ over an arbitrary field of
characteristic $\ne 2$ or  $3$.
\end{remark}

\end{document}